\documentclass[12pt,a4paper]{amsart}
\usepackage{amsmath,amsthm,amssymb,latexsym,a4wide,bbm,framed,mathdots}
\usepackage{pict2e,curve2e}
\newcommand{\drawline}{\Line}
\newcommand{\dashline}{\Dashline}
\newcommand{\dottedline}{\Dotline}

\newtheorem{theorem}{Theorem}
\newtheorem{lemma}[theorem]{Lemma}
\newtheorem{corollary}[theorem]{Corollary}
\newtheorem{proposition}[theorem]{Proposition}

\newtheorem{problem}[theorem]{Problem}
\newtheorem{remark}[theorem]{Remark}
\usepackage[all]{xy}
\usepackage[active]{srcltx}
\usepackage[parfill]{parskip}
\usepackage{enumerate}
\numberwithin{equation}{section}

\font\sc=rsfs10

\newcommand{\cP}{\sc\mbox{P}\hspace{1.0pt}}

\begin{document}

\title[Principal ideals in $d$-tonal partition monoid]
{On the number of principal ideals\\ 
in $d$-tonal partition monoids}
\author{Chwas Ahmed,  Paul Martin and Volodymyr Mazorchuk}
%\date{\today}

\begin{abstract}
For a positive integer $d$, a non-negative integer $n$  and a non-negative
integer $h\leq n$, we study 
the number $C_{n}^{(d)}$ of principal ideals; and 
the number $C_{n,h}^{(d)}$ of principal ideals generated by an element of rank $h$, 
in the $d$-tonal partition monoid on $n$ elements. 
We compute closed forms for the first family,
as partial cumulative sums of known sequences.
The second gives an infinite family of new integral sequences. 
We discuss 
their connections to certain integral lattices as well as to
combinatorics of partitions. 
\end{abstract}
\maketitle

\section{Introduction and description of the results}\label{s1}

Enumeration is often the starting point in understanding a given mathematical structure.
Twisted monoid algebras \cite{GMS,RZ} 
of $d$-tonal partition monoids appear in \cite{Ta} as right Schur-Weyl
duals for generalized symmetric groups. These algebras are subalgebras of the classical partition
algebras from \cite{Mar0,Mar} and \cite{Jo}. 
The monoids underlying the 
latter algebras have relatively simple principal ideal structure 
and well studied representation theory, see \cite{Mar,Mar2}. 
The $d$-tonal subalgebras of partition algebras are  
more complicated. Some basics representation theory of these and related algebras was
developed in \cite{Ko1,Ko2,Ko3} and \cite{Or}. However 
in the monoid case, for example, these studies cover only a trivial quotient.

The motivation for the present paper comes from our attempt to
understand the structure of $d$-tonal partition
algebras using combinatorics of Green's relations 
for the finite  $d$-tonal partition monoid. 
The main question we 
answer 
in the present paper is what is the number of different principal
$2$-sided ideals in such a monoid. 
This already depends on two parameters: the difference parameter $d$
and the parameter $n$ which controls 
the size of our partitions. We denote the number of such ideals by $C_n^{(d)}$. Algebraically, there is
a natural third parameter which enters the picture: the rank $h\in\{0,1,\dots,n\}$ of the generating partition. 
Using this parameter we write
\begin{displaymath}
C_n^{(d)}=C_{n,0}^{(d)}+C_{n,1}^{(d)}+C_{n,2}^{(d)}+\dots+C_{n,n}^{(d)},
\end{displaymath}
where $C_{n,h}^{(d)}$ denotes the number
of ideals generated by an element of rank $h$.
We seek a closed formula for both $C_{n,h}^{(d)}$ and for $C_n^{(d)}$.
Cases $d=1$  %is trivial and the case
and $d=2$ turn out to be easy.

In Section~\ref{s2} we give an alternative, purely combinatorial, definition for 
the numbers $C_n^{(d)}$ as enumerators  of layers in certain graded posets. 
These are related to the original motivation in Section~\ref{s5}.
The main part of the paper is devoted to the study of the case $d=3$ which occupies
Section~\ref{s3}. 

Extra motivation for the case $d=3$ comes from its intrinsic 
geometric-physical interest.
We give an explicit formula for $C_{n,h}^{(3)}$ in case
$h$ is relatively big (i.e. $h\geq \lfloor\frac{n}{2}\rfloor$), see Proposition~\ref{prop5}, 
and in case $h$ is relatively small (i.e. $h\leq \lceil\frac{n}{3}\rceil$), see Proposition~\ref{prop7}.
The former gives a connection of our sequence to partitions with at most three parts while the
latter shows a connection to triangular numbers (in fact, to a special counting of triangular numbers
modulo $3$). 
Our first main result is that the sequence  $C_n^{(3)}$ is given by 
the ``Cyvin sequence'' 
($A028289$ in \cite{OEIS}) 
which enumerates the number of isomorphism classes
of hollow hexagons (representing polycyclic hydrocarbons), see \cite{CBC,Pol}. 
In Theorem~\ref{thmn72} of Section~\ref{s55} we even give an explicit bijection 
between hollow hexagons and the graded poset underlying the definition
of $C_n^{(3)}$ 
given in Section~\ref{s2}. 

In Section~\ref{s4} we relate our graded posets to combinatorics of partitions and in Section~\ref{s5} we  makes precise the connection between the combinatorially defined 
data discussed in the paper and the algebraic structures which motivate our investigation.
Combinatorics which underlines the algebraic structure allows us to
determine $C_n^{(d)}$ for all $d$ and $n$ in terms of partitions 
with at most $d$ parts, see Theorem~\ref{thmdpart} in Section~\ref{s6}. 
As a corollary of this uniform description for all $d$, we obtain an alternative,
simpler, description of $A028289$ using partitions with at most $3$ parts. 

\noindent
{\bf Acknowledgements.} An essential part of the research was done during the 
visit of the third author to Leeds in October 2014, which was
supported by EPSRC under grant EP/I038683/1.   
The paper was completed during the stay of the second and the third
authors at the Institute Mittag-Leffler in the spring of 2015. 
The financial support and hospitality of both the University of Leeds and
the Institute Mittag-Leffler are gratefully acknowledged.
The first author is supported by the KRG via the HCDP Scholarship program.
For the third author the research was partially supported by 
the Royal Swedish Academy of Sciences, Knut and Alice Wallenbergs Stiftelse 
and the Swedish Research Council. We thank M.~Kosuda for helpful discussions.
We thank the referee for helpful comments and for pointing out several
inaccuracies in the original version of the paper.

\section{Graded posets}\label{s2}

\subsection{Notation and general construction}\label{s2.1}

We denote by $\mathbb{R}$ the set of all real numbers,
by $\mathbb{R}_{\geq 0}$ the set of all non-negative  real numbers,
by $\mathbb{Z}$ the set of all integers, by $\mathbb{N}$ 
the set of all positive integers and by $\mathbb{Z}_{\geq 0}$ 
the set of all non-negative integers.

Consider the set $\mathbb{Z}^d$ for some fixed $d\in\mathbb{N}$. 
Elements of $\mathbb{Z}^d$ are vectors $\mathbf{v}=(v_1,v_2,\dots,v_d)$ 
such that $v_i\in \mathbb{Z}$ for all $i=1,2,\dots,d$. The number $v_1+v_2+\dots+v_d\in \mathbb{Z}$ 
is called the {\em height} of $\mathbf{v}$ and denoted $\mathrm{ht}(\mathbf{v})$. The set
$\mathbb{Z}^d$ has the natural structure of an abelian group given by addition. The map
$\mathrm{ht}:\mathbb{Z}^d\to \mathbb{Z}$ is a surjective group homomorphism. For $i=1,2,\dots,d$, 
we denote by $\mathbf{e}(i)$ the standard basis vector
$(0,0,\dots,0,1,0,0,\dots,0)$ in $\mathbb{Z}^d$,  
in which the only non-zero element $1$ stands in position $i$. Note that each $\mathbf{e}(i)$ has height $1$.

Denote by $\Lambda_d$ the subset $\mathbb{Z}_{\geq 0}^d$ in
$\mathbb{Z}^d$. For $h\in \mathbb{Z}_{\geq 0}$,  
we denote by $\Lambda_d^{(h)}$ the set of all elements in $\Lambda_d$
of height $h$ and note that the set  
$\Lambda_d^{(h)}$ is finite, in fact, a standard combinatorial exercise shows that
\begin{equation}\label{eq5}
|\Lambda_d^{(h)}|=\binom{h+d-1}{d-1}. 
\end{equation}
Define
%\begin{displaymath}
$
%%X\subset
\mathbb{Z}^d_{(h)} = 
 \{\mathbf{v}\in \mathbb{Z}^d\,:\, \mathrm{ht}(\mathbf{v})= h \} 
$.
%\end{displaymath}
For a fixed subset
\begin{displaymath}
%$
X \subset \mathbb{Z}^d_{(-1)}
%$ 
\end{displaymath}
define on $\Lambda_d$ the structure of a {\em poset} using the
transitive closure $<_X$ 
of the following 
%{\em covering relation}:
manifestly antisymmetric relation:
\begin{equation}\label{eq1}
\mathbf{v}\lessdot_X\mathbf{w}\quad\text{ if and only if }\quad \mathbf{v}-\mathbf{w}\in X.
\end{equation}
Note from the construction that $\lessdot_X$ is a covering relation. 
Directly from the definitions we have that
$\mathbf{v}\lessdot_X\mathbf{w}$ implies  
$\mathrm{ht}(\mathbf{v})=\mathrm{ht}(\mathbf{w})-1$, for all
$\mathbf{v}$ and $\mathbf{w}$. 
In particular, the poset $(\Lambda_d,<_X)$ is a graded poset with 
{\em  rank function}  
$\mathrm{ht}:\Lambda_d\to \mathbb{Z}$. Note that $X\neq X'$ implies
$\lessdot_X\neq \lessdot_{X'}$.

\subsection{The poset $\cP_{d}$}\label{s2.2}

Consider the set
\begin{displaymath}
X_d:=\{\mathbf{e}(k)-\mathbf{e}(i)-\mathbf{e}(j)\in\mathbb{Z}^d\,:\,
i,j,k\in\{1,2,\dots,d\}\text{ such that } k\equiv i+j\text{ mod } d
\}.
\end{displaymath}
For example, 
\begin{gather*}
X_1=\{(-1)\};\quad\quad   X_2=\{(-2,1),(0,-1)\};\\
X_3=\{(1,-2,0),(-2,1,0),(-1,-1,1),(0,0,-1)\};\\
X_4=\{(0,0,0,-1),(-1,-1,1,0),(-1,0,-1,1),(1,-1,-1,0),\\(-2,1,0,0),(0,-2,0,1),(0,1,-2,0)\}.
\end{gather*}
Note that $<_{X_d}$ is defined. 
Denote by $\cP_{d}$ the poset $(\Lambda_d,<_{X_d})$. 
Finite principal ideals of $\cP_{d}$ are the main objects of interest 
in this paper. For simplicity, we will denote the relation $<_{X_d}$ by $\prec$.
For the record, we note the following.

\begin{lemma}\label{lemnnn}
We have $\lvert X_d \rvert = \frac{d(d-1)}{2} +1$.
\end{lemma}

\begin{proof}
The pair $\{i,j\}$ from the definition of $X_d$ can be chosen in
$\binom{d}{2}$ different ways for $i\neq j$ and in $d$ different ways for $i=j$.
After choosing $\{i,j\}$, the element $k$ is uniquely defined.
Note that the $d$ choices when $\{i,j\}\cap\{d\}\neq \varnothing$ result in the
same vector $-\mathbf{e}(d)$. The claim follows.  
\end{proof}

For $\mathbf{v}\in \cP_{d}$, we denote by $I(\mathbf{v})$ the principal ideal of
$\cP_{d}$ generated by $\mathbf{v}$, that is
\begin{displaymath}
I(\mathbf{v}):=\{\mathbf{v}\}\bigcup\{\mathbf{w}\in
              \cP_{d}\,:\,\mathbf{w}\prec\mathbf{v}\}. 
\end{displaymath}

For $n\in\mathbb{Z}_{\geq 0}$, we set $C_{n}^{(d)}:=|I(n\mathbf{e}(1))|$.
For $h=0,1,2,\dots,n$, we also define 
\begin{displaymath}
C_{n,h}^{(d)}:=|I(n\mathbf{e}(1))\cap \Lambda_{d}^{(h)}|. 
\end{displaymath}
Then we have $C_{n}^{(d)}=C_{n,0}^{(d)}+C_{n,1}^{(d)}+\dots+C_{n,n}^{(d)}$. 
Our interest in $I(n\mathbf{e}(1))$ will be explained in Section~\ref{s5} (see Theorem~\ref{thm34}).

We observe the following structural property of $\cP_{d}$: for $k=1,2,\dots,d$, 
consider the set $\Lambda_{d,k}$ which consists of all $\mathbf{v}\in \Lambda_{d}$
such that $v_1+2v_2+3v_3+\dots+dv_d\equiv k\text{ mod }d$. Note that $\Lambda_{d,k}\cap\Lambda_{d,k'}=
\varnothing$ if $k\neq k'$. Denote by $\cP_{d,k}$ the poset with
the underlying set $\Lambda_{d,k}$ obtained by restricting the relation $<_{X_d}$
to $\Lambda_{d,k}$. For $h\in\mathbb{Z}_{\geq 0}$, set  $\Lambda_{d,k}^{(h)}:=
\Lambda_{d,k}\cap \Lambda_{d}^{(h)}$.

\begin{proposition}\label{prop1}
{\hspace{2mm}}

\begin{enumerate}[$($i$)$]
\item\label{prop1.1} The poset  $\cP_{d}$ is a disjoint sum of subposets
$\cP_{d,k}$ for $k=1,2,\dots,d$.
\item\label{prop1.2} Each $\cP_{d,k}$ is an indecomposable poset in the sense that it is not
isomorphic to the disjoint sum of two non-empty posets. 
\end{enumerate}
\end{proposition}

\begin{proof}
Claim~\eqref{prop1.1} follows from the definitions since $d$ divides $v_1+2v_2+\dots+dv_d$,
for each $\mathbf{v}\in X_d$.

Note that $\mathbf{e}(k)\in \cP_{d,k}$. Therefore, to prove claim~\eqref{prop1.2} it is
enough to show that $\mathbf{e}(k)\prec \mathbf{v}$, for any $\mathbf{v}\in \cP_{d,k}$
of height at least $2$. 
However, if $\mathbf{v}$ has height at least $2$, then either $\mathbf{v}$
has a coordinate which is greater than or equal to $2$, or $\mathbf{v}$ has at least two
non-zero coordinates. Therefore there is $\mathbf{x}\in X_d$ such that 
$\mathbf{v}+\mathbf{x}\in \Lambda_{d}$. We have $\mathbf{v}+\mathbf{x}\prec \mathbf{v}$
and from the observation in the previous paragraph we
see that $\mathbf{v}+\mathbf{x}\in \Lambda_{d,k}$. Therefore 
$\mathbf{e}(k)\prec \mathbf{v}$ follows by induction on the height of $\mathbf{v}$.
This completes the proof.
\end{proof}

From the above proof it follows that, for $k\neq d$, the element $\mathbf{e}(k)$ is the
minimum element in $\cP_{d,k}$ and that the minimum element in $\cP_{d,d}$
is $\mathbf{0}:=(0,0,\dots,0)$.

\subsection{The case $d=1$}\label{s2.3}

In the case $d=1$, the map 
\begin{displaymath}
\begin{array}{ccc}
\cP_{1}&\to & (\mathbb{Z}_{\geq 0},<),\\
(i)&\mapsto & i
\end{array}
\end{displaymath}
is an isomorphism of posets. For $n\in\mathbb{Z}_{\geq 0}$, we have
\begin{displaymath}
I(n\mathbf{e}(1))=\{(0),(1),(2),\dots,(n)\} 
\end{displaymath}
and thus $C_{n}^{(1)}=n+1$. Note that in this case the poset $\cP_{1}=\cP_{1,1}$ is indecomposable.

\subsection{The case $d=2$}\label{s2.4}

Our first observation in this case  is that the maps
\begin{displaymath}
\begin{array}{ccc}
\cP_{2,1}&\to & \cP_{2,2},\\
\mathbf{v}&\mapsto & \mathbf{v}-(1,0)
\end{array}\quad\text{ and }\quad
\begin{array}{ccc}
\cP_{2,2}&\to & \cP_{2,1},\\
\mathbf{v}&\mapsto & \mathbf{v}+(1,0)
\end{array}
\end{displaymath}
are mutually inverse isomorphisms of posets. Consequently, we have
$C_{n}^{(2)}=C_{n+1}^{(2)}$, for all even $n\in\mathbb{Z}_{\geq 0}$.
The lower part of the Hasse diagram for $\cP_{2,2}$ is shown in Figure~\ref{fig1}.

\begin{figure}
\begin{displaymath}
\xymatrix{ 
\dots\ar@{-}[rd]&&\dots\ar@{-}[rd]\ar@{-}[ld]&&\dots\ar@{-}[rd]\ar@{-}[ld]&&\dots\ar@{-}[ld]\\ 
&(0,5)\ar@{-}[rd]&&(2,3)\ar@{-}[rd]\ar@{-}[ld]&&(4,1)\ar@{-}[rd]\ar@{-}[ld]&\\
&&(0,4)\ar@{-}[rd]&&(2,2)\ar@{-}[ld]\ar@{-}[rd]&&(4,0)\ar@{-}[ld]\\
&&&(0,3)\ar@{-}[rd]&&(2,1)\ar@{-}[rd]\ar@{-}[ld]&\\
&&&&(0,2)\ar@{-}[rd]&&(2,0)\ar@{-}[ld]\\
&&&&&(0,1)\ar@{-}[rd]&\\
&&&&&&(0,0)\\
} 
\end{displaymath}
\caption{Hasse diagram for $\cP_{2,2}$} \label{fig1}
\end{figure}

It follows immediately that, for $k\in\mathbb{Z}_{\geq 0}$, we have
\begin{displaymath}
C_{2k}^{(2)}=\frac{(k+1)(k+2)}{2}.
\end{displaymath}
We also note that $I(2k\mathbf{e}(1))\subset I(2(k+1)\mathbf{e}(1))$, for all
$k\in\mathbb{Z}_{\geq 0}$, and that 
\begin{displaymath}
\bigcup_{k\in\mathbb{Z}_{\geq 0}}I(2k\mathbf{e}(1))=\cP_{2,2}. 
\end{displaymath}
It is also worth pointing out that, for each $k\in\mathbb{Z}_{\geq 0}$, the poset
$I(2k\mathbf{e}(1))$ is isomorphic to the poset $I(2k\mathbf{e}(1))^{\mathrm{op}}$
(the latter is obtained from $I(2k\mathbf{e}(1))$ by reversing the partial order).

\section{The case $d=3$}\label{s3}

As we will show below, the case $d=3$ has several interesting connections 
to integral sequences. Our study of this case is the main part of 
the present paper.

\subsection{Isomorphism of $\cP_{3,1}$ and $\cP_{3,2}$}\label{s3.1}

The symmetric group $S_2$ acts on $\Lambda_3$ as follows: for $\mathbf{v}=(v_1,v_2,v_3)$
and $\pi\in S_2$ we have $\pi\cdot \mathbf{v}=(v_{\pi(1)},v_{\pi(2)},v_3)$. 
Note that the set $X_3$ (which can be found in Subsection~\ref{s2.2}) is invariant 
with respect to the action of $S_2$. Therefore this action induces an action on
$\cP_{3}$ by automorphisms. Using this action, we can swap $\mathbf{e}(1)$ and
$\mathbf{e}(2)$ and hence  $\cP_{3,1}$ and $\cP_{3,2}$ (cf. proof of Proposition~\ref{prop1}).
Therefore the posets $\cP_{3,1}$ and $\cP_{3,2}$ are isomorphic.

\subsection{An alternative description}\label{s3.2}

In this subsection we observe that $\cP_{3}$ can be defined by restriction from $\mathbb{Z}^3$.
This is a useful property for computations using computers. 

We mimic the definition of $\cP_{3}$ starting from $\mathbb{Z}^3$ instead of $\Lambda_3$.
Consider the set $X_3$ as defined in Subsection~\ref{s2.2}. Use \eqref{eq1} to define
the covering relation on $\mathbb{Z}^3$ and let $\prec'$ denote the partial order on 
$\mathbb{Z}^3$ induced by this covering relation. Our main observation here is the
following:

\begin{proposition}\label{prop2}
The relation $\prec$ coincides with the restriction of the relation $\prec'$ to $\Lambda_3$.
\end{proposition}

\begin{proof}
Let $\underline{\prec'}$ denote the restriction of the relation $\prec'$ to $\Lambda_3$.
Clearly, $\prec\subset\underline{\prec'}$, so we only need to show that 
$\underline{\prec'}\subset\prec$.

Let $\mathbf{v},\mathbf{w}\in \Lambda_3$ be such that $\mathbf{v}\prec'\mathbf{w}$. We have
to show that $\mathbf{v}\prec\mathbf{w}$. Assume that this is not the case and that
the pair $(\mathbf{v},\mathbf{w})$ satisfying $\mathbf{v}\prec'\mathbf{w}$ and
$\mathbf{v}\not \prec\mathbf{w}$ is chosen such that 
$\mathrm{ht}(\mathbf{w}-\mathbf{v})=k\in\mathbb{N}$ is minimal possible. As $\mathbf{v}\prec'\mathbf{w}$, there 
is a sequence of elements
$\mathbf{x}_1,\mathbf{x}_2,\dots,\mathbf{x}_k\in X_3$ such that 
\begin{displaymath}
\mathbf{v}=\mathbf{w}+ \mathbf{x}_1+\mathbf{x}_2+\dots+\mathbf{x}_k.
\end{displaymath}
Consider $\mathbf{v}_i=\mathbf{v}-\mathbf{x}_i$ for $i=1,2,\dots,k$. We claim that
all $\mathbf{v}_i\not\in \Lambda_3$. Indeed, if $\mathbf{v}_i\in \Lambda_3$, then
we would have $\mathbf{v}\prec\mathbf{v}_i$ and $\mathbf{v}_i\prec'\mathbf{w}$.
This would imply $\mathbf{v}_i\not \prec\mathbf{w}$
which would contradict our minimal choice of $k$. In particular, 
none of the $\mathbf{x}_i$'s equals $(0,0,-1)$ since $\mathbf{v}-(0,0,-1)\in \Lambda_3$
because $\mathbf{v}\in \Lambda_3$.

The next step is to show that none of the $\mathbf{x}_i$'s equals $(-1,-1,1)$.
Otherwise, without loss of generality we may assume that $\mathbf{x}_k=(-1,-1,1)$. Then we
have $\mathbf{v}_k\not\in \Lambda_3$ and hence 
$\mathbf{v}=(*,*,0)$ and $\mathbf{v}_k=(*,*,-1)$. Furthermore, we have
\begin{displaymath}
(*,*,-1)=\mathbf{v}_k=\mathbf{w}+ \mathbf{x}_1+\mathbf{x}_2+\dots+\mathbf{x}_{k-1}.
\end{displaymath}
Since $\mathbf{w}\in \Lambda_3$, the third coordinate in $\mathbf{w}$ is non-negative.
This means that at least one of the $\mathbf{x}_i$'s must have negative third coordinate.
The only element in $X_3$ with negative third coordinate is $(0,0,-1)$. However, in the
previous paragraph we already established that none of the $\mathbf{x}_i$'s equals $(0,0,-1)$,
a contradiction.
 
Therefore each  $\mathbf{x}_i$ is equal to either $(-2,1,0)$ or $(1,-2,0)$. 
Assume that all $\mathbf{x}_i$ are equal, say to $(-2,1,0)$ (the case of $(1,-2,0)$
is similar). Then $\mathbf{v}=\mathbf{w}+k(-2,1,0)$. Since both $\mathbf{v}$ and 
$\mathbf{w}$ are in $\Lambda_3$, we have $\mathbf{w}+i(-2,1,0)\in\Lambda_3$ for all $i$ such that
$1\leq i\leq k$. Therefore $\mathbf{v}\prec \mathbf{w}$, a contradiction.

The last paragraph establishes that at least one of the $\mathbf{x}_i$'s equals
$(-2,1,0)$ and at least one equals $(1,-2,0)$. This implies
$\mathbf{v}-(-2,1,0)-(1,-2,0)=\mathbf{v}+(1,1,0)\prec' \mathbf{w}$. At the same time, we have
\begin{displaymath}
\mathbf{v}+(0,0,1),\mathbf{v}+(1,1,0)\in \Lambda_3 
\end{displaymath}
as $\mathbf{v}\in \Lambda_3$ and 
\begin{displaymath}
\mathbf{v}\prec \mathbf{v}+(0,0,1)\prec \mathbf{v}+(0,0,1)+(1,1,-1)=\mathbf{v}+(1,1,0).
\end{displaymath}
This implies $\mathbf{v}+(1,1,0)\not \prec \mathbf{w}$ which again contradicts our
minimal choice of $k$. The claim follows.
\end{proof}

\subsection{Small values}\label{s3.3}

The table of $C_{n,h}^{(3)}$ for small values of $n$ is given in Figure~\ref{fig2}
(computed first by hands, up to $n=15$, and then checked and extended using 
Proposition~\ref{prop2} and MAPLE). Please ignore the underlines and the overlines for the moment.

\begin{figure}
{\fontsize{3pt}{0.5pt}
\begin{displaymath}
\begin{array}{c|cccccccccccccccccccccccccc}
C_{n,h}^{(3)}&1&1&2&4&5&7&11&13&17&23&27&33&42&48&57&69&78&90&106&118&134&154&170&190&215&235\\ 
\hline
25&&&&&&&&&&&&&&&&&&&&&&&&&&1\\
24&&&&&&&&&&&&&&&&&&&&&&&&&1&1\\
23&&&&&&&&&&&&&&&&&&&&&&&&1&1&2\\
22&&&&&&&&&&&&&&&&&&&&&&&1&1&2&3\\
21&&&&&&&&&&&&&&&&&&&&&&1&1&2&3&4\\
20&&&&&&&&&&&&&&&&&&&&&1&1&2&3&4&5\\
19&&&&&&&&&&&&&&&&&&&&1&1&2&3&4&5&7\\
18&&&&&&&&&&&&&&&&&&&1&1&2&3&4&5&7&8\\
17&&&&&&&&&&&&&&&&&&1&1&2&3&4&5&7&8&10\\
16&&&&&&&&&&&&&&&&&1&1&2&3&4&5&7&8&10&12\\
15&&&&&&&&&&&&&&&&1&1&2&3&4&5&7&8&10&12&14\\
14&&&&&&&&&&&&&&&1&1&2&3&4&5&7&8&10&12&14&16\\
13&&&&&&&&&&&&&&1&1&2&3&4&5&7&8&10&12&14&16&\underline{19}\\
12&&&&&&&&&&&&&1&1&2&3&4&5&7&8&10&12&14&\underline{16}&\underline{19}&20\\
11&&&&&&&&&&&&1&1&2&3&4&5&7&8&10&12&\underline{14}&\underline{16}&18&19&21\\
10&&&&&&&&&&&1&1&2&3&4&5&7&8&10&\underline{12}&\underline{14}&15&17&18&19&20\\
9&&&&&&&&&&1&1&2&3&4&5&7&8&\underline{10}&\underline{12}&13&14&16&16&17&18&\overline{18}\\
8&&&&&&&&&1&1&2&3&4&5&7&\underline{8}&\underline{10}&11&12&13&14&14&
\overline{15}&\overline{15}&\overline{15}&15\\
7&&&&&&&&1&1&2&3&4&5&\underline{7}&\underline{8}&9&10&11&11&\overline{12}&
\overline{12}&\overline{12}&12&12&12&12\\
6&&&&&&&1&1&2&3&4&\underline{5}&\underline{7}&7&8&9&\overline{9}&
\overline{9}&\overline{{10}}&9&9&{10}&9&9&{10}&9\\
5&&&&&&1&1&2&3&\underline{4}&\underline{5}&6&6&\overline{7}&\overline{7}&
\overline{7}&7&7&7&7&7&7&7&7&7&7\\
4&&&&&1&1&2&\underline{3}&\underline{4}&4&\overline{5}&\overline{5}&
\overline{5}&5&5&5&5&5&5&5&5&5&5&5&5&5\\
3&&&&1&1&\underline{2}&\underline{3}&\overline{3}&\overline{3}&
\overline{{4}}&3&3&{4}&3&3&{4}
&3&3&{4}&3&3&{4}&3&3&{4}&3\\
2&&&1&\underline{1}&\overline{\underline{2}}&\overline{2}&
\overline{2}&2&2&2&2&2&2&2&2&2&2&2&2&2&2&2&2&2&2&2\\
1&&\overline{\underline{1}}&\overline{\underline{1}}&\overline{1}
&1&1&1&1&1&1&1&1&1&1&1&1&1&1&1&1&1&1&1&1&1&1\\
0&{1}&&&{1}&&&{1}&&&{1}&&&
{1}&&&{1}&&&{1}&&&{1}&&&{1}&\\
\hline
h/n&0&1&2&3&4&5&6&7&8&9&10&11&12&13&14&15&16&17&18&19&20&21&22&23&24&25\\
\end{array}
\end{displaymath}
}
\caption{Values of $C_{n,h}^{(3)}$ for $n\leq 25$}\label{fig2}
\end{figure}

\subsection{Values of $C_{n,h}^{(3)}$ for large $h$}\label{s3.4}

The sequence $A001399(n)$ in \cite{OEIS} lists the number of partitions of $n$ into at most 3 parts.
Here are the first $25$ elements in this
sequence:
\begin{displaymath}
1, 1, 2, 3, 4, 5, 7, 8, 10, 12, 14, 16, 19, 21, 24,\dots
\end{displaymath}
Comparison with columns of Figure~\ref{fig2} suggests that the upper part of each
column in Figure~\ref{fig2} bounded by the underlined element 
is an initial segment of $A001399$. Indeed, we have the following claim:

\begin{proposition}\label{prop5}
For $h\geq \lceil\frac{n}{2}\rceil$, we have $C_{n,h}^{(3)}=A001399(n-h)$.
\end{proposition}

\begin{proof}
Let $(a,b,c)$ be a partition of $n-h$ in at most three parts, that is 
$a,b,c\in\mathbb{Z}_{\geq 0}$, $a\geq b\geq c$ and $a+b+c=n-h$. Then we claim that
\begin{displaymath}
\mathbf{v}_{(a,b,c)}:=(n,0,0)+a(-2,1,0)+b(-1,-1,1)+c(0,0,-1)\prec (n,0,0). 
\end{displaymath}
By Proposition~\ref{prop2}, it is enough to show that 
$\mathbf{v}_{(a,b,c)}\in \Lambda_3$. The latter however
follows from $2a+b\leq n$ (thanks to $h\geq \lceil\frac{n}{2}\rceil$)
and $b\geq c$ (thanks to the fact that $(a,b,c)$ is a partition).

The vectors $(-2,1,0)$, $(-1,-1,1)$ and $(0,0,-1)$ are linearly independent, which
implies that $\mathbf{v}_{(a,b,c)}\neq \mathbf{v}_{(a',b',c')}$ provided that 
$(a,b,c)\neq (a',b',c')$. Therefore $C_{n,h}^{(3)}\geq A001399(n-h)$.

Now consider some $v\in I(n\mathbf{e}(1))$ with height $h$. Then
\begin{equation}\label{eq2prime}
v = (n, 0, 0) + a(-2, 1, 0) + b(-1, -1, 1) + c(0, 0, -1) + d(1,-2, 0), 
\end{equation}
for some $a,b,c,d\in\mathbb{Z}_{\geq 0}$. Since the second coordinate of $v$ is nonnegative,
we have $a\geq d$. Since
\begin{displaymath}
(1, -2, 0) = -(-2, 1, 0)+(-1, -1, 1) + (0, 0, -1) ,
\end{displaymath}
we have
\begin{displaymath}
v = (n, 0, 0) + (a-d)(-2, 1, 0) + (b+d)(-1, -1, 1) + (c+d)(0, 0, -1)
\end{displaymath}
and thus may assume that $d=0$ in \eqref{eq2prime}.
We have $a+b+c=n-h$ since $v$ has height $h$,
$b\geq c$ as the third coordinate of $v$ is non-negative and
$a\geq b$ as the second coordinate of $v$ is non-negative.
Therefore $v=\mathbf{v}_{(a,b,c)}$, for the
partition $(a, b, c)$ of $n-h$. The claim of the
proposition follows.
\end{proof}

\subsection{Values of $C_{n,h}^{(3)}$ for small $h$}\label{s3.5}

We start this subsection with the following observation:

\begin{proposition}\label{prop7}
For $h\leq \lceil\frac{n}{3}\rceil$, we have $\Lambda_3^{(h)}\cap I(n\mathbf{e}(1))=
\Lambda_3^{(h)}\cap \Lambda_{3,k}$, where $k\in\{1,2,3\}$ is such that 
$n\equiv k(\mathrm{mod}\, 3)$. 
\end{proposition}

\begin{proof}
As $I(n\mathbf{e}(1))\subset \Lambda_{3,k}$, for our choice of $k$, to prove the assertion of
this proposition we only need to show that
$(\Lambda_3^{(h)}\cap \Lambda_{3,k})\subset (\Lambda_3^{(h)}\cap I(n\mathbf{e}(1)))$.
If $v\in (\Lambda_3^{(h)}\cap \Lambda_{3,k})\setminus (\Lambda_3^{(h)}\cap I(n\mathbf{e}(1)))$,
for some $h$, 
then $v+(0,0,1)\in (\Lambda_3^{(h+1)}\cap \Lambda_{3,k})\setminus (\Lambda_3^{(h+1)}\cap I(n\mathbf{e}(1)))$
since $I(n\mathbf{e}(1)))$ is an ideal.
Therefore it is enough to prove the proposition for $h=\lceil\frac{n}{3}\rceil$ which we from now
on assume. Set $q:=\lfloor\frac{n}{3}\rfloor$. We will have to consider three different cases
depending on $k$.

{\bf Case~1: $k=3$.} In this case $q=h=\frac{n}{3}$.
Let $(a,b,c)\in \Lambda_3^{(h)}\cap \Lambda_{3,k}$, that is $a,b,c\in\mathbb{Z}_{\geq 0}$,
$a+b+c=h$ and $3$ divides $a+2b$. In this case we have
\begin{equation}\label{eq3}
(n-3c-2b,b,c)=(n,0,0)+(b+c)(-2,1,0)+c(-1,-1,1)\prec (n,0,0).
\end{equation}
Now, $n-3c-2b=3a+b$. Since $3$ divides both $a+2b$ and $3a+3b$, it also divides $2a+b$.
Therefore  there is $p\in\mathbb{Z}_{\geq 0}$ such that $2a+b=3p$. We have
\begin{multline}\label{eq4}
(a,b,c)=(n-3c-2b-3p,b,c)=\\=(n-3c-2b,b,c)+p(-2,1,0)+p(-1,-1,1)+p(0,0,-1)\prec (n-3c-2b,b,c).
\end{multline}
Combining \eqref{eq3} and \eqref{eq4} implies $(a,b,c)\prec (n,0,0)$ and 
hence $(a,b,c)\in I(n\mathbf{e}(1))$.

{\bf Case~2: $k=2$.} In this case $q=h-1$ and $n=3h-1$. 
Let $(a,b,c)\in \Lambda_3^{(h)}\cap \Lambda_{3,k}$, 
that is $a,b,c\in\mathbb{Z}_{\geq 0}$, $a+b+c=h$ and $3$ divides $a+2b-2$. 
Formula~\eqref{eq3} still holds in this case.
Now, $n-3c-2b=3a+b-1$. Since $3$ divides both $a+2b-2$ and $3a+3b$, it also divides $2a+b-1$.
Therefore  there is $p\in\mathbb{Z}_{\geq 0}$ such that $2a+b-1=3p$. 
Formula~\eqref{eq4} still holds in this case.
Again it follows that $(a,b,c)\in I(n\mathbf{e}(1))$.

{\bf Case~3: $k=1$.} In this case $q=h-1$ and $n=3h-2$. 
Let $(a,b,c)\in \Lambda_3^{(h)}\cap \Lambda_{3,k}$, 
that is $a,b,c\in\mathbb{Z}_{\geq 0}$, $a+b+c=h$ and $3$ divides $a+2b-1$. 
Formula~\eqref{eq3} still holds in this case.
Now, $n-3c-2b=3a+b-2$. Since $3$ divides both $a+2b-1$ and $3a+3b$, it also divides $2a+b-2$.
Therefore  there is $p\in\mathbb{Z}_{\geq 0}$ such that $2a+b-2=3p$. 
Formula~\eqref{eq4} still holds in this case and we similarly obtain $(a,b,c)\in I(n\mathbf{e}(1))$.
\end{proof}

\begin{lemma}\label{lem8}
{\hspace{1mm}}

\begin{enumerate}[$($i$)$]
\item\label{lem8.1} If $3$ does not divide $h$, then $|\Lambda_{3,1}^{(h)}|=|\Lambda_{3,2}^{(h)}|=
|\Lambda_{3,3}^{(h)}|$.
\item\label{lem8.2} If $3$ divides $h$, then $|\Lambda_{3,1}^{(h)}|=|\Lambda_{3,2}^{(h)}|=
|\Lambda_{3,3}^{(h)}|-1$.
\end{enumerate}
\end{lemma}

\begin{proof}
We prove both statements at the same time  by induction on $h$. 
Let us arrange elements of $\Lambda_{3}^{(h)}$ in a triangular array as shown on 
Figure~\ref{fig3}.

\begin{figure}
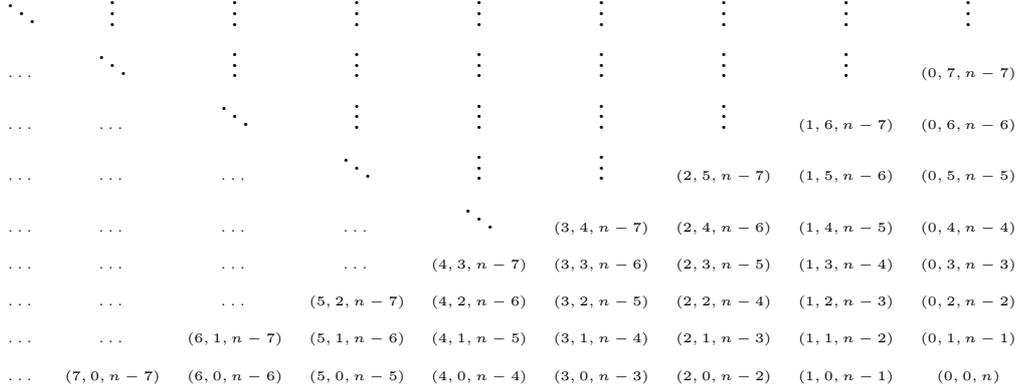

{\fontsize{5pt}{1pt}
\begin{displaymath}
\begin{array}{ccccccccc}
&&&&&&&&(0,h,0)\\
&&&&&&&\iddots&\vdots\\ 
&&&&&&(h-6,6,0)&\dots&(0,6,h-6)\\ 
&&&&&(h-5,5,0)&\dots&(1,5,h-6)&(0,5,h-5)\\ 
&&&&(h-4,4,0)&\dots&(2,4,h-6)&(1,4,h-5)&(0,4,h-4)\\ 
&&&(h-3,3,0)&\dots&(3,3,h-6)&(2,3,h-5)&(1,3,h-4)&(0,3,h-3)\\ 
&&(h-2,2,0)&\dots&(4,2,h-6)&(3,2,h-5)&(2,2,h-4)&(1,2,h-3)&(0,2,h-2)\\ 
&(h-1,1,0)&\dots&(5,1,h-6)&(4,1,h-5)&(3,1,h-4)&(2,1,h-3)&(1,1,h-2)&(0,1,h-1)\\ 
(h,0,0)&\dots&(6,0,h-6)&(5,0,h-5)&(4,0,h-4)&(3,0,h-3)&(2,0,h-2)&(1,0,h-1)&(0,0,h)\\ 
\end{array}
\end{displaymath}
}
\caption{Triangular arrangement of $\Lambda_{3}^{(h)}$}\label{fig3} 
\end{figure}

Writing down the residue modulo $3$ of the expression $a+2b$ for each element $(a,b,c)$
in Figure~\ref{fig3} we get
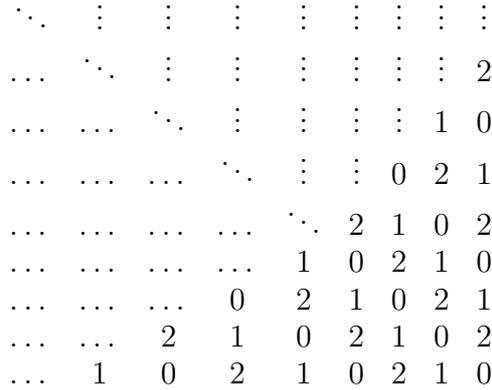
\begin{displaymath}
\begin{array}{ccccccccc}
&&&&&&&&\dots\\
&&&&&&&\dots&2\\ 
&&&&&&\dots&1&0\\ 
&&&&&\dots&0&2&1\\ 
&&&&\dots&2&1&0&2\\ 
&&&\dots&1&0&2&1&0\\ 
&&\dots&0&2&1&0&2&1\\ 
&\dots&2&1&0&2&1&0&2\\ 
\dots&1&0&2&1&0&2&1&0\\ 
\end{array}
\end{displaymath}
For a fixed $h$, the set $\Lambda_3^{(h)}$ corresponds to the first $h+1$ 
``bottom-left-to-top-right'' diagonals starting from the bottom right corner.
The induction step $h\to h+1$ corresponds to adding the next diagonal.

By a direct inspection of the above tables, we have  
\begin{gather*}
0=|\Lambda_{3,1}^{(0)}|=|\Lambda_{3,2}^{(0)}|=
|\Lambda_{3,3}^{(0)}|-1,\\ 
1=|\Lambda_{3,1}^{(1)}|=|\Lambda_{3,2}^{(1)}|=
|\Lambda_{3,3}^{(1)}|,\\ 
2=|\Lambda_{3,1}^{(2)}|=|\Lambda_{3,2}^{(2)}|=
|\Lambda_{3,3}^{(2)}|,
\end{gather*}
which establishes the basis of our induction.

Note that the residues in each diagonal follow a cyclic order on $0,1,2$
(as one step up along the diagonal decreases the first coordinate by $1$
and increases the second coordinate by $1$, thus changing $a+2b$ to $(a-1)+2(b+1)$). 
In particular, if the number of elements on a new diagonal is divisible by $3$, it 
contains the same number of $0$'s, $1$'s and $2$'s. This proves the induction 
step in the case when $3$ divides $h-1$.

If $3$ divides $h-2$, then the new diagonal contains an extra zero compared to the 
common number of $1$'s and $2$'s. If $3$ divides $h$, then the new diagonal 
contains one zero less than the  common number of $1$'s and $2$'s.
Put together this implies the  induction step and completes the proof of the proposition. 
\end{proof}

For a set $X$, we denote by $\delta_{X}$ the indicator function of $X$, that is
\begin{displaymath}
\delta_{X}(x)=\begin{cases}1,& x\in X;\\0,& x\not\in X.\end{cases} 
\end{displaymath}

\begin{corollary}\label{cor9}
For $h\leq \lceil\frac{n}{3}\rceil$ we have
\begin{displaymath}
C_{n,h}^{(3)}= \frac{(h+1)(h+2)+(6\delta_{3\mathbb{Z}}(n) -2)\delta_{3\mathbb{Z}}(h)}{6}.
\end{displaymath}
\end{corollary}

\begin{proof}
We have $C_{n,h}^{(3)}=|I(n\mathbf{e}(1))\cap \Lambda_3^{(h)}|$ by definition and
$|I(n\mathbf{e}(1))\cap \Lambda_3^{(h)}|=|\Lambda_3^{(h)}\cap \Lambda_{3,k}|$ by
Proposition~\ref{prop7}. Further, $|\Lambda_3^{(h)}\cap \Lambda_{3,k}|=|\Lambda_{3,k}^{(h)}|$
again by definition. Now the claim follows
applying Lemma~\ref{lem8} and Formula~\eqref{eq5} and
keeping $|\Lambda_{3}^{(h)}|=|\Lambda_{3,1}^{(h)}|+|\Lambda_{3,2}^{(h)}|+|\Lambda_{3,3}^{(h)}|$
in mind. 
\end{proof}

In the case when $3$ does not divide $n$, the sequence
$\frac{(h+1)(h+2)-2\delta_{3\mathbb{Z}}(h)}{6}$ is $A001840(h)$ from \cite{OEIS}.
In the case when $3$ divides $n$, the sequence
$\frac{(h+1)(h+2)+4\delta_{3\mathbb{Z}}(h)}{6}$ is $A007997(h+2)$ from \cite{OEIS}.
However, it seems that our interpretation of both these sequences does not appear
on \cite{OEIS} at the moment. We note that the sequence $A001840(h+1)-A001840(h)$
is the sequence
\begin{displaymath}
1,1,1,2,2,2,3,3,3,4,4,4,5,5,5,\dots,
\end{displaymath}
while the sequence $A007997(h+3)-A007997(h+2)$ is the sequence
\begin{displaymath}
1,0,1,2,1,2,3,2,3,4,3,4,5,4,5,6,5,6,\dots.
\end{displaymath}
The latter sequence should be compared with the fourth sequence which will be constructed
in Subsection~\ref{s3.6} below.

Each column in Figure~\ref{fig2} contains a unique overlined element. This element
corresponds to the upper bound $\lceil\frac{n}{3}\rceil$ for the value of $h$ for
which $C_{n,h}^{(3)}$ is given by Corollary~\ref{cor9}. In other words, this element 
and all elements below it in the same column are given by an initial segment of 
$A007997(h+2)$ or $A001840(h)$,  if $3$ does or does not divide $n$, respectively.

\begin{problem}\label{problem1}
Find a closed formula for $C_{n,h}^{(3)}$, where $\lceil\frac{n}{3}\rceil<h<\lceil\frac{n}{2}\rceil$.
\end{problem}

We do not know how hard this problem is, we do not see how to approach it in full generality.

\subsection{Sequence $A028289$}\label{s3.6}

The sequence $A028289$ in \cite{OEIS} lists coefficients in the expansion of 
$\frac{1+t^2+t^3+t^5}{(1-t)(1-t^3)(1-t^4)(1-t^6)}$. Here are the first $25$ elements in this
sequence:
\begin{displaymath}
1, 1, 2, 4, 5, 7, 11, 13, 17, 23, 27, 33, 42, 48, 57, 69, 78, 90, 106, 118, 134, 154, 170, 190, 215, 235,\dots 
\end{displaymath}
This sequence appears in \cite{CBC}. Comparison with the first row of 
Figure~\ref{fig2} suggests that $C_{n}^{(3)}=A028289(n)$ for all $n$. We will prove this in the next subsection.
In this subsection we propose two constructions of $A028289$, alternative to its definition on
\cite{OEIS}. The first construction consists of five combinatorial steps.

\begin{itemize}
\item Consider first the sequence $0,1,2,3,4,5\dots$ of all non-negative integers.
\item Construct the second sequence $0,1,1,2,2,3,3,4,4,5,5,\dots$  by repeating all
non-zero terms in the previous sequence twice.
\item Define the third sequence as the sequence of partial sums of the second sequence:
$0,1,2,4,6,9,12,16,20,25,30,\dots$.
\item Construction of the fourth sequence is the most complicated one. The sequence is:
\begin{displaymath}
1,0,1,2,1,2,4,2,4,6,4,6,9,6,9,12,9,12,16,12,16,20,16,20,\dots
\end{displaymath}
and this is obtained by adding appropriately shifted five-term frames of the form $(i,0,i,0,i)$,
where $i$ an element of the third sequence, as shown here:
\begin{displaymath}
\begin{array}{rcccccccccccccccccc}
&1&0&1&0&1\\
&&&&2&0&2&0&2\\
&&&&&&&4&0&4&0&4\\
&&&&&&&&&&6&0&6&0&6\\
&&&&&&&&&&&&&\dots&\dots&\dots\\
\hline
\text{add:}&1&0&1&2&1&2&4&2&4&6&4&6&\dots&\dots&\dots
\end{array}
\end{displaymath}
\item The final, fifth, sequence is the sequence of partial sums of the fourth sequence:
{\fontsize{10pt}{1pt}
\begin{displaymath}
1,1,2,4,5,7,11,13,17,23,27,33,42,48,57,69,78,90,106,118,134,154,170,190,\dots
\end{displaymath}
}
\end{itemize}

\begin{proposition}\label{prop4}
The fifth sequence constructed above coincides with  $A028289$.
\end{proposition}

\begin{proof}
Let us compute the generating function of all sequences constructed above. 
For the first sequence the generating function is
\begin{displaymath}
f(t):=\frac{t}{(1-t)^2}. 
\end{displaymath}
For the second sequence we get
\begin{displaymath}
f(t^2)+\frac{1}{t}f(t^2)=\frac{t+t^2}{(1-t^2)^2}.
\end{displaymath}
Convolution with $1,1,1,\dots$, that is the sequence with generating function $\frac{1}{1-t}$, 
implies that the generating function for the third sequence is
\begin{displaymath}
g(t)=\frac{t+t^2}{(1-t)(1-t^2)^2}. 
\end{displaymath}
The generating function for the fourth sequence is
\begin{displaymath}
\frac{g(t^3)}{t^3}+\frac{t^2g(t^3)}{t^3}+\frac{t^4g(t^3)}{t^3}= 
\frac{(1+t^3)(1+t^2+t^4)}{(1-t^6)^2(1-t^3)}= 
\frac{1+t^2+t^4}{(1-t^6)(1-t^3)^2}.
\end{displaymath}
Finally, yet another convolution with $1,1,1,\dots$ gives the generating function
\begin{displaymath}
\frac{1+t^2+t^4}{(1-t)(1-t^6)(1-t^3)^2}
\end{displaymath}
for the fifth sequence. The latter generating function coincides with  the generating function
\begin{displaymath}
\frac{1+t^2+t^3+t^5}{(1-t)(1-t^3)(1-t^4)(1-t^6)}
\end{displaymath}
of $A028289$ since
\begin{displaymath}
(1+t^2+t^4)(1-t^4)=1+t^2-t^6-t^8=(1+t^2+t^3+t^5)(1-t^3).
\end{displaymath}
The claim follows.
\end{proof}

Our second construction of $A028289$ (which is relevant for Theorem~\ref{thm33} in the following subsection)
uses the following observation:

\begin{lemma}\label{lem4-1}
We have
\begin{displaymath}
\frac{1+t^2+t^3+t^5}{(1-t)(1-t^3)(1-t^4)(1-t^6)}=
\frac{1}{1-t}\cdot\frac{1}{1-t^3}\cdot (1+t^2+t^3+t^4+t^5+t^7)\cdot
\frac{1}{(1-t^6)^2}.
\end{displaymath}
\end{lemma}

\begin{proof}
We have to check that 
\begin{displaymath}
\frac{1+t^2+t^3+t^5}{1-t^4}=
\frac{1+t^2+t^3+t^4+t^5+t^7}{1-t^6}.
\end{displaymath}
This is a straightforward computation.
\end{proof}

Lemma~\ref{lem4-1} implies that $A028289$ can be constructed in the following four combinatorial steps.

\begin{itemize}
\item Consider first the sequence $1,2,3,4,5\dots$ of all positive integers.
\item Construct the second sequence $1,0,1,1,1,1,2,1,2,2,2,2,3,2,\dots$  by repeating the
pattern $i,*,i,i,i,i,*,i$ of the terms in the previous sequence using shift in six positions.
\item Construct the third sequence $1,0,1,2,1,2,4,2,4,\dots$ by convolution of the second sequence with 
$1,0,0,1,0,0,1,0,0,1,\dots$.
\item The final, fourth, sequence is the sequence of partial sums of the third sequence:
{\fontsize{10pt}{1pt}
\begin{displaymath}
1,1,2,4,5,7,11,13,17,23,27,33,42,48,57,69,78,90,106,118,134,154,170,190,\dots
\end{displaymath}
}
\end{itemize}

\subsection{Computation of $C_n^{(3)}$}\label{s3.7}
Here we prove our first main result.

\begin{theorem}\label{thm33}
We have $C_{n}^{(3)}=A028289(n)$ for all $n\in\mathbb{Z}_{\geq0}$. 
\end{theorem}

\begin{proof}
Taking Lemma~\ref{lem4-1} into account, to prove the assertion of our theorem, it is enough to show that 
\begin{equation}\label{eqnn21}
\sum_{n\geq 0} C_{n}^{(3)}t^n=\frac{1}{1-t}\cdot\frac{1}{1-t^3}\cdot (1+t^2+t^3+t^4+t^5+t^7)\cdot
\frac{1}{(1-t^6)^2}.
\end{equation}

For a variable $n$, consider the sets
\begin{gather*}
\tilde{D}_n:=(n,0,0)+\mathbb{Z}(-2,1,0)+
\mathbb{Z}(-1,-1,1)+\mathbb{Z}(0,0,-1),\\
D_n:=(n,0,0)+\mathbb{Z}(-2,1,0)+\mathbb{Z}(-3,0,1).
\end{gather*}
For $v\in \tilde{D}_n$, we denote by $\Phi(v)=\Phi_n(v)$ the unique element in $D_n$ 
for which we have  $ \Phi(v)-v\in\mathbb{Z}(-3,0,0)$. This is well-defined as
the $\mathbb{Z}$-linear span of the linearly independent vectors
$(-2,1,0)$, $(-1,-1,0)$ and $(0,0,1)$ coincides with 
the $\mathbb{Z}$-linear span of the linearly independent vectors
$(-2,1,0)$, $(-3,0,1)$ and $(-3,0,0)$. 

Our first observation is the following:

\begin{lemma}\label{lemma33-1}
For any $\mathbf{v}\in I(n\mathbf{e}(1))$ we have
\begin{displaymath}
\Phi(\mathbf{v}) \in(n,0,0)+\mathbb{Z}_{\geq 0}(-2,1,0)+ \mathbb{Z}_{\geq 0}(-3,0,1). 
\end{displaymath} 
\end{lemma}

\begin{proof}
We have 
\begin{displaymath}
\mathbf{v}=(n,0,0)+a(-2,1,0)+b(1,-2,0)+c(-1,-1,1)+d(0,0,-1),
\end{displaymath}
for $a,b,c,d,\in \mathbb{Z}_{\geq 0}$, by definition. From 
$\mathbf{v}\in I(n\mathbf{e}(1))$, by looking at the second and the third coordinates of $\mathbf{v}$,
we obtain $a-2b-c\geq 0$ and $c-d\geq 0$. Now we rewrite
\begin{displaymath}
\mathbf{v}=(n,0,0)+(a-2b-c)(-2,1,0)+(c-d)(-3,0,1)+(b+d)(-3,0,0).
\end{displaymath}
As $(n,0,0)+(a-2b-c)(-2,1,0)+(c-d)(-3,0,1)\in D_n$, the claim of the lemma follows from the 
definition of $\Phi$.
\end{proof}

Motivated by Lemma~\ref{lemma33-1}, for $i\in\mathbb{Z}_{\geq 0}$, set $f_i:=|T_i|$, where
\begin{displaymath}
T_i:=\{\mathbf{v}\in (n,0,0)+\mathbb{Z}_{\geq 0}(-2,1,0)+ \mathbb{Z}_{\geq 0}(-3,0,1)\,:\,
\mathbf{v}=(n-i,*,*)\}. 
\end{displaymath}

\begin{lemma}\label{lemma33-2}
We have
\begin{displaymath}
\sum_{i\geq 0}f_it^i=\frac{1+t^2+t^3+t^4+t^5+t^7}{(1-t^6)^2}. 
\end{displaymath} 
\end{lemma}

\begin{proof}
The proof is illustrated in Figure~\ref{fig7}. By definition, $f_i$ enumerates 
lattice points of the two-dimensional cone 
$\mathbf{C}:=(n,0,0)+\mathbb{Z}_{\geq 0}(-2,1,0)+ \mathbb{Z}_{\geq 0}(-3,0,1)$
(the points of $\mathbf{C}$ are depicted as bullet points in Figure~\ref{fig7}) 
belonging to the line, in the plane determined by the cone, given by the condition 
that the first coordinate of the point on the line equals $n-i$ (these lines are
depicted as dashed lines in Figure~\ref{fig7}).

A direct calculation gives the following values for small $i$:
\begin{equation}\label{eqnew91}
\begin{array}{c|cccccccc}
i&0&1&2&3&4&5&6&7\\
\hline
f_i&1&0&1&1&1&1&2&1
\end{array}
\end{equation}
For $i=6$, we, for the first time, have $f_i=2>1$. This implies that 
$f_i$ satisfies the recursion $f_{i+6}=f_{i}+1$, for all $i\geq 0$. 
Indeed, from Figure~\ref{fig7} we see that  $\mathbf{C}$ is the disjoint union 
of the shifted copy $\mathbf{C}+(-6,0,2)$ of $\mathbf{C}$ and the remaining strip of width $2$
going north-east in Figure~\ref{fig7}. The recursion follows by noting that, for $i\neq 1$, 
the corresponding dashed line intersects the remaining strip in exactly one point and that,
for $i\geq 6$, the intersection of the corresponding dashed line with $\mathbf{C}+(-6,0,2)$ has exactly
$f_{i-6}$ points, due to the observation above and the linearity of the definitions.

The claim of the lemma follows by combining the recursion $f_{i+6}=f_{i}+1$ with
the initial values listed in \eqref{eqnew91}.
\end{proof}

\begin{figure}
\special{em:linewidth 0.4pt} \unitlength 0.80mm
\begin{picture}(200.00,150.00)
%%%%%%%%%%%%%%%%%%%%%%%%%%%%%%%%%%%%%%%%%%%%%%%%%%%%%%%%
\put(80.00,130.00){\makebox(0,0)[cc]{$\bullet$}}
\put(100.00,130.00){\makebox(0,0)[cc]{$\bullet$}}
\put(120.00,130.00){\makebox(0,0)[cc]{$\bullet$}}
\put(140.00,130.00){\makebox(0,0)[cc]{$\bullet$}}
\put(160.00,130.00){\makebox(0,0)[cc]{$\bullet$}}
\put(90.00,110.00){\makebox(0,0)[cc]{$\bullet$}}
\put(110.00,110.00){\makebox(0,0)[cc]{$\bullet$}}
\put(130.00,110.00){\makebox(0,0)[cc]{$\bullet$}}
\put(150.00,110.00){\makebox(0,0)[cc]{$\bullet$}}
\put(100.00,90.00){\makebox(0,0)[cc]{$\bullet$}}
\put(120.00,90.00){\makebox(0,0)[cc]{$\bullet$}}
\put(140.00,90.00){\makebox(0,0)[cc]{$\bullet$}}
\put(110.00,70.00){\makebox(0,0)[cc]{$\bullet$}}
\put(130.00,70.00){\makebox(0,0)[cc]{$\bullet$}}
\put(120.00,50.00){\makebox(0,0)[cc]{$\bullet$}}
%%%%%%%%%%%%%%%%%%%%%%%%%%%%%%%%%%%%%%%%%%%%%%%%%%
\dashline(60.00,90.00)(170.00,134.00){1}
\dashline(60.00,82.00)(170.00,126.00){1}
\dashline(60.00,74.00)(170.00,118.00){1}
\dashline(60.00,66.00)(170.00,110.00){1}
\dashline(60.00,58.00)(170.00,102.00){1}
\dashline(60.00,50.00)(170.00,94.00){1}
\dashline(60.00,42.00)(170.00,86.00){1}
\dashline(60.00,34.00)(170.00,78.00){1}
\dashline(60.00,26.00)(170.00,70.00){1}
%%%%%%%%%%%%%%%%%%%%%%%%%%%%%%%%%%%%%%%%%%%%%%%%%%
%\dottedline[.]{2}(60.00,70.00)(170.00,120.00)
%%%%%%%%%%%%%%%%%%%%%%%%%%%%%%%%%%%%%%%%%%%%%%%%%%
\linethickness{1pt}
\drawline(60.00,20.00)(60.00,110.00)
\drawline(50.00,20.00)(50.00,110.00)
\drawline(40.00,20.00)(40.00,110.00)
\drawline(30.00,20.00)(30.00,110.00)
\drawline(20.00,20.00)(20.00,110.00)
%%%%%%%%%%%%%%%%%%%%%%%%%%%%%%%%%%%%%%%%%%%%%%%%%%
\linethickness{1pt}
\drawline(120.00,50.00)(160.00,130.00)
\drawline(110.00,70.00)(140.00,130.00)
\drawline(100.00,90.00)(120.00,130.00)
\drawline(90.00,110.00)(100.00,130.00)
\drawline(128.00,67.70)(130.00,70.00)
\drawline(129.45,66.95)(130.00,70.00)
\drawline(120.00,50.00)(80.00,130.00)
\drawline(130.00,70.00)(100.00,130.00)
\drawline(140.00,90.00)(120.00,130.00)
\drawline(150.00,110.00)(140.00,130.00)
\drawline(110.00,70.00)(130.00,70.00)
\drawline(100.00,90.00)(140.00,90.00)
\drawline(90.00,110.00)(150.00,110.00)
\drawline(80.00,130.00)(160.00,130.00)
\drawline(111.80,68.00)(110.00,70.00)
\drawline(110.50,67.30)(110.00,70.00)
%%%%%%%%%%%%%%%%%%%%%%%%%%%%%%%%%%%%%%%%%%%%%%%%%%
\put(55.00,10.00){\makebox(0,0)[cc]{$i$}}
\put(45.00,10.00){\makebox(0,0)[cc]{$f_i$}}
\put(35.00,10.00){\makebox(0,0)[cc]{$g_i$}}
\put(25.00,10.00){\makebox(0,0)[cc]{$C_i^{(3)}$}}
\put(55.00,26.00){\makebox(0,0)[cc]{$0$}}
\put(45.00,26.00){\makebox(0,0)[cc]{$1$}}
\put(35.00,26.00){\makebox(0,0)[cc]{$1$}}
\put(25.00,26.00){\makebox(0,0)[cc]{$1$}}
\put(55.00,34.00){\makebox(0,0)[cc]{$1$}}
\put(45.00,34.00){\makebox(0,0)[cc]{$0$}}
\put(35.00,34.00){\makebox(0,0)[cc]{$0$}}
\put(25.00,34.00){\makebox(0,0)[cc]{$1$}}
\put(55.00,42.00){\makebox(0,0)[cc]{$2$}}
\put(45.00,42.00){\makebox(0,0)[cc]{$1$}}
\put(35.00,42.00){\makebox(0,0)[cc]{$1$}}
\put(25.00,42.00){\makebox(0,0)[cc]{$2$}}
\put(55.00,50.00){\makebox(0,0)[cc]{$3$}}
\put(45.00,50.00){\makebox(0,0)[cc]{$1$}}
\put(35.00,50.00){\makebox(0,0)[cc]{$2$}}
\put(25.00,50.00){\makebox(0,0)[cc]{$4$}}
\put(55.00,58.00){\makebox(0,0)[cc]{$4$}}
\put(45.00,58.00){\makebox(0,0)[cc]{$1$}}
\put(35.00,58.00){\makebox(0,0)[cc]{$1$}}
\put(25.00,58.00){\makebox(0,0)[cc]{$5$}}
\put(55.00,66.00){\makebox(0,0)[cc]{$5$}}
\put(45.00,66.00){\makebox(0,0)[cc]{$1$}}
\put(35.00,66.00){\makebox(0,0)[cc]{$2$}}
\put(25.00,66.00){\makebox(0,0)[cc]{$7$}}
\put(55.00,74.00){\makebox(0,0)[cc]{$6$}}
\put(45.00,74.00){\makebox(0,0)[cc]{$2$}}
\put(35.00,74.00){\makebox(0,0)[cc]{$4$}}
\put(25.00,74.00){\makebox(0,0)[cc]{$11$}}
\put(55.00,82.00){\makebox(0,0)[cc]{$7$}}
\put(45.00,82.00){\makebox(0,0)[cc]{$1$}}
\put(35.00,82.00){\makebox(0,0)[cc]{$2$}}
\put(25.00,82.00){\makebox(0,0)[cc]{$13$}}
\put(55.00,90.00){\makebox(0,0)[cc]{$8$}}
\put(45.00,90.00){\makebox(0,0)[cc]{$2$}}
\put(35.00,90.00){\makebox(0,0)[cc]{$4$}}
\put(25.00,90.00){\makebox(0,0)[cc]{$17$}}
\put(55.00,98.00){\makebox(0,0)[cc]{$\vdots$}}
\put(45.00,98.00){\makebox(0,0)[cc]{$\vdots$}}
\put(35.00,98.00){\makebox(0,0)[cc]{$\vdots$}}
\put(25.00,98.00){\makebox(0,0)[cc]{$\vdots$}}
%%%%%%%%%%%%%%%%%%%%%%%%%%%%%%%%%%%%%%%%%%%%%%%%%%%%%%%%
\put(120.00,46.00){\makebox(0,0)[cc]{{\fontsize{4pt}{1pt}$(n,0,0)$}}}
\put(120.00,86.00){\makebox(0,0)[cc]{{\fontsize{4pt}{1pt}$(n-5,1,1)$}}}
\put(130.00,106.00){\makebox(0,0)[cc]{{\fontsize{4pt}{1pt}$(n-7,2,1)$}}}
\put(110.00,106.00){\makebox(0,0)[cc]{{\fontsize{4pt}{1pt}$(n-8,1,2)$}}}
\put(140.00,134.00){\makebox(0,0)[cc]{{\fontsize{4pt}{1pt}$(n-9,3,1)$}}}
\put(120.00,134.00){\makebox(0,0)[cc]{{\fontsize{4pt}{1pt}$(n-10,2,2)$}}}
\put(100.00,134.00){\makebox(0,0)[cc]{{\fontsize{4pt}{1pt}$(n-11,1,3)$}}}
\put(140.00,70.00){\makebox(0,0)[cc]{{\fontsize{4pt}{1pt}$(n-2,1,0)$}}}
\put(100.00,70.00){\makebox(0,0)[cc]{{\fontsize{4pt}{1pt}$(n-3,0,1)$}}}
\put(150.00,90.00){\makebox(0,0)[cc]{{\fontsize{4pt}{1pt}$(n-4,2,0)$}}}
\put(160.00,110.00){\makebox(0,0)[cc]{{\fontsize{4pt}{1pt}$(n-6,3,0)$}}}
\put(170.00,130.00){\makebox(0,0)[cc]{{\fontsize{4pt}{1pt}$(n-8,4,0)$}}}
\put(90.00,90.00){\makebox(0,0)[cc]{{\fontsize{4pt}{1pt}$(n-6,0,2)$}}}
\put(80.00,110.00){\makebox(0,0)[cc]{{\fontsize{4pt}{1pt}$(n-9,0,3)$}}}
\put(70.00,130.00){\makebox(0,0)[cc]{{\fontsize{4pt}{1pt}$(n-12,0,4)$}}}
\put(131.00,58.00){\makebox(0,0)[cc]{{\fontsize{4pt}{1pt}$(-2,1,0)$}}}
\put(109.00,58.00){\makebox(0,0)[cc]{{\fontsize{4pt}{1pt}$(-3,0,1)$}}}
%%%%%%%%%%%%%%%%%%%%%%%%%%%%%%%%%%%%%%%%%%%%%%%%%%%%%%%%
\end{picture}
\caption{Geometric illustration of the proof of Lemma~\ref{lemma33-2} and Theorem~\ref{thm33}} \label{fig7}
\end{figure}
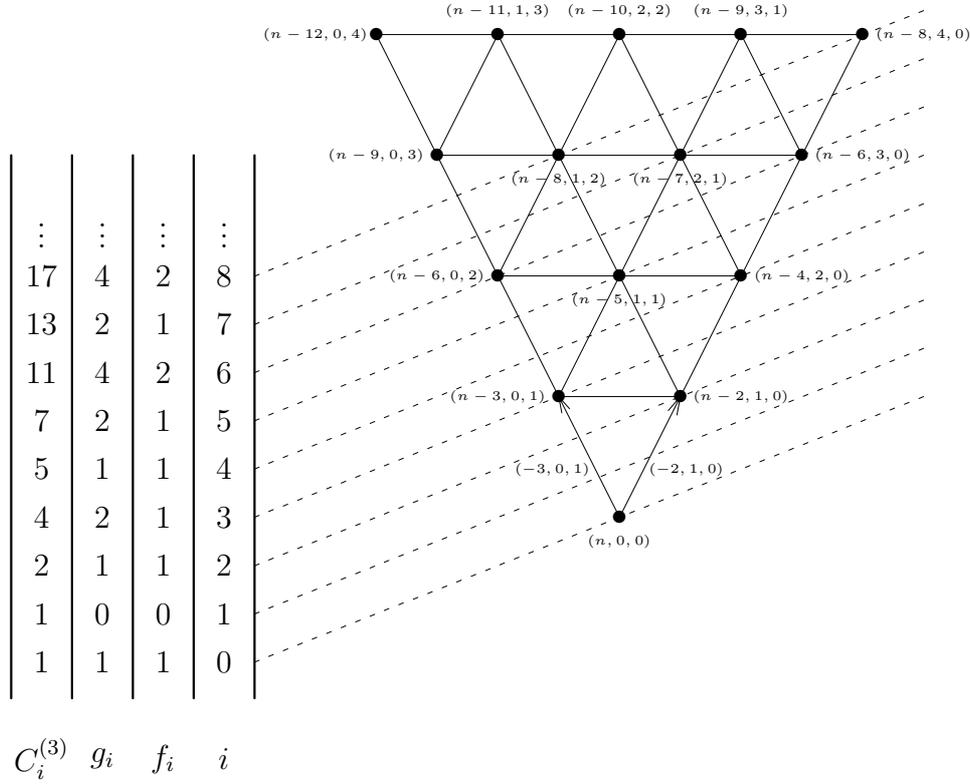

For each $n\in\mathbb{Z}_{\geq 0}$, mapping $\mathbf{v}\mapsto \mathbf{v}+(1,0,0)$ defines an injection
from $I((n-1)\mathbf{e}(1))$ to $I(n\mathbf{e}(1))$. Set $g_n:=C_{n}^{(3)}-C_{n-1}^{(3)}\geq 0$
(under the convention $C_{-1}^{(3)}=0$). Then we have
\begin{equation}\label{eqnn25}
C_{n}^{(3)}=g_n+g_{n-1}+\dots+g_0 
\end{equation}
by construction. This reduces the claim of the theorem to the following crucial observation:

\begin{lemma}\label{lemma33-3}
We have  $g_n=f_n+f_{n-3}+f_{n-6}+\dots$, for all $n$, where we assume $f_k=0$, for $k<0$.
\end{lemma}

\begin{proof}
First of all, we claim  that the map 
\begin{displaymath}
\begin{array}{rcl}
I((n-1)\mathbf{e}(1))&\to &I(n\mathbf{e}(1))\\
\mathbf{v}&\mapsto &\mathbf{v}+(1,0,0)
\end{array}
\end{displaymath}
induces a bijection between $I((n-1)\mathbf{e}(1))$ and the set 
\begin{displaymath}
\{\mathbf{v}=(v_1,v_2,v_3)\in I(n\mathbf{e}(1))\,:\,
v_1\neq 0\}. 
\end{displaymath}
Indeed, the inverse map is easily seen to be given by $\mathbf{w}\mapsto \mathbf{w}-(1,0,0)$.

Therefore we need to show that $f_n+f_{n-3}+f_{n-6}+\dots$ enumerates the set
\begin{displaymath}
R:=\{\mathbf{v}=(v_1,v_2,v_3)\in I(n\mathbf{e}(1))\,:\,
v_1=0\}. 
\end{displaymath}
Note that the restriction of $\Phi=\Phi_n$ to $R$ is injective since the linear span of $R$ intersects 
$\mathbb{Z}(-3,0,0)$, see the definition of $\Phi$, in exactly one element, namely $(0,0,0)$.  Therefore it is enough to enumerate $|\Phi(R)|$.
We claim that  
\begin{equation}\label{eq531new}
\Phi(R)=T_{n}\cup T_{n-3}\cup T_{n-6}\cup \dots 
\end{equation}
(note that this union is automatically disjoint), which is a reformulation of the assertion of the 
lemma due to the rule of sum.

We have $\Phi(R)\subset T_{n}\cup T_{n-1}\cup T_{n-2}\cup\dots$ by Lemma~\ref{lemma33-1}.
As the first coordinate of each $v\in R$ is zero and
$\Phi(v)-v\in\mathbb{Z}(-3,0,0)$ by definition, it follows that the
first coordinate of $\Phi(v)$, for $v\in R$, must be divisible by $3$, that is
$\Phi(R)\subset T_{n}\cup T_{n-3}\cup T_{n-6}\cup \dots$.

For the inverse inclusion, take some $v=(n-2a-3b,a,b)\in T_{n-3j}$, for some $j\geq 0$.
Then $n-2a-3b=3j$, by definition. Clearly, $v\prec n\mathbf{e}(1)$.
If $j=0$, then $v\in R$ and $\Phi(v)=v$.
If $j>0$, then we apply to $v$ the following procedure $j$ times:
first add $(-2,1,0)$, then add  $(-1,-1,1)$, then add $(0,0,-1)$.
It is easy to see that, by doing this, we follow the order $\prec$ inside $I(n\mathbf{e}(1))$.
As $(-3,0,0)=(-2,1,0)+(-1,-1,1)+(0,0,-1)$, as the result, we obtain 
$(0,a,b)\prec v$. In particular, we have $(0,a,b)\in R$.
From the definition of $\Phi$ and Lemma~\ref{lemma33-1}, we obtain
$\Phi((0,a,b))=v$. Therefore $T_{n}\cup T_{n-3}\cup T_{n-6}\cup \dots\subset \Phi(R)$.
This proves \eqref{eq531new} and completes the proof of the lemma. 
\end{proof}

To prove our theorem, we need to prove \eqref{eqnn21}. 
Lemma~\ref{lemma33-2} corresponds to the last two factors on the right hand side of \eqref{eqnn21},
formula~\eqref{eqnn25} corresponds to the first factor and
Lemma~\ref{lemma33-3} corresponds to the second factor. The claim of the theorem now follows
using the rule of sum.
The intuitive picture behind this proof is given in Figure~\ref{fig7}.
\end{proof}

As a direct consequence of Theorem~\ref{thm33} and \cite{CBCBB,CBC}, we get:

\begin{corollary}\label{corthm33}
For $i\in\mathbb{N}$ we have:
\begin{displaymath}
\begin{array}{rcl}
C_{3(i-1)}^{(3)}&=&\frac{1}{8}\big((i+1)(2i^2+i+1)-\frac{1}{2}(1+(-1)^i)\big),\\ 
C_{3(i-1)+1}^{(3)}&=&\frac{1}{8}\big((i+1)(2i^2+3i-1)+\frac{1}{2}(1+(-1)^i)\big),\\ 
C_{3(i-1)+2}^{(3)}&=&\frac{1}{8}\big((i+1)(2i^2+5i+1)-\frac{1}{2}(1+(-1)^i)\big). 
\end{array}
\end{displaymath}
\end{corollary}

\section{$C_{n}^{(3)}$ and hollow hexagons}\label{s55}

\subsection{Triangular tilings}\label{s55.1}

Consider a regular triangular tiling of a Euclidean plane as shown in Figure~\ref{fign1}.
We assume that the side of the basic equilateral triangle (the fundamental region) of this tiling has length $1$.
Each intersection point if called a {\em vertex} of the tiling. Each straight line of the tiling
is called a {\em tiling line}. A horizontal tiling line will be called a line of {\em type $1$}.
A tiling line of  {\em type $2$} is a tiling line obtained from a tiling line of type $1$ by
a {\em clockwise} rotation by $\frac{\pi}{3}$.
A tiling line of  {\em type $3$} is a tiling line obtained from a tiling line of type $1$ by
a {\em clockwise} rotation by $\frac{2\pi}{3}$.

\begin{figure}
{\fontsize{5pt}{1pt}
\begin{picture}(200.00,150.00)
%%%%%%%%%%%%%%%%%%%%%%%%%%%%%%%%%%%%
\thinlines
\drawline(00.00,00.00)(200.00,00.00)
\drawline(00.00,34.64)(200.00,34.64)
\drawline(00.00,69.28)(200.00,69.28)
\drawline(00.00,103.92)(200.00,103.92)
\drawline(00.00,138.56)(200.00,138.56)
\drawline(10.00,17.32)(190.00,17.32)
\drawline(10.00,51.96)(190.00,51.96)
\drawline(10.00,86.60)(190.00,86.60)
\drawline(10.00,121.24)(190.00,121.24)
\drawline(00.00,34.64)(20.00,00.00)
\drawline(00.00,69.28)(40.00,00.00)
\drawline(00.00,103.92)(60.00,00.00)
\drawline(00.00,138.56)(80.00,00.00)
\drawline(20.00,138.56)(100.00,00.00)
\drawline(40.00,138.56)(120.00,00.00)
\drawline(60.00,138.56)(140.00,00.00)
\drawline(80.00,138.56)(160.00,00.00)
\drawline(100.00,138.56)(180.00,00.00)
\drawline(120.00,138.56)(200.00,00.00)
\drawline(140.00,138.56)(200.00,34.64)
\drawline(160.00,138.56)(200.00,69.28)
\drawline(180.00,138.56)(200.00,103.92)
\drawline(180.00,00.00)(200.00,34.64)
\drawline(160.00,00.00)(200.00,69.28)
\drawline(140.00,00.00)(200.00,103.92)
\drawline(120.00,00.00)(200.00,138.56)
\drawline(100.00,00.00)(180.00,138.56)
\drawline(80.00,00.00)(160.00,138.56)
\drawline(60.00,00.00)(140.00,138.56)
\drawline(40.00,00.00)(120.00,138.56)
\drawline(20.00,00.00)(100.00,138.56)
\drawline(00.00,00.00)(80.00,138.56)
\drawline(00.00,34.64)(60.00,138.56)
\drawline(00.00,69.28)(40.00,138.56)
\drawline(00.00,103.92)(20.00,138.56)
%%%%%%%%%%%%%%%%%%%%%%%%%%%%%%%%%%%%%
\put(00.00,00.00){\makebox(0,0)[cc]{$\bullet$}} 
\put(20.00,00.00){\makebox(0,0)[cc]{$\bullet$}} 
\put(40.00,00.00){\makebox(0,0)[cc]{$\bullet$}} 
\put(60.00,00.00){\makebox(0,0)[cc]{$\bullet$}} 
\put(80.00,00.00){\makebox(0,0)[cc]{$\bullet$}} 
\put(100.00,00.00){\makebox(0,0)[cc]{$\bullet$}} 
\put(120.00,00.00){\makebox(0,0)[cc]{$\bullet$}} 
\put(140.00,00.00){\makebox(0,0)[cc]{$\bullet$}} 
\put(160.00,00.00){\makebox(0,0)[cc]{$\bullet$}} 
\put(180.00,00.00){\makebox(0,0)[cc]{$\bullet$}} 
\put(200.00,00.00){\makebox(0,0)[cc]{$\bullet$}} 
\put(10.00,17.32){\makebox(0,0)[cc]{$\bullet$}} 
\put(30.00,17.32){\makebox(0,0)[cc]{$\bullet$}} 
\put(50.00,17.32){\makebox(0,0)[cc]{$\bullet$}} 
\put(70.00,17.32){\makebox(0,0)[cc]{$\bullet$}} 
\put(90.00,17.32){\makebox(0,0)[cc]{$\bullet$}} 
\put(110.00,17.32){\makebox(0,0)[cc]{$\bullet$}} 
\put(130.00,17.32){\makebox(0,0)[cc]{$\bullet$}} 
\put(150.00,17.32){\makebox(0,0)[cc]{$\bullet$}} 
\put(170.00,17.32){\makebox(0,0)[cc]{$\bullet$}} 
\put(190.00,17.32){\makebox(0,0)[cc]{$\bullet$}} 
\put(00.00,34.64){\makebox(0,0)[cc]{$\bullet$}} 
\put(20.00,34.64){\makebox(0,0)[cc]{$\bullet$}} 
\put(40.00,34.64){\makebox(0,0)[cc]{$\bullet$}} 
\put(60.00,34.64){\makebox(0,0)[cc]{$\bullet$}} 
\put(80.00,34.64){\makebox(0,0)[cc]{$\bullet$}} 
\put(100.00,34.64){\makebox(0,0)[cc]{$\bullet$}} 
\put(120.00,34.64){\makebox(0,0)[cc]{$\bullet$}} 
\put(140.00,34.64){\makebox(0,0)[cc]{$\bullet$}} 
\put(160.00,34.64){\makebox(0,0)[cc]{$\bullet$}} 
\put(180.00,34.64){\makebox(0,0)[cc]{$\bullet$}} 
\put(200.00,34.64){\makebox(0,0)[cc]{$\bullet$}} 
\put(10.00,51.96){\makebox(0,0)[cc]{$\bullet$}} 
\put(30.00,51.96){\makebox(0,0)[cc]{$\bullet$}} 
\put(50.00,51.96){\makebox(0,0)[cc]{$\bullet$}} 
\put(70.00,51.96){\makebox(0,0)[cc]{$\bullet$}} 
\put(90.00,51.96){\makebox(0,0)[cc]{$\bullet$}} 
\put(110.00,51.96){\makebox(0,0)[cc]{$\bullet$}} 
\put(130.00,51.96){\makebox(0,0)[cc]{$\bullet$}} 
\put(150.00,51.96){\makebox(0,0)[cc]{$\bullet$}} 
\put(170.00,51.96){\makebox(0,0)[cc]{$\bullet$}} 
\put(190.00,51.96){\makebox(0,0)[cc]{$\bullet$}} 
\put(00.00,69.28){\makebox(0,0)[cc]{$\bullet$}} 
\put(20.00,69.28){\makebox(0,0)[cc]{$\bullet$}} 
\put(40.00,69.28){\makebox(0,0)[cc]{$\bullet$}} 
\put(60.00,69.28){\makebox(0,0)[cc]{$\bullet$}} 
\put(80.00,69.28){\makebox(0,0)[cc]{$\bullet$}} 
\put(100.00,69.28){\makebox(0,0)[cc]{$\bullet$}} 
\put(120.00,69.28){\makebox(0,0)[cc]{$\bullet$}} 
\put(140.00,69.28){\makebox(0,0)[cc]{$\bullet$}} 
\put(160.00,69.28){\makebox(0,0)[cc]{$\bullet$}} 
\put(180.00,69.28){\makebox(0,0)[cc]{$\bullet$}} 
\put(200.00,69.28){\makebox(0,0)[cc]{$\bullet$}} 
\put(10.00,86.60){\makebox(0,0)[cc]{$\bullet$}} 
\put(30.00,86.60){\makebox(0,0)[cc]{$\bullet$}} 
\put(50.00,86.60){\makebox(0,0)[cc]{$\bullet$}} 
\put(70.00,86.60){\makebox(0,0)[cc]{$\bullet$}} 
\put(90.00,86.60){\makebox(0,0)[cc]{$\bullet$}} 
\put(110.00,86.60){\makebox(0,0)[cc]{$\bullet$}} 
\put(130.00,86.60){\makebox(0,0)[cc]{$\bullet$}} 
\put(150.00,86.60){\makebox(0,0)[cc]{$\bullet$}} 
\put(170.00,86.60){\makebox(0,0)[cc]{$\bullet$}} 
\put(190.00,86.60){\makebox(0,0)[cc]{$\bullet$}} 
\put(00.00,103.92){\makebox(0,0)[cc]{$\bullet$}} 
\put(20.00,103.92){\makebox(0,0)[cc]{$\bullet$}} 
\put(40.00,103.92){\makebox(0,0)[cc]{$\bullet$}} 
\put(60.00,103.92){\makebox(0,0)[cc]{$\bullet$}} 
\put(80.00,103.92){\makebox(0,0)[cc]{$\bullet$}} 
\put(100.00,103.92){\makebox(0,0)[cc]{$\bullet$}} 
\put(120.00,103.92){\makebox(0,0)[cc]{$\bullet$}} 
\put(140.00,103.92){\makebox(0,0)[cc]{$\bullet$}} 
\put(160.00,103.92){\makebox(0,0)[cc]{$\bullet$}} 
\put(180.00,103.92){\makebox(0,0)[cc]{$\bullet$}} 
\put(200.00,103.92){\makebox(0,0)[cc]{$\bullet$}} 
\put(10.00,121.24){\makebox(0,0)[cc]{$\bullet$}} 
\put(30.00,121.24){\makebox(0,0)[cc]{$\bullet$}} 
\put(50.00,121.24){\makebox(0,0)[cc]{$\bullet$}} 
\put(70.00,121.24){\makebox(0,0)[cc]{$\bullet$}} 
\put(90.00,121.24){\makebox(0,0)[cc]{$\bullet$}} 
\put(110.00,121.24){\makebox(0,0)[cc]{$\bullet$}} 
\put(130.00,121.24){\makebox(0,0)[cc]{$\bullet$}} 
\put(150.00,121.24){\makebox(0,0)[cc]{$\bullet$}} 
\put(170.00,121.24){\makebox(0,0)[cc]{$\bullet$}} 
\put(190.00,121.24){\makebox(0,0)[cc]{$\bullet$}} 
\put(00.00,138.56){\makebox(0,0)[cc]{$\bullet$}} 
\put(20.00,138.56){\makebox(0,0)[cc]{$\bullet$}} 
\put(40.00,138.56){\makebox(0,0)[cc]{$\bullet$}} 
\put(60.00,138.56){\makebox(0,0)[cc]{$\bullet$}} 
\put(80.00,138.56){\makebox(0,0)[cc]{$\bullet$}} 
\put(100.00,138.56){\makebox(0,0)[cc]{$\bullet$}} 
\put(120.00,138.56){\makebox(0,0)[cc]{$\bullet$}} 
\put(140.00,138.56){\makebox(0,0)[cc]{$\bullet$}} 
\put(160.00,138.56){\makebox(0,0)[cc]{$\bullet$}} 
\put(180.00,138.56){\makebox(0,0)[cc]{$\bullet$}} 
\put(200.00,138.56){\makebox(0,0)[cc]{$\bullet$}} 
%%%%%%%%%%%%%%%%%%%%%%%%%%%%%%%%%%%%%%%%%%%%%%%%%
\put(130.00,05.77){\makebox(0,0)[cc]{$*$}}
\put(110.00,05.77){\makebox(0,0)[cc]{$*$}}
\put(120.00,11.44){\makebox(0,0)[cc]{$*$}}
\put(120.00,22.88){\makebox(0,0)[cc]{$*$}}
\put(100.00,11.44){\makebox(0,0)[cc]{$*$}}
\put(100.00,22.88){\makebox(0,0)[cc]{$*$}}
\put(100.00,45.72){\makebox(0,0)[cc]{$*$}}
\put(100.00,57.16){\makebox(0,0)[cc]{$*$}}
\put(110.00,28.76){\makebox(0,0)[cc]{$*$}}
\put(110.00,40.20){\makebox(0,0)[cc]{$*$}}
\put(90.00,28.76){\makebox(0,0)[cc]{$*$}}
\put(90.00,40.20){\makebox(0,0)[cc]{$*$}}
\put(90.00,74.52){\makebox(0,0)[cc]{$*$}}
\put(90.00,63.08){\makebox(0,0)[cc]{$*$}}
\put(80.00,45.72){\makebox(0,0)[cc]{$*$}}
\put(80.00,57.16){\makebox(0,0)[cc]{$*$}}
\put(80.00,80.04){\makebox(0,0)[cc]{$*$}}
\put(80.00,91.48){\makebox(0,0)[cc]{$*$}}
\put(70.00,74.52){\makebox(0,0)[cc]{$*$}}
\put(70.00,63.08){\makebox(0,0)[cc]{$*$}}
\put(70.00,97.40){\makebox(0,0)[cc]{$*$}}
\put(70.00,108.84){\makebox(0,0)[cc]{$*$}}
\put(60.00,80.04){\makebox(0,0)[cc]{$*$}}
\put(60.00,91.48){\makebox(0,0)[cc]{$*$}}
\put(60.00,114.36){\makebox(0,0)[cc]{$*$}}
\put(60.00,125.80){\makebox(0,0)[cc]{$*$}}
\put(50.00,97.40){\makebox(0,0)[cc]{$*$}}
\put(50.00,108.84){\makebox(0,0)[cc]{$*$}}
\put(50.00,131.72){\makebox(0,0)[cc]{$*$}}
\put(40.00,114.36){\makebox(0,0)[cc]{$*$}}
\put(40.00,125.80){\makebox(0,0)[cc]{$*$}}
\put(30.00,131.72){\makebox(0,0)[cc]{$*$}}
\end{picture}
}
\caption{Triangular tiling, vertices, and tiling lines; 
triangles marked with $*$ form a tiling strip of type $2$}\label{fign1} 
\end{figure}
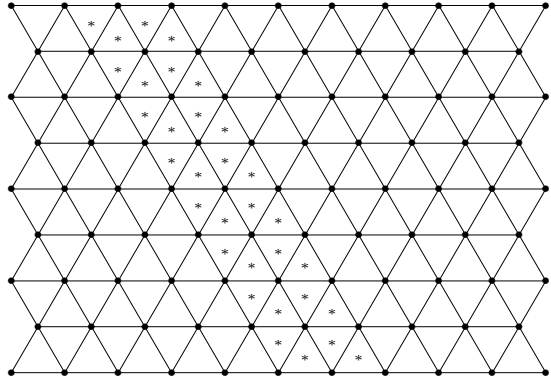

\subsection{T-hexagons and their h-envelopes}\label{s55.2}

For $i=1,2,3$, a {\em tiling strip of type $i$} is the region between two tiling lines of type $i$,
including these lines (see Figure~\ref{fign1} for an example of a tiling strip of type $2$).
In particular, if these two lines coincide, then the corresponding tiling strip coincides with 
each of these  tiling lines. A {\em t-hexagon} is, by definition, the intersection of three tiling 
strips, one for each type. Note that a t-hexagon can be:
\begin{itemize}
\item empty;
\item equal to a vertex of the tiling;
\item equal to a bounded line segment of a tiling line;
\item a polygon with three, four, five or six vertices.
\end{itemize}
By the {\em perimeter} of a t-hexagon we mean its perimeter as a polygon. Clearly, each t-hexagon has
finite perimeter. The perimeter of a vertex is zero,
while the perimeter of a bounded line segment is twice the length of this line segment.

The group of symmetries of the triangular tiling is the triangle group 
\begin{displaymath}
\Delta(3,3,3)=\langle a,b,c\,:\, a^2=b^2=c^2=(ab)^3=(bc)^3=(ca)^3=1\rangle
\end{displaymath}
generated by reflections with respect to the sides of the fundamental region of the tiling.
Two t-hexagons which can be obtained from each other applying some element in $\Delta(3,3,3)$ will be called
{\em isomorphic}. For $n\in\mathbb{Z}_{\geq 0}$, we denote by $T_n$ the number of isomorphism classes of
t-hexagons with perimeter $2n$.

Centroids of tiling triangles form a dual {\em hexagonal tiling} of our plane. Given a
t-hexagon $H$, its {\em hexagonal envelope} $E(H)$ is the union of all hexagons in the hexagonal tiling which 
intersect $H$, see Figure~\ref{fign2} for an example of a t-hexagon (bold lines) and its
hexagonal envelope (dotted lines).

\begin{figure}
{\fontsize{5pt}{1pt}
\begin{picture}(200.00,150.00)
%%%%%%%%%%%%%%%%%%%%%%%%%%%%%%%%%%%%
\thinlines
\drawline(00.00,00.00)(200.00,00.00)
\drawline(00.00,34.64)(200.00,34.64)
\drawline(00.00,69.28)(200.00,69.28)
\drawline(00.00,103.92)(200.00,103.92)
\drawline(00.00,138.56)(200.00,138.56)
\drawline(10.00,17.32)(190.00,17.32)
\drawline(10.00,51.96)(190.00,51.96)
\drawline(10.00,86.60)(190.00,86.60)
\drawline(10.00,121.24)(190.00,121.24)
\drawline(00.00,34.64)(20.00,00.00)
\drawline(00.00,69.28)(40.00,00.00)
\drawline(00.00,103.92)(60.00,00.00)
\drawline(00.00,138.56)(80.00,00.00)
\drawline(20.00,138.56)(100.00,00.00)
\drawline(40.00,138.56)(120.00,00.00)
\drawline(60.00,138.56)(140.00,00.00)
\drawline(80.00,138.56)(160.00,00.00)
\drawline(100.00,138.56)(180.00,00.00)
\drawline(120.00,138.56)(200.00,00.00)
\drawline(140.00,138.56)(200.00,34.64)
\drawline(160.00,138.56)(200.00,69.28)
\drawline(180.00,138.56)(200.00,103.92)
\drawline(180.00,00.00)(200.00,34.64)
\drawline(160.00,00.00)(200.00,69.28)
\drawline(140.00,00.00)(200.00,103.92)
\drawline(120.00,00.00)(200.00,138.56)
\drawline(100.00,00.00)(180.00,138.56)
\drawline(80.00,00.00)(160.00,138.56)
\drawline(60.00,00.00)(140.00,138.56)
\drawline(40.00,00.00)(120.00,138.56)
\drawline(20.00,00.00)(100.00,138.56)
\drawline(00.00,00.00)(80.00,138.56)
\drawline(00.00,34.64)(60.00,138.56)
\drawline(00.00,69.28)(40.00,138.56)
\drawline(00.00,103.92)(20.00,138.56)
%%%%%%%%%%%%%%%%%%%%%%%%%%%%%%%%%%%%%
\put(00.00,00.00){\makebox(0,0)[cc]{$\bullet$}} 
\put(20.00,00.00){\makebox(0,0)[cc]{$\bullet$}} 
\put(40.00,00.00){\makebox(0,0)[cc]{$\bullet$}} 
\put(60.00,00.00){\makebox(0,0)[cc]{$\bullet$}} 
\put(80.00,00.00){\makebox(0,0)[cc]{$\bullet$}} 
\put(100.00,00.00){\makebox(0,0)[cc]{$\bullet$}} 
\put(120.00,00.00){\makebox(0,0)[cc]{$\bullet$}} 
\put(140.00,00.00){\makebox(0,0)[cc]{$\bullet$}} 
\put(160.00,00.00){\makebox(0,0)[cc]{$\bullet$}} 
\put(180.00,00.00){\makebox(0,0)[cc]{$\bullet$}} 
\put(200.00,00.00){\makebox(0,0)[cc]{$\bullet$}} 
\put(10.00,17.32){\makebox(0,0)[cc]{$\bullet$}} 
\put(30.00,17.32){\makebox(0,0)[cc]{$\bullet$}} 
\put(50.00,17.32){\makebox(0,0)[cc]{$\bullet$}} 
\put(70.00,17.32){\makebox(0,0)[cc]{$\bullet$}} 
\put(90.00,17.32){\makebox(0,0)[cc]{$\bullet$}} 
\put(110.00,17.32){\makebox(0,0)[cc]{$\bullet$}} 
\put(130.00,17.32){\makebox(0,0)[cc]{$\bullet$}} 
\put(150.00,17.32){\makebox(0,0)[cc]{$\bullet$}} 
\put(170.00,17.32){\makebox(0,0)[cc]{$\bullet$}} 
\put(190.00,17.32){\makebox(0,0)[cc]{$\bullet$}} 
\put(00.00,34.64){\makebox(0,0)[cc]{$\bullet$}} 
\put(20.00,34.64){\makebox(0,0)[cc]{$\bullet$}} 
\put(40.00,34.64){\makebox(0,0)[cc]{$\bullet$}} 
\put(60.00,34.64){\makebox(0,0)[cc]{$\bullet$}} 
\put(80.00,34.64){\makebox(0,0)[cc]{$\bullet$}} 
\put(100.00,34.64){\makebox(0,0)[cc]{$\bullet$}} 
\put(120.00,34.64){\makebox(0,0)[cc]{$\bullet$}} 
\put(140.00,34.64){\makebox(0,0)[cc]{$\bullet$}} 
\put(160.00,34.64){\makebox(0,0)[cc]{$\bullet$}} 
\put(180.00,34.64){\makebox(0,0)[cc]{$\bullet$}} 
\put(200.00,34.64){\makebox(0,0)[cc]{$\bullet$}} 
\put(10.00,51.96){\makebox(0,0)[cc]{$\bullet$}} 
\put(30.00,51.96){\makebox(0,0)[cc]{$\bullet$}} 
\put(50.00,51.96){\makebox(0,0)[cc]{$\bullet$}} 
\put(70.00,51.96){\makebox(0,0)[cc]{$\bullet$}} 
\put(90.00,51.96){\makebox(0,0)[cc]{$\bullet$}} 
\put(110.00,51.96){\makebox(0,0)[cc]{$\bullet$}} 
\put(130.00,51.96){\makebox(0,0)[cc]{$\bullet$}} 
\put(150.00,51.96){\makebox(0,0)[cc]{$\bullet$}} 
\put(170.00,51.96){\makebox(0,0)[cc]{$\bullet$}} 
\put(190.00,51.96){\makebox(0,0)[cc]{$\bullet$}} 
\put(00.00,69.28){\makebox(0,0)[cc]{$\bullet$}} 
\put(20.00,69.28){\makebox(0,0)[cc]{$\bullet$}} 
\put(40.00,69.28){\makebox(0,0)[cc]{$\bullet$}} 
\put(60.00,69.28){\makebox(0,0)[cc]{$\bullet$}} 
\put(80.00,69.28){\makebox(0,0)[cc]{$\bullet$}} 
\put(100.00,69.28){\makebox(0,0)[cc]{$\bullet$}} 
\put(120.00,69.28){\makebox(0,0)[cc]{$\bullet$}} 
\put(140.00,69.28){\makebox(0,0)[cc]{$\bullet$}} 
\put(160.00,69.28){\makebox(0,0)[cc]{$\bullet$}} 
\put(180.00,69.28){\makebox(0,0)[cc]{$\bullet$}} 
\put(200.00,69.28){\makebox(0,0)[cc]{$\bullet$}} 
\put(10.00,86.60){\makebox(0,0)[cc]{$\bullet$}} 
\put(30.00,86.60){\makebox(0,0)[cc]{$\bullet$}} 
\put(50.00,86.60){\makebox(0,0)[cc]{$\bullet$}} 
\put(70.00,86.60){\makebox(0,0)[cc]{$\bullet$}} 
\put(90.00,86.60){\makebox(0,0)[cc]{$\bullet$}} 
\put(110.00,86.60){\makebox(0,0)[cc]{$\bullet$}} 
\put(130.00,86.60){\makebox(0,0)[cc]{$\bullet$}} 
\put(150.00,86.60){\makebox(0,0)[cc]{$\bullet$}} 
\put(170.00,86.60){\makebox(0,0)[cc]{$\bullet$}} 
\put(190.00,86.60){\makebox(0,0)[cc]{$\bullet$}} 
\put(00.00,103.92){\makebox(0,0)[cc]{$\bullet$}} 
\put(20.00,103.92){\makebox(0,0)[cc]{$\bullet$}} 
\put(40.00,103.92){\makebox(0,0)[cc]{$\bullet$}} 
\put(60.00,103.92){\makebox(0,0)[cc]{$\bullet$}} 
\put(80.00,103.92){\makebox(0,0)[cc]{$\bullet$}} 
\put(100.00,103.92){\makebox(0,0)[cc]{$\bullet$}} 
\put(120.00,103.92){\makebox(0,0)[cc]{$\bullet$}} 
\put(140.00,103.92){\makebox(0,0)[cc]{$\bullet$}} 
\put(160.00,103.92){\makebox(0,0)[cc]{$\bullet$}} 
\put(180.00,103.92){\makebox(0,0)[cc]{$\bullet$}} 
\put(200.00,103.92){\makebox(0,0)[cc]{$\bullet$}} 
\put(10.00,121.24){\makebox(0,0)[cc]{$\bullet$}} 
\put(30.00,121.24){\makebox(0,0)[cc]{$\bullet$}} 
\put(50.00,121.24){\makebox(0,0)[cc]{$\bullet$}} 
\put(70.00,121.24){\makebox(0,0)[cc]{$\bullet$}} 
\put(90.00,121.24){\makebox(0,0)[cc]{$\bullet$}} 
\put(110.00,121.24){\makebox(0,0)[cc]{$\bullet$}} 
\put(130.00,121.24){\makebox(0,0)[cc]{$\bullet$}} 
\put(150.00,121.24){\makebox(0,0)[cc]{$\bullet$}} 
\put(170.00,121.24){\makebox(0,0)[cc]{$\bullet$}} 
\put(190.00,121.24){\makebox(0,0)[cc]{$\bullet$}} 
\put(00.00,138.56){\makebox(0,0)[cc]{$\bullet$}} 
\put(20.00,138.56){\makebox(0,0)[cc]{$\bullet$}} 
\put(40.00,138.56){\makebox(0,0)[cc]{$\bullet$}} 
\put(60.00,138.56){\makebox(0,0)[cc]{$\bullet$}} 
\put(80.00,138.56){\makebox(0,0)[cc]{$\bullet$}} 
\put(100.00,138.56){\makebox(0,0)[cc]{$\bullet$}} 
\put(120.00,138.56){\makebox(0,0)[cc]{$\bullet$}} 
\put(140.00,138.56){\makebox(0,0)[cc]{$\bullet$}} 
\put(160.00,138.56){\makebox(0,0)[cc]{$\bullet$}} 
\put(180.00,138.56){\makebox(0,0)[cc]{$\bullet$}} 
\put(200.00,138.56){\makebox(0,0)[cc]{$\bullet$}} 
%%%%%%%%%%%%%%%%%%%%%%%%%%%%%%%%%%%%%%%%%%%%%%%%%
\linethickness{2pt}
\drawline(80.00,34.64)(140.00,34.64)
\drawline(150.00,51.96)(140.00,34.64)
\drawline(150.00,51.96)(140.00,69.28)
\drawline(140.00,69.28)(100.00,69.28)
\drawline(80.00,34.64)(100.00,69.28)
%%%%%%%%%%%%%%%%%%%%%%%%%%%%%%%%%%%%
\linethickness{1pt}
\dottedline(80.00,22.88)(90.00,28.76){1}
\dottedline(100.00,22.88)(90.00,28.76){1}
\dottedline(100.00,22.88)(110.00,28.76){1}
\dottedline(120.00,22.88)(110.00,28.76){1}
\dottedline(120.00,22.88)(130.00,28.76){1}
\dottedline(140.00,22.88)(130.00,28.76){1}
\dottedline(140.00,22.88)(150.00,28.76){1}
\dottedline(150.00,40.20)(150.00,28.76){1}
\dottedline(150.00,40.20)(160.00,45.72){1}
\dottedline(160.00,57.16)(160.00,45.72){1}
\dottedline(160.00,57.16)(150.00,63.08){1}
\dottedline(150.00,74.52)(150.00,63.08){1}
\dottedline(150.00,74.52)(140.00,80.04){1}
\dottedline(130.00,74.52)(140.00,80.04){1}
\dottedline(130.00,74.52)(120.00,80.04){1}
\dottedline(110.00,74.52)(120.00,80.04){1}
\dottedline(110.00,74.52)(100.00,80.04){1}
\dottedline(90.00,74.52)(100.00,80.04){1}
\dottedline(90.00,74.52)(90.00,63.08){1}
\dottedline(80.00,57.16)(90.00,63.08){1}
\dottedline(80.00,57.16)(80.00,45.72){1}
\dottedline(70.00,40.20)(80.00,45.72){1}
\dottedline(70.00,40.20)(70.00,28.76){1}
\dottedline(80.00,22.88)(70.00,28.76){1}
\end{picture}
}
\caption{A t-hexagon and its hexagonal envelope}\label{fign2} 
\end{figure}
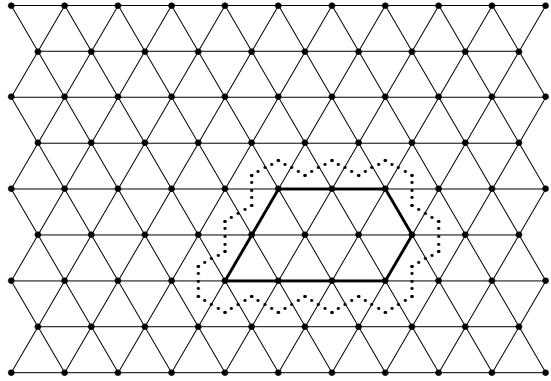

\begin{lemma}\label{lemn77}
Let $H$ be a t-hexagon of perimeter $i$ for some $i\in\mathbb{Z}_{\geq 0}$. Then the
hexagonal envelope of $H$ has $6+2i$ vertices.
\end{lemma}

\begin{proof}
For $i=0,1,2,3,4,5$ the statement of the lemma follows by inspecting all t-hexagons 
of perimeter $i$. These t-hexagons are given in the following list:
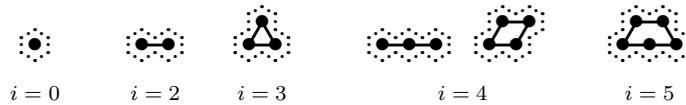
\begin{center}
\begin{picture}(260.00,45.00)
%%%%%%%%%%%%%%%%%%%%%%%%%%%%%%%%%%%%
\put(10.00,30.00){\makebox(0,0)[cc]{$\bullet$}}
\put(10.00,12.00){\makebox(0,0)[cc]{\tiny$i=0$}}
\put(50.00,30.00){\makebox(0,0)[cc]{$\bullet$}}
\put(60.00,30.00){\makebox(0,0)[cc]{$\bullet$}}
\put(55.00,12.00){\makebox(0,0)[cc]{\tiny$i=2$}}
\put(90.00,30.00){\makebox(0,0)[cc]{$\bullet$}}
\put(100.00,30.00){\makebox(0,0)[cc]{$\bullet$}}
\put(95.00,38.66){\makebox(0,0)[cc]{$\bullet$}}
\put(95.00,12.00){\makebox(0,0)[cc]{\tiny$i=3$}}
\put(140.00,30.00){\makebox(0,0)[cc]{$\bullet$}}
\put(150.00,30.00){\makebox(0,0)[cc]{$\bullet$}}
\put(160.00,30.00){\makebox(0,0)[cc]{$\bullet$}}
\put(180.00,30.00){\makebox(0,0)[cc]{$\bullet$}}
\put(190.00,30.00){\makebox(0,0)[cc]{$\bullet$}}
\put(195.00,38.66){\makebox(0,0)[cc]{$\bullet$}}
\put(185.00,38.66){\makebox(0,0)[cc]{$\bullet$}}
\put(170.00,12.00){\makebox(0,0)[cc]{\tiny$i=4$}}
\put(230.00,30.00){\makebox(0,0)[cc]{$\bullet$}}
\put(240.00,30.00){\makebox(0,0)[cc]{$\bullet$}}
\put(250.00,30.00){\makebox(0,0)[cc]{$\bullet$}}
\put(235.00,38.66){\makebox(0,0)[cc]{$\bullet$}}
\put(245.00,38.66){\makebox(0,0)[cc]{$\bullet$}}
\put(240.00,12.00){\makebox(0,0)[cc]{\tiny$i=5$}}
%%%%%%%%%%%%%%%%%%%%%%%%%%%%%%%%%%%%
\linethickness{1pt}
\drawline(50.00,30.00)(60.00,30.00)
\drawline(90.00,30.00)(100.00,30.00)
\drawline(90.00,30.00)(95.00,38.66)
\drawline(100.00,30.00)(95.00,38.66)
\drawline(140.00,30.00)(160.00,30.00)
\drawline(180.00,30.00)(190.00,30.00)
\drawline(180.00,30.00)(185.00,38.66)
\drawline(185.00,38.66)(195.00,38.66)
\drawline(190.00,30.00)(195.00,38.66)
\drawline(230.00,30.00)(250.00,30.00)
\drawline(235.00,38.66)(245.00,38.66)
\drawline(250.00,30.00)(245.00,38.66)
\drawline(235.00,38.66)(230.00,30.00)
%%%%%%%%%%%%%%%%%%%%%%%%%%%%%%%%%%%%%%%%%%%%%%
\thinlines
\dottedline(10.00,35.77)(14.99,32.88){1}
\dottedline(14.99,27.12)(14.99,32.88){1}
\dottedline(14.99,27.12)(10.00,24.23){1}
\dottedline(05.01,27.12)(10.00,24.23){1}
\dottedline(05.01,27.12)(05.01,32.88){1}
\dottedline(10.00,35.77)(05.01,32.88){1}
\dottedline(50.00,35.77)(54.99,32.88){1}
\dottedline(54.99,27.12)(50.00,24.23){1}
\dottedline(45.01,27.12)(50.00,24.23){1}
\dottedline(45.01,27.12)(45.01,32.88){1}
\dottedline(50.00,35.77)(45.01,32.88){1}
\dottedline(60.00,35.77)(64.99,32.88){1}
\dottedline(64.99,27.12)(64.99,32.88){1}
\dottedline(64.99,27.12)(60.00,24.23){1}
\dottedline(55.01,27.12)(60.00,24.23){1}
\dottedline(60.00,35.77)(55.01,32.88){1}
\dottedline(94.99,27.12)(90.00,24.23){1}
\dottedline(85.01,27.12)(90.00,24.23){1}
\dottedline(85.01,27.12)(85.01,32.88){1}
\dottedline(90.00,35.77)(85.01,32.88){1}
\dottedline(90.00,35.77)(90.00,41.54){1}
\dottedline(100.00,35.77)(100.00,41.54){1}
\dottedline(95.00,44.42)(100.00,41.54){1}
\dottedline(95.00,44.42)(90.00,41.54){1}
\dottedline(100.00,35.77)(104.99,32.88){1}
\dottedline(104.99,27.12)(104.99,32.88){1}
\dottedline(104.99,27.12)(100.00,24.23){1}
\dottedline(95.01,27.12)(100.00,24.23){1}
\dottedline(140.00,35.77)(144.99,32.88){1}
\dottedline(144.99,27.12)(140.00,24.23){1}
\dottedline(135.01,27.12)(140.00,24.23){1}
\dottedline(135.01,27.12)(135.01,32.88){1}
\dottedline(140.00,35.77)(135.01,32.88){1}
\dottedline(150.00,35.77)(154.99,32.88){1}
\dottedline(160.00,35.77)(164.99,32.88){1}
\dottedline(164.99,27.12)(164.99,32.88){1}
\dottedline(154.99,27.12)(150.00,24.23){1}
\dottedline(145.01,27.12)(150.00,24.23){1}
\dottedline(150.00,35.77)(145.01,32.88){1}
\dottedline(164.99,27.12)(160.00,24.23){1}
\dottedline(155.01,27.12)(160.00,24.23){1}
\dottedline(160.00,35.77)(155.01,32.88){1}
\dottedline(250.00,35.77)(254.99,32.88){1}
\dottedline(254.99,27.12)(254.99,32.88){1}
\dottedline(244.99,27.12)(240.00,24.23){1}
\dottedline(235.01,27.12)(240.00,24.23){1}
\dottedline(254.99,27.12)(250.00,24.23){1}
\dottedline(245.01,27.12)(250.00,24.23){1}
\dottedline(234.99,27.12)(230.00,24.23){1}
\dottedline(225.01,27.12)(230.00,24.23){1}
\dottedline(225.01,27.12)(225.01,32.88){1}
\dottedline(230.00,35.77)(225.01,32.88){1}
\dottedline(230.00,35.77)(230.00,41.54){1}
\dottedline(234.99,44.42)(230.00,41.54){1}
\dottedline(235.01,44.42)(240.00,41.54){1}
\dottedline(244.99,44.42)(240.00,41.54){1}
\dottedline(245.01,44.42)(250.00,41.54){1}
\dottedline(250.00,35.77)(250.00,41.54){1}
\dottedline(184.99,27.12)(180.00,24.23){1}
\dottedline(175.01,27.12)(180.00,24.23){1}
\dottedline(175.01,27.12)(175.01,32.88){1}
\dottedline(180.00,35.77)(175.01,32.88){1}
\dottedline(180.00,35.77)(180.00,41.54){1}
\dottedline(185.00,44.42)(180.00,41.54){1}
\dottedline(185.00,44.42)(190.00,41.54){1}
\dottedline(195.00,44.42)(190.00,41.54){1}
\dottedline(195.00,44.42)(200.00,41.54){1}
\dottedline(200.00,35.77)(200.00,41.54){1}
\dottedline(200.00,35.77)(194.99,32.88){1}
\dottedline(194.99,27.12)(194.99,32.88){1}
\dottedline(194.99,27.12)(190.00,24.23){1}
\dottedline(185.01,27.12)(190.00,24.23){1}
\end{picture}
\end{center}

We claim that the rest follows by induction on $i$.
Indeed, assume that $H$ is the intersection of three tiling strips (one for each type).
We can, in turn, pull the lines defining these strips closer to each other, one step at a time.
Eventually by one such step we will get a smaller t-hexagon $H'$. There are four possible cases.

{\bf Case~1.} The t-hexagon $H'$ is obtained from $H$ by collapsing a line segment to a vertex
as illustrated here:

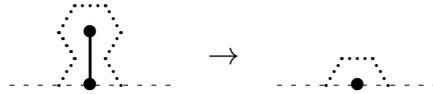
\begin{center}
\begin{picture}(200.00,50.00)
%%%%%%%%%%%%%%%%%%%%%%%%%%%%%%%%%%%%
\dashline(20.00,10.00)(80.00,10.00){2}
\dashline(120.00,10.00)(180.00,10.00){2}
%%%%%%%%%%%%%%%%%%%%%%%%%%%%%%%%%%%%
\put(50.00,10.00){\makebox(0,0)[cc]{$\bullet$}} 
\put(50.00,30.00){\makebox(0,0)[cc]{$\bullet$}} 
\put(150.00,10.00){\makebox(0,0)[cc]{$\bullet$}} 
\put(100.00,20.00){\makebox(0,0)[cc]{$\rightarrow$}} 
%%%%%%%%%%%%%%%%%%%%%%%%%%%%%%%%%%%%
\linethickness{1pt}
\drawline(50.00,10.00)(50.00,30.00)
%%%%%%%%%%%%%%%%%%%%%%%%%%%%%%%%%%%%
\linethickness{1pt}
\dottedline(44.33,20.00)(38.46,10.00){1}
\dottedline(44.33,20.00)(38.46,30.00){1}
\dottedline(44.33,20.00)(38.46,30.00){1}
\dottedline(44.33,40.00)(38.46,30.00){1}
\dottedline(44.33,40.00)(55.77,40.00){1}
\dottedline(61.54,30.00)(55.77,40.00){1}
\dottedline(61.54,30.00)(55.77,20.00){1}
\dottedline(61.54,10.00)(55.77,20.00){1}
\dottedline(144.33,20.00)(138.46,10.00){1}
\dottedline(161.54,10.00)(155.77,20.00){1}
\dottedline(144.33,20.00)(155.77,20.00){1}
\end{picture}
\end{center}

In this case we see that the perimeter of $H$ decreases by $2$ and the number of vertices of the hexagonal envelope 
decreases  by $4$.

{\bf Case~2.} The t-hexagon $H'$ is obtained from $H$ by collapsing a trapezoid segment onto its basis
as illustrated here (the length of the segment can be arbitrary):

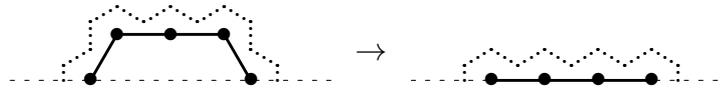
\begin{center}
\begin{picture}(300.00,50.00)
%%%%%%%%%%%%%%%%%%%%%%%%%%%%%%%%%%%%
\dashline(20.00,10.00)(140.00,10.00){2}
\dashline(170.00,10.00)(290.00,10.00){2}
%%%%%%%%%%%%%%%%%%%%%%%%%%%%%%%%%%%%
\put(50.00,10.00){\makebox(0,0)[cc]{$\bullet$}}  
\put(110.00,10.00){\makebox(0,0)[cc]{$\bullet$}} 
\put(60.00,27.32){\makebox(0,0)[cc]{$\bullet$}} 
\put(80.00,27.32){\makebox(0,0)[cc]{$\bullet$}} 
\put(100.00,27.32){\makebox(0,0)[cc]{$\bullet$}} 
\put(200.00,10.00){\makebox(0,0)[cc]{$\bullet$}} 
\put(220.00,10.00){\makebox(0,0)[cc]{$\bullet$}} 
\put(240.00,10.00){\makebox(0,0)[cc]{$\bullet$}} 
\put(260.00,10.00){\makebox(0,0)[cc]{$\bullet$}} 
\put(155.00,20.00){\makebox(0,0)[cc]{$\rightarrow$}} 
%%%%%%%%%%%%%%%%%%%%%%%%%%%%%%%%%%%%
\linethickness{1pt}
\drawline(50.00,10.00)(60.00,27.32)
\drawline(100.00,27.32)(60.00,27.32)
\drawline(100.00,27.32)(110.00,10.00)
\drawline(200.00,10.00)(260.00,10.00)
%%%%%%%%%%%%%%%%%%%%%%%%%%%%%%%%%%%%
\linethickness{1pt}
\dottedline(40.00,10.00)(40.00,15.77){1}
\dottedline(50.00,21.54)(40.00,15.77){1}
\dottedline(50.00,21.54)(50.00,32.08){1}
\dottedline(60.00,37.85)(50.00,32.08){1}
\dottedline(60.00,37.85)(70.00,32.08){1}
\dottedline(80.00,37.85)(70.00,32.08){1}
\dottedline(80.00,37.85)(90.00,32.08){1}
\dottedline(100.00,37.85)(90.00,32.08){1}
\dottedline(100.00,37.85)(110.00,32.08){1}
\dottedline(110.00,21.54)(110.00,32.08){1}
\dottedline(110.00,21.54)(120.00,15.77){1}
\dottedline(120.00,10.00)(120.00,15.77){1}
\dottedline(190.00,10.00)(190.00,15.77){1}
\dottedline(200.00,21.54)(190.00,15.77){1}
\dottedline(210.00,15.77)(200.00,21.54){1}
\dottedline(220.00,21.54)(210.00,15.77){1}
\dottedline(230.00,15.77)(220.00,21.54){1}
\dottedline(240.00,21.54)(230.00,15.77){1}
\dottedline(250.00,15.77)(240.00,21.54){1}
\dottedline(260.00,21.54)(250.00,15.77){1}
\dottedline(270.00,15.77)(260.00,21.54){1}
\dottedline(270.00,15.77)(270.00,10.00){1}
\end{picture}
\end{center}

In this case we see that the perimeter of $H$ decreases by $1$ and the number of vertices of the hexagonal envelope 
by $2$.

{\bf Case~3.} The t-hexagon $H'$ is obtained from $H$ by collapsing a trapezoid segment onto its basis
as illustrated here (the length of the segment can be arbitrary):

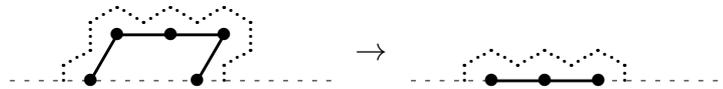
\begin{center}
\begin{picture}(300.00,50.00)
%%%%%%%%%%%%%%%%%%%%%%%%%%%%%%%%%%%%
\dashline(20.00,10.00)(140.00,10.00){2}
\dashline(170.00,10.00)(290.00,10.00){2}
%%%%%%%%%%%%%%%%%%%%%%%%%%%%%%%%%%%%
\put(50.00,10.00){\makebox(0,0)[cc]{$\bullet$}}  
\put(90.00,10.00){\makebox(0,0)[cc]{$\bullet$}} 
\put(60.00,27.32){\makebox(0,0)[cc]{$\bullet$}} 
\put(80.00,27.32){\makebox(0,0)[cc]{$\bullet$}} 
\put(100.00,27.32){\makebox(0,0)[cc]{$\bullet$}} 
\put(200.00,10.00){\makebox(0,0)[cc]{$\bullet$}} 
\put(220.00,10.00){\makebox(0,0)[cc]{$\bullet$}} 
\put(240.00,10.00){\makebox(0,0)[cc]{$\bullet$}}  
\put(155.00,20.00){\makebox(0,0)[cc]{$\rightarrow$}} 
%%%%%%%%%%%%%%%%%%%%%%%%%%%%%%%%%%%%
\linethickness{1pt}
\drawline(50.00,10.00)(60.00,27.32)
\drawline(100.00,27.32)(60.00,27.32)
\drawline(100.00,27.32)(90.00,10.00)
\drawline(200.00,10.00)(240.00,10.00)
%%%%%%%%%%%%%%%%%%%%%%%%%%%%%%%%%%%%
\linethickness{1pt}
\dottedline(40.00,10.00)(40.00,15.77){1}
\dottedline(50.00,21.54)(40.00,15.77){1}
\dottedline(50.00,21.54)(50.00,32.08){1}
\dottedline(60.00,37.85)(50.00,32.08){1}
\dottedline(60.00,37.85)(70.00,32.08){1}
\dottedline(80.00,37.85)(70.00,32.08){1}
\dottedline(80.00,37.85)(90.00,32.08){1}
\dottedline(100.00,37.85)(90.00,32.08){1}
\dottedline(100.00,37.85)(110.00,32.08){1}
\dottedline(110.00,21.54)(110.00,32.08){1}
\dottedline(110.00,21.54)(100.00,15.77){1}
\dottedline(100.00,10.00)(100.00,15.77){1}
\dottedline(190.00,10.00)(190.00,15.77){1}
\dottedline(200.00,21.54)(190.00,15.77){1}
\dottedline(210.00,15.77)(200.00,21.54){1}
\dottedline(220.00,21.54)(210.00,15.77){1}
\dottedline(230.00,15.77)(220.00,21.54){1}
\dottedline(240.00,21.54)(230.00,15.77){1}
\dottedline(250.00,15.77)(240.00,21.54){1}
\dottedline(250.00,15.77)(250.00,10.00){1}
\end{picture}
\end{center}

In this case we see that the perimeter of $H$ decreases by $2$ and the number of vertices of the hexagonal envelope 
by $4$.

{\bf Case~4.} The t-hexagon $H'$ is obtained from $H$ by collapsing a trapezoid segment to its basis
as illustrated here:

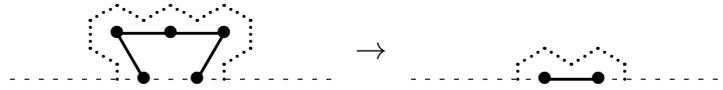
\begin{center}
\begin{picture}(300.00,50.00)
%%%%%%%%%%%%%%%%%%%%%%%%%%%%%%%%%%%%
\dashline(20.00,10.00)(140.00,10.00){2}
\dashline(170.00,10.00)(290.00,10.00){2}
%%%%%%%%%%%%%%%%%%%%%%%%%%%%%%%%%%%%
\put(70.00,10.00){\makebox(0,0)[cc]{$\bullet$}}  
\put(90.00,10.00){\makebox(0,0)[cc]{$\bullet$}} 
\put(60.00,27.32){\makebox(0,0)[cc]{$\bullet$}} 
\put(80.00,27.32){\makebox(0,0)[cc]{$\bullet$}} 
\put(100.00,27.32){\makebox(0,0)[cc]{$\bullet$}} 
\put(220.00,10.00){\makebox(0,0)[cc]{$\bullet$}} 
\put(240.00,10.00){\makebox(0,0)[cc]{$\bullet$}}  
\put(155.00,20.00){\makebox(0,0)[cc]{$\rightarrow$}} 
%%%%%%%%%%%%%%%%%%%%%%%%%%%%%%%%%%%%
\linethickness{1pt}
\drawline(70.00,10.00)(60.00,27.32)
\drawline(100.00,27.32)(60.00,27.32)
\drawline(100.00,27.32)(90.00,10.00)
\drawline(220.00,10.00)(240.00,10.00)
%%%%%%%%%%%%%%%%%%%%%%%%%%%%%%%%%%%%
\linethickness{1pt}
\dottedline(60.00,10.00)(60.00,15.77){1}
\dottedline(50.00,21.54)(60.00,15.77){1}
\dottedline(50.00,21.54)(50.00,32.08){1}
\dottedline(60.00,37.85)(50.00,32.08){1}
\dottedline(60.00,37.85)(70.00,32.08){1}
\dottedline(80.00,37.85)(70.00,32.08){1}
\dottedline(80.00,37.85)(90.00,32.08){1}
\dottedline(100.00,37.85)(90.00,32.08){1}
\dottedline(100.00,37.85)(110.00,32.08){1}
\dottedline(110.00,21.54)(110.00,32.08){1}
\dottedline(110.00,21.54)(100.00,15.77){1}
\dottedline(100.00,10.00)(100.00,15.77){1}
\dottedline(210.00,10.00)(210.00,15.77){1}
\dottedline(220.00,21.54)(210.00,15.77){1}
\dottedline(230.00,15.77)(220.00,21.54){1}
\dottedline(240.00,21.54)(230.00,15.77){1}
\dottedline(250.00,15.77)(240.00,21.54){1}
\dottedline(250.00,15.77)(250.00,10.00){1}
\end{picture}
\end{center}

In this case we see that the perimeter of $H$ decreases by $3$ and the number of vertices of the hexagonal envelope 
by $6$. 

Since all the above changes agree, by linearity, with the desired formula, the claim of the lemma follows by induction.
\end{proof}

Hexagonal envelopes of t-hexagons seem to be exactly the {\em hollow hexagons} considered in \cite{CBCBB,CBC}
(the latter papers do not really have any mathematically precise definition of hollow hexagons).

\subsection{Characters of t-hexagons}\label{s55.3}

We would like to encode t-hexagons using vectors with non-negative integral coordinates. 
For this we will need some notation. Denote by $v_1$, $v_2$ and $v_3$ the vectors in the
Euclidean plane as shown in Figure~\ref{fign3}. Note that all these vectors have length one
and that $v_1+v_2+v_3=0$. In Figure~\ref{fign3} we also see a $*$-marked t-hexagon which is
the intersection of the tiling strips formed by thick lines. The tiling lines which bound
the tiling strips are marked by numbers  $1,2,3,4,5,6$ which correspond to going along the
perimeter of the hexagon starting from the bottom side and going into the clockwise direction.

\begin{figure}
{\fontsize{5pt}{1pt}
\begin{picture}(210.00,150.00)
%%%%%%%%%%%%%%%%%%%%%%%%%%%%%%%%%%%%
\thinlines
\drawline(00.00,00.00)(200.00,00.00)
\drawline(00.00,34.64)(200.00,34.64)
\drawline(00.00,69.28)(200.00,69.28)
\drawline(00.00,103.92)(200.00,103.92)
\drawline(00.00,138.56)(200.00,138.56)
\drawline(10.00,17.32)(190.00,17.32)
\drawline(10.00,51.96)(190.00,51.96)
\drawline(10.00,86.60)(190.00,86.60)
\drawline(10.00,121.24)(190.00,121.24)
\drawline(00.00,34.64)(20.00,00.00)
\drawline(00.00,69.28)(40.00,00.00)
\drawline(00.00,103.92)(60.00,00.00)
\drawline(00.00,138.56)(80.00,00.00)
\drawline(20.00,138.56)(100.00,00.00)
\drawline(40.00,138.56)(120.00,00.00)
\drawline(60.00,138.56)(140.00,00.00)
\drawline(80.00,138.56)(160.00,00.00)
\drawline(100.00,138.56)(180.00,00.00)
\drawline(120.00,138.56)(200.00,00.00)
\drawline(140.00,138.56)(200.00,34.64)
\drawline(160.00,138.56)(200.00,69.28)
\drawline(180.00,138.56)(200.00,103.92)
\drawline(180.00,00.00)(200.00,34.64)
\drawline(160.00,00.00)(200.00,69.28)
\drawline(140.00,00.00)(200.00,103.92)
\drawline(120.00,00.00)(200.00,138.56)
\drawline(100.00,00.00)(180.00,138.56)
\drawline(80.00,00.00)(160.00,138.56)
\drawline(60.00,00.00)(140.00,138.56)
\drawline(40.00,00.00)(120.00,138.56)
\drawline(20.00,00.00)(100.00,138.56)
\drawline(00.00,00.00)(80.00,138.56)
\drawline(00.00,34.64)(60.00,138.56)
\drawline(00.00,69.28)(40.00,138.56)
\drawline(00.00,103.92)(20.00,138.56)
%%%%%%%%%%%%%%%%%%%%%%%%%%%%%%%%%%%%%
\put(00.00,00.00){\makebox(0,0)[cc]{$\bullet$}} 
\put(20.00,00.00){\makebox(0,0)[cc]{$\bullet$}} 
\put(40.00,00.00){\makebox(0,0)[cc]{$\bullet$}} 
\put(60.00,00.00){\makebox(0,0)[cc]{$\bullet$}} 
\put(80.00,00.00){\makebox(0,0)[cc]{$\bullet$}} 
\put(100.00,00.00){\makebox(0,0)[cc]{$\bullet$}} 
\put(120.00,00.00){\makebox(0,0)[cc]{$\bullet$}} 
\put(140.00,00.00){\makebox(0,0)[cc]{$\bullet$}} 
\put(160.00,00.00){\makebox(0,0)[cc]{$\bullet$}} 
\put(180.00,00.00){\makebox(0,0)[cc]{$\bullet$}} 
\put(200.00,00.00){\makebox(0,0)[cc]{$\bullet$}} 
\put(10.00,17.32){\makebox(0,0)[cc]{$\bullet$}} 
\put(30.00,17.32){\makebox(0,0)[cc]{$\bullet$}} 
\put(50.00,17.32){\makebox(0,0)[cc]{$\bullet$}} 
\put(70.00,17.32){\makebox(0,0)[cc]{$\bullet$}} 
\put(90.00,17.32){\makebox(0,0)[cc]{$\bullet$}} 
\put(110.00,17.32){\makebox(0,0)[cc]{$\bullet$}} 
\put(130.00,17.32){\makebox(0,0)[cc]{$\bullet$}} 
\put(150.00,17.32){\makebox(0,0)[cc]{$\bullet$}} 
\put(170.00,17.32){\makebox(0,0)[cc]{$\bullet$}} 
\put(190.00,17.32){\makebox(0,0)[cc]{$\bullet$}} 
\put(00.00,34.64){\makebox(0,0)[cc]{$\bullet$}} 
\put(20.00,34.64){\makebox(0,0)[cc]{$\bullet$}} 
\put(40.00,34.64){\makebox(0,0)[cc]{$\bullet$}} 
\put(60.00,34.64){\makebox(0,0)[cc]{$\bullet$}} 
\put(80.00,34.64){\makebox(0,0)[cc]{$\bullet$}} 
\put(100.00,34.64){\makebox(0,0)[cc]{$\bullet$}} 
\put(120.00,34.64){\makebox(0,0)[cc]{$\bullet$}} 
\put(140.00,34.64){\makebox(0,0)[cc]{$\bullet$}} 
\put(160.00,34.64){\makebox(0,0)[cc]{$\bullet$}} 
\put(180.00,34.64){\makebox(0,0)[cc]{$\bullet$}} 
\put(200.00,34.64){\makebox(0,0)[cc]{$\bullet$}} 
\put(10.00,51.96){\makebox(0,0)[cc]{$\bullet$}} 
\put(30.00,51.96){\makebox(0,0)[cc]{$\bullet$}} 
\put(50.00,51.96){\makebox(0,0)[cc]{$\bullet$}} 
\put(70.00,51.96){\makebox(0,0)[cc]{$\bullet$}} 
\put(90.00,51.96){\makebox(0,0)[cc]{$\bullet$}} 
\put(110.00,51.96){\makebox(0,0)[cc]{$\bullet$}} 
\put(130.00,51.96){\makebox(0,0)[cc]{$\bullet$}} 
\put(150.00,51.96){\makebox(0,0)[cc]{$\bullet$}} 
\put(170.00,51.96){\makebox(0,0)[cc]{$\bullet$}} 
\put(190.00,51.96){\makebox(0,0)[cc]{$\bullet$}} 
\put(00.00,69.28){\makebox(0,0)[cc]{$\bullet$}} 
\put(20.00,69.28){\makebox(0,0)[cc]{$\bullet$}} 
\put(40.00,69.28){\makebox(0,0)[cc]{$\bullet$}} 
\put(60.00,69.28){\makebox(0,0)[cc]{$\bullet$}} 
\put(80.00,69.28){\makebox(0,0)[cc]{$\bullet$}} 
\put(100.00,69.28){\makebox(0,0)[cc]{$\bullet$}} 
\put(120.00,69.28){\makebox(0,0)[cc]{$\bullet$}} 
\put(140.00,69.28){\makebox(0,0)[cc]{$\bullet$}} 
\put(160.00,69.28){\makebox(0,0)[cc]{$\bullet$}} 
\put(180.00,69.28){\makebox(0,0)[cc]{$\bullet$}} 
\put(200.00,69.28){\makebox(0,0)[cc]{$\bullet$}} 
\put(10.00,86.60){\makebox(0,0)[cc]{$\bullet$}} 
\put(30.00,86.60){\makebox(0,0)[cc]{$\bullet$}} 
\put(50.00,86.60){\makebox(0,0)[cc]{$\bullet$}} 
\put(70.00,86.60){\makebox(0,0)[cc]{$\bullet$}} 
\put(90.00,86.60){\makebox(0,0)[cc]{$\bullet$}} 
\put(110.00,86.60){\makebox(0,0)[cc]{$\bullet$}} 
\put(130.00,86.60){\makebox(0,0)[cc]{$\bullet$}} 
\put(150.00,86.60){\makebox(0,0)[cc]{$\bullet$}} 
\put(170.00,86.60){\makebox(0,0)[cc]{$\bullet$}} 
\put(190.00,86.60){\makebox(0,0)[cc]{$\bullet$}} 
\put(00.00,103.92){\makebox(0,0)[cc]{$\bullet$}} 
\put(20.00,103.92){\makebox(0,0)[cc]{$\bullet$}} 
\put(40.00,103.92){\makebox(0,0)[cc]{$\bullet$}} 
\put(60.00,103.92){\makebox(0,0)[cc]{$\bullet$}} 
\put(80.00,103.92){\makebox(0,0)[cc]{$\bullet$}} 
\put(100.00,103.92){\makebox(0,0)[cc]{$\bullet$}} 
\put(120.00,103.92){\makebox(0,0)[cc]{$\bullet$}} 
\put(140.00,103.92){\makebox(0,0)[cc]{$\bullet$}} 
\put(160.00,103.92){\makebox(0,0)[cc]{$\bullet$}} 
\put(180.00,103.92){\makebox(0,0)[cc]{$\bullet$}} 
\put(200.00,103.92){\makebox(0,0)[cc]{$\bullet$}} 
\put(10.00,121.24){\makebox(0,0)[cc]{$\bullet$}} 
\put(30.00,121.24){\makebox(0,0)[cc]{$\bullet$}} 
\put(50.00,121.24){\makebox(0,0)[cc]{$\bullet$}} 
\put(70.00,121.24){\makebox(0,0)[cc]{$\bullet$}} 
\put(90.00,121.24){\makebox(0,0)[cc]{$\bullet$}} 
\put(110.00,121.24){\makebox(0,0)[cc]{$\bullet$}} 
\put(130.00,121.24){\makebox(0,0)[cc]{$\bullet$}} 
\put(150.00,121.24){\makebox(0,0)[cc]{$\bullet$}} 
\put(170.00,121.24){\makebox(0,0)[cc]{$\bullet$}} 
\put(190.00,121.24){\makebox(0,0)[cc]{$\bullet$}} 
\put(00.00,138.56){\makebox(0,0)[cc]{$\bullet$}} 
\put(20.00,138.56){\makebox(0,0)[cc]{$\bullet$}} 
\put(40.00,138.56){\makebox(0,0)[cc]{$\bullet$}} 
\put(60.00,138.56){\makebox(0,0)[cc]{$\bullet$}} 
\put(80.00,138.56){\makebox(0,0)[cc]{$\bullet$}} 
\put(100.00,138.56){\makebox(0,0)[cc]{$\bullet$}} 
\put(120.00,138.56){\makebox(0,0)[cc]{$\bullet$}} 
\put(140.00,138.56){\makebox(0,0)[cc]{$\bullet$}} 
\put(160.00,138.56){\makebox(0,0)[cc]{$\bullet$}} 
\put(180.00,138.56){\makebox(0,0)[cc]{$\bullet$}} 
\put(200.00,138.56){\makebox(0,0)[cc]{$\bullet$}} 
%%%%%%%%%%%%%%%%%%%%%%%%%%%%%%%%%%%%%%%%%%%%%%%%%
\linethickness{2pt}
\drawline(50.00,17.32)(30.00,17.32)
\drawline(34.00,18.32)(30.00,17.32)
\drawline(34.00,16.32)(30.00,17.32)
\drawline(40.00,34.64)(30.00,17.32)
\drawline(40.00,34.64)(37.00,31.70)
\drawline(40.00,34.64)(38.70,30.50)
\drawline(40.00,34.64)(50.00,17.32)
\drawline(47.50,19.90)(50.00,17.32)
\drawline(48.90,20.90)(50.00,17.32)
\put(40.00,12.90){\makebox(0,0)[cc]{$v_1$}} 
\put(31.00,28.00){\makebox(0,0)[cc]{$v_2$}} 
\put(49.00,28.00){\makebox(0,0)[cc]{$v_3$}} 
%%%%%%%%%%%%%%%%%%%%%%%%%%%%%%%%%%%%%%%%%%%%%%%%%
\linethickness{2pt}
\drawline(00.00,69.28)(200.00,69.28)
\drawline(00.00,103.92)(200.00,103.92)
\drawline(160.00,138.56)(80.00,00.00)
\drawline(200.00,138.56)(120.00,00.00)
\drawline(100.00,138.56)(180.00,00.00)
\drawline(140.00,138.56)(200.00,34.64)
\put(205.00,69.28){\makebox(0,0)[cc]{$1$}} 
\put(205.00,103.92){\makebox(0,0)[cc]{$4$}} 
\put(100.00,145.00){\makebox(0,0)[cc]{$2$}} 
\put(160.00,145.00){\makebox(0,0)[cc]{$3$}} 
\put(140.00,145.00){\makebox(0,0)[cc]{$5$}} 
\put(200.00,145.00){\makebox(0,0)[cc]{$6$}} 
\put(150.00,75.05){\makebox(0,0)[cc]{$*$}} 
\put(150.00,98.14){\makebox(0,0)[cc]{$*$}} 
\put(140.00,80.82){\makebox(0,0)[cc]{$*$}} 
\put(160.00,80.82){\makebox(0,0)[cc]{$*$}} 
\put(140.00,92.36){\makebox(0,0)[cc]{$*$}} 
\put(160.00,92.36){\makebox(0,0)[cc]{$*$}} 
\end{picture}
}
\caption{Basic vectors and lines}\label{fign3} 
\end{figure}
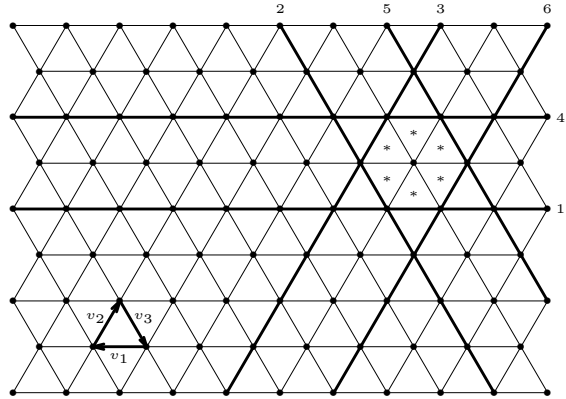

Let now $H$ be a nonempty t-hexagon given as the intersection of three tiling strips, one for each type.
Without loss of generality we may assume that each tiling line which bounds each of these tiling
strips has a non-empty intersection with $H$. We number the tiling lines forming the boundaries 
of the tiling strips in the same way as in Figure~\ref{fign3}. Note that, if two tiling lines 
coincide, we still count them as two different lines in our numbering. This corresponds to 
walking along the boundary of $H$, starting with the bottom side, first along $v_1$, then along
$-v_3$, then along $v_2$, then along $-v_1$, then along $v_3$ and, finally, along $-v_2$.

The intersection of a boundary tiling line of a tiling strip with $H$ is then either a vertex 
or a side of $H$. We denote by $\chi(H)$ the vector $(a_1,a_2,a_3,a_4,a_5,a_6)$ where for 
$i=1,2,3,4,5,6$ the number $a_i$ is the length of the intersection of the line $i$ with $H$.
For example, for the $*$-marked t-hexagon in Figure~\ref{fign3} we have $\chi(H)=(1,1,1,1,1,1)$,
while for the thick t-hexagon in Figure~\ref{fign2} we have $\chi(H)=(3,0,2,2,1,1)$.
The vector $\chi(H)$ will be called the {\em character} of $H$.

One could also give a description of $\chi(H)$ as follows: Start with the rightmost vertex on the 
bottom edge of $H$. Walk along $v_1$ until the next vertex (which might coincide with the starting one).
The number $a_1$ is the length of this walk. Continue along $-v_3$ to record $a_2$, then along $v_2$ 
to record $a_3$ and so on in the order described above.

The original action of $\Delta(3,3,3)$ induces an action on the set of characters of t-hexagons which
is generated by the cyclic permutations of components of the character and the flip
\begin{displaymath}
(a_1,a_2,a_3,a_4,a_5,a_6)\mapsto (a_6,a_5,a_4,a_3,a_2,a_1). 
\end{displaymath}
Using this action, we can change $H$ to an isomorphic t-hexagon $H'$ such that we have 
$\chi(H')=(a_1,a_2,a_3,a_4,a_5,a_6)$ where the following conditions are satisfied:
\begin{equation}\label{eqn1}
a_1+a_3+a_5\leq a_2+a_4+a_6\quad\text{ and }\quad a_1\geq a_3\geq a_5. 
\end{equation}
Such $H'$ as well as its character will be called {\em distinguished}.
It is easy to see that a distinguished representative in the isomorphism class of $H$ is unique up to shift of tiling.
As an example, the regular hexagon in Figure~\ref{fign3} is distinguished, while the t-hexagon in
Figure~\ref{fign3} is not distinguished since the first inequality in \eqref{eqn1} fails.

Note that our walk along the perimeter of $H$ always returns to the original point. 
From this it follows that $(a_1,a_2,a_3,a_4,a_5,a_6)\in\mathbb{Z}_{\geq 0}^6$ is the character of
some t-hexagon if and only if 
\begin{equation}\label{eqn2}
(a_1-a_4)v_1+(a_5-a_2)v_3+(a_3-a_6)v_2=0.
\end{equation}
Taking into account $v_3=-v_1-v_2$ and linear independence of $v_1$ and $v_2$, Equation~\eqref{eqn2} is
equivalent to 
\begin{equation}\label{eqn3}
a_1+a_2-a_4-a_5=0 \quad\text{ and }\quad  a_2+a_3-a_5-a_6=0.
\end{equation}

\begin{lemma}\label{lemn17}
Let $H$ be a distinguished t-hexagon and $\chi(H)=(a_1,a_2,a_3,a_4,a_5,a_6)$. Then we have
\begin{displaymath}
a_1\leq a_4, \quad a_5\leq a_2\quad\text{ and }\quad  a_3\leq a_6.
\end{displaymath} 
\end{lemma}

\begin{proof}
From Equation~\eqref{eqn3} we have $a_1-a_4=a_5-a_2=a_3-a_6$. Plugging 
$a_1-a_4=a_5-a_2$ into the first inequality in \eqref{eqn1} implies $\quad  a_3\leq a_6$.
Plugging  $a_5-a_2=a_3-a_6$ into the first inequality in \eqref{eqn1} implies $a_1\leq a_4$.
Plugging  $a_1-a_4=a_3-a_6$ into the first inequality in \eqref{eqn1} implies $a_5\leq a_2$.
\end{proof}

\subsection{Elementary operations on distinguished t-hexagons}\label{s55.4}

Let $H$ be a distinguished t-hexagon and $\chi(H)=(a_1,a_2,a_3,a_4,a_5,a_6)$. 
We consider four {\em elementary} operations on hexagons.

{\bf Operation $\Phi$.} Assume $a_1-a_3\geq 2$. Then, using Lemma~\ref{lemn17}, 
it is easy to check that the vector 
\begin{displaymath}
(a_1-1,a_2,a_3+1,a_4-1,a_5,a_6+1) 
\end{displaymath}
satisfies all conditions in \eqref{eqn1} and \eqref{eqn3} and hence is the character of a unique
distinguished t-hexagon which we denote by $\Phi(H)$. Here is an example of this operation:

\begin{center}
\begin{picture}(250.00,70.00)
%%%%%%%%%%%%%%%%%%%%%%%%%%%%%%%%%%%%
\put(20.00,10.00){\makebox(0,0)[cc]{$\bullet$}}  
\put(40.00,10.00){\makebox(0,0)[cc]{$\bullet$}}  
\put(60.00,10.00){\makebox(0,0)[cc]{$\bullet$}}  
\put(80.00,10.00){\makebox(0,0)[cc]{$\bullet$}}  
\put(10.00,27.32){\makebox(0,0)[cc]{$\bullet$}}  
\put(90.00,27.32){\makebox(0,0)[cc]{$\bullet$}}  
\put(20.00,44.64){\makebox(0,0)[cc]{$\bullet$}}  
\put(40.00,44.64){\makebox(0,0)[cc]{$\bullet$}}  
\put(60.00,44.64){\makebox(0,0)[cc]{$\bullet$}}  
\put(80.00,44.64){\makebox(0,0)[cc]{$\bullet$}}  
\put(160.00,10.00){\makebox(0,0)[cc]{$\bullet$}}  
\put(180.00,10.00){\makebox(0,0)[cc]{$\bullet$}}  
\put(200.00,10.00){\makebox(0,0)[cc]{$\bullet$}}  
\put(150.00,27.32){\makebox(0,0)[cc]{$\bullet$}}  
\put(210.00,27.32){\makebox(0,0)[cc]{$\bullet$}}  
\put(160.00,44.64){\makebox(0,0)[cc]{$\bullet$}}  
\put(220.00,44.64){\makebox(0,0)[cc]{$\bullet$}}  
\put(170.00,61.96){\makebox(0,0)[cc]{$\bullet$}}  
\put(190.00,61.96){\makebox(0,0)[cc]{$\bullet$}}  
\put(210.00,61.96){\makebox(0,0)[cc]{$\bullet$}}  
\put(120.00,30.00){\makebox(0,0)[cc]{$\overset{\Phi}{\rightarrow}$}} 
%%%%%%%%%%%%%%%%%%%%%%%%%%%%%%%%%%%%
\linethickness{1pt}
\drawline(20.00,10.00)(80.00,10.00)
\drawline(90.00,27.32)(80.00,10.00)
\drawline(90.00,27.32)(80.00,44.64)
\drawline(20.00,44.64)(80.00,44.64)
\drawline(20.00,44.64)(10.00,27.32)
\drawline(20.00,10.00)(10.00,27.32)
\drawline(160.00,10.00)(200.00,10.00)
\drawline(220.00,44.64)(200.00,10.00)
\drawline(220.00,44.64)(210.00,61.96)
\drawline(170.00,61.96)(210.00,61.96)
\drawline(170.00,61.96)(150.00,27.32)
\drawline(160.00,10.00)(150.00,27.32)
%%%%%%%%%%%%%%%%%%%%%%%%%%%%%%%%%%%%
\linethickness{1pt}
\dottedline(20.00,44.64)(40.00,10.00){1}
\dottedline(40.00,44.64)(60.00,10.00){1}
\dottedline(60.00,44.64)(80.00,10.00){1}
\dottedline(40.00,44.64)(20.00,10.00){1}
\dottedline(60.00,44.64)(40.00,10.00){1}
\dottedline(80.00,44.64)(60.00,10.00){1}
\dottedline(10.00,27.32)(90.00,27.32){1}
\dottedline(160.00,10.00)(190.00,61.96){1}
\dottedline(180.00,10.00)(210.00,61.96){1}
\dottedline(210.00,27.32)(190.00,61.96){1}
\dottedline(200.00,10.00)(170.00,61.96){1}
\dottedline(210.00,27.32)(150.00,27.32){1}
\dottedline(160.00,44.64)(220.00,44.64){1}
\dottedline(180.00,10.00)(160.00,44.64){1}
\end{picture}
\end{center}

{\bf Operation $\Psi$.} Assume $a_1>a_3>a_5$. Then, using Lemma~\ref{lemn17}, it is easy to check that the vector 
\begin{displaymath}
(a_1-1,a_2+1,a_3,a_4-1,a_5+1,a_6) 
\end{displaymath}
satisfies all conditions in \eqref{eqn1} and \eqref{eqn3} and hence is the character of a unique
distinguished t-hexagon which we denote by $\Psi(H)$. 
Here is an example of this operation:

\begin{center}
\begin{picture}(250.00,80.00)
%%%%%%%%%%%%%%%%%%%%%%%%%%%%%%%%%%%%
\put(20.00,10.00){\makebox(0,0)[cc]{$\bullet$}}  
\put(40.00,10.00){\makebox(0,0)[cc]{$\bullet$}}  
\put(60.00,10.00){\makebox(0,0)[cc]{$\bullet$}}  
\put(80.00,10.00){\makebox(0,0)[cc]{$\bullet$}}  
\put(10.00,27.32){\makebox(0,0)[cc]{$\bullet$}} 
\put(90.00,27.32){\makebox(0,0)[cc]{$\bullet$}}  
\put(20.00,44.64){\makebox(0,0)[cc]{$\bullet$}}  
\put(100.00,44.64){\makebox(0,0)[cc]{$\bullet$}}  
\put(30.00,61.96){\makebox(0,0)[cc]{$\bullet$}}  
\put(50.00,61.96){\makebox(0,0)[cc]{$\bullet$}}  
\put(70.00,61.96){\makebox(0,0)[cc]{$\bullet$}}  
\put(90.00,61.96){\makebox(0,0)[cc]{$\bullet$}}  
\put(160.00,44.64){\makebox(0,0)[cc]{$\bullet$}}  
\put(170.00,27.32){\makebox(0,0)[cc]{$\bullet$}}  
\put(170.00,61.96){\makebox(0,0)[cc]{$\bullet$}}  
\put(180.00,10.00){\makebox(0,0)[cc]{$\bullet$}}  
\put(180.00,79.28){\makebox(0,0)[cc]{$\bullet$}}  
\put(200.00,10.00){\makebox(0,0)[cc]{$\bullet$}}  
\put(200.00,79.28){\makebox(0,0)[cc]{$\bullet$}}  
\put(220.00,10.00){\makebox(0,0)[cc]{$\bullet$}}  
\put(220.00,79.28){\makebox(0,0)[cc]{$\bullet$}}  
\put(230.00,27.32){\makebox(0,0)[cc]{$\bullet$}}  
\put(230.00,61.96){\makebox(0,0)[cc]{$\bullet$}}  
\put(240.00,44.64){\makebox(0,0)[cc]{$\bullet$}}  
\put(130.00,50.00){\makebox(0,0)[cc]{$\overset{\Psi}{\rightarrow}$}} 
%%%%%%%%%%%%%%%%%%%%%%%%%%%%%%%%%%%%
\linethickness{1pt}
\drawline(20.00,10.00)(80.00,10.00)
\drawline(100.00,44.64)(80.00,10.00)
\drawline(100.00,44.64)(90.00,61.96)
\drawline(30.00,61.96)(90.00,61.96)
\drawline(30.00,61.96)(10.00,27.32)
\drawline(20.00,10.00)(10.00,27.32)
\drawline(160.00,44.64)(180.00,10.00)
\drawline(220.00,10.00)(180.00,10.00)
\drawline(220.00,10.00)(240.00,44.64)
\drawline(220.00,79.28)(240.00,44.64)
\drawline(220.00,79.28)(180.00,79.28)
\drawline(160.00,44.64)(180.00,79.28)
%%%%%%%%%%%%%%%%%%%%%%%%%%%%%%%%%%%%
\linethickness{1pt}
\dottedline(20.00,44.64)(40.00,10.00){1}
\dottedline(30.00,61.96)(60.00,10.00){1}
\dottedline(50.00,61.96)(80.00,10.00){1}
\dottedline(70.00,61.96)(90.00,27.32){1}
\dottedline(50.00,61.96)(20.00,10.00){1}
\dottedline(70.00,61.96)(40.00,10.00){1}
\dottedline(90.00,61.96)(60.00,10.00){1}
\dottedline(10.00,27.32)(90.00,27.32){1}
\dottedline(20.00,44.64)(100.00,44.64){1}
\dottedline(170.00,27.32)(200.00,79.28){1}
\dottedline(180.00,10.00)(220.00,79.28){1}
\dottedline(200.00,10.00)(230.00,61.96){1}
\dottedline(220.00,10.00)(180.00,79.28){1}
\dottedline(200.00,10.00)(170.00,61.96){1}
\dottedline(230.00,27.32)(170.00,27.32){1}
\dottedline(230.00,61.96)(170.00,61.96){1}
\dottedline(160.00,44.64)(240.00,44.64){1}
\dottedline(230.00,27.32)(200.00,79.28){1}
\end{picture}
\end{center}

{\bf Operation $\Theta$.} Assume $a_5>0$. Then, using Lemma~\ref{lemn17}, it is easy to check that the vector 
\begin{displaymath}
(a_1-1,a_2+1,a_3-1,a_4+1,a_5-1,a_6+1) 
\end{displaymath}
satisfies all conditions in \eqref{eqn1} and \eqref{eqn3} and hence is the character of a unique
distinguished t-hexagon which we denote by $\Theta(H)$. 
Here is an example of this operation:

\begin{center}
\begin{picture}(250.00,80.00)
%%%%%%%%%%%%%%%%%%%%%%%%%%%%%%%%%%%%  
\put(10.00,44.64){\makebox(0,0)[cc]{$\bullet$}}  
\put(20.00,27.32){\makebox(0,0)[cc]{$\bullet$}}  
\put(20.00,61.96){\makebox(0,0)[cc]{$\bullet$}}  
\put(30.00,10.00){\makebox(0,0)[cc]{$\bullet$}}  
\put(30.00,79.28){\makebox(0,0)[cc]{$\bullet$}}  
\put(50.00,10.00){\makebox(0,0)[cc]{$\bullet$}}  
\put(50.00,79.28){\makebox(0,0)[cc]{$\bullet$}}  
\put(70.00,10.00){\makebox(0,0)[cc]{$\bullet$}}  
\put(70.00,79.28){\makebox(0,0)[cc]{$\bullet$}}  
\put(80.00,27.32){\makebox(0,0)[cc]{$\bullet$}}  
\put(80.00,61.96){\makebox(0,0)[cc]{$\bullet$}}  
\put(90.00,44.64){\makebox(0,0)[cc]{$\bullet$}}  
\put(150.00,61.96){\makebox(0,0)[cc]{$\bullet$}}  
\put(160.00,79.28){\makebox(0,0)[cc]{$\bullet$}}  
\put(160.00,44.64){\makebox(0,0)[cc]{$\bullet$}}  
\put(170.00,27.32){\makebox(0,0)[cc]{$\bullet$}}  
\put(180.00,79.28){\makebox(0,0)[cc]{$\bullet$}}  
\put(180.00,10.00){\makebox(0,0)[cc]{$\bullet$}}  
\put(200.00,79.28){\makebox(0,0)[cc]{$\bullet$}}  
\put(200.00,10.00){\makebox(0,0)[cc]{$\bullet$}}  
\put(210.00,27.32){\makebox(0,0)[cc]{$\bullet$}}  
\put(220.00,79.28){\makebox(0,0)[cc]{$\bullet$}}  
\put(220.00,44.64){\makebox(0,0)[cc]{$\bullet$}}  
\put(230.00,61.96){\makebox(0,0)[cc]{$\bullet$}}  
\put(120.00,50.00){\makebox(0,0)[cc]{$\overset{\Theta}{\rightarrow}$}} 
%%%%%%%%%%%%%%%%%%%%%%%%%%%%%%%%%%%%
\linethickness{1pt}
\drawline(10.00,44.64)(30.00,10.00)
\drawline(70.00,10.00)(30.00,10.00)
\drawline(70.00,10.00)(90.00,44.64)
\drawline(70.00,79.28)(90.00,44.64)
\drawline(70.00,79.28)(30.00,79.28)
\drawline(10.00,44.64)(30.00,79.28)
\drawline(180.00,10.00)(200.00,10.00)
\drawline(230.00,61.96)(200.00,10.00)
\drawline(230.00,61.96)(220.00,79.28)
\drawline(160.00,79.28)(220.00,79.28)
\drawline(160.00,79.28)(150.00,61.96)
\drawline(180.00,10.00)(150.00,61.96)
%%%%%%%%%%%%%%%%%%%%%%%%%%%%%%%%%%%%
\linethickness{1pt}
\dottedline(20.00,27.32)(50.00,79.28){1}
\dottedline(30.00,10.00)(70.00,79.28){1}
\dottedline(50.00,10.00)(80.00,61.96){1}
\dottedline(70.00,10.00)(30.00,79.28){1}
\dottedline(50.00,10.00)(20.00,61.96){1}
\dottedline(80.00,27.32)(20.00,27.32){1}
\dottedline(80.00,61.96)(20.00,61.96){1}
\dottedline(10.00,44.64)(90.00,44.64){1}
\dottedline(80.00,27.32)(50.00,79.28){1}
\dottedline(160.00,79.28)(200.00,10.00){1}
\dottedline(180.00,79.28)(210.00,27.32){1}
\dottedline(200.00,79.28)(220.00,44.64){1}
\dottedline(180.00,79.28)(160.00,44.64){1}
\dottedline(200.00,79.28)(170.00,27.32){1}
\dottedline(220.00,79.28)(180.00,10.00){1}
\dottedline(230.00,61.96)(150.00,61.96){1}
\dottedline(220.00,44.64)(160.00,44.64){1}
\dottedline(210.00,27.32)(170.00,27.32){1}
\end{picture}
\end{center}

{\bf Operation $\Lambda$.} Assume $a_3-a_5\geq 2$. Then, using Lemma~\ref{lemn17}, it is easy to check that the vector 
\begin{displaymath}
(a_1-1,a_2+2,a_3-2,a_4+1,a_5,a_6) 
\end{displaymath}
satisfies all conditions in \eqref{eqn1} and \eqref{eqn3} and hence is the character of a unique
distinguished t-hexagon which we denote by $\Lambda(H)$. 
Here is an example of this operation:

\begin{center}
\begin{picture}(300.00,80.00)
%%%%%%%%%%%%%%%%%%%%%%%%%%%%%%%%%%%%  
\put(10.00,44.64){\makebox(0,0)[cc]{$\bullet$}}  
\put(20.00,27.32){\makebox(0,0)[cc]{$\bullet$}}  
\put(20.00,61.96){\makebox(0,0)[cc]{$\bullet$}}  
\put(30.00,10.00){\makebox(0,0)[cc]{$\bullet$}}  
\put(30.00,79.28){\makebox(0,0)[cc]{$\bullet$}}  
\put(50.00,10.00){\makebox(0,0)[cc]{$\bullet$}}  
\put(50.00,79.28){\makebox(0,0)[cc]{$\bullet$}}  
\put(70.00,10.00){\makebox(0,0)[cc]{$\bullet$}}  
\put(70.00,79.28){\makebox(0,0)[cc]{$\bullet$}}  
\put(80.00,27.32){\makebox(0,0)[cc]{$\bullet$}}  
\put(90.00,79.28){\makebox(0,0)[cc]{$\bullet$}}    
\put(90.00,44.64){\makebox(0,0)[cc]{$\bullet$}}  
\put(100.00,61.96){\makebox(0,0)[cc]{$\bullet$}}  
\put(110.00,79.28){\makebox(0,0)[cc]{$\bullet$}}  
\put(170.00,79.28){\makebox(0,0)[cc]{$\bullet$}}  
\put(180.00,61.96){\makebox(0,0)[cc]{$\bullet$}}  
\put(190.00,79.28){\makebox(0,0)[cc]{$\bullet$}}  
\put(190.00,44.64){\makebox(0,0)[cc]{$\bullet$}}  
\put(200.00,27.32){\makebox(0,0)[cc]{$\bullet$}}  
\put(210.00,10.00){\makebox(0,0)[cc]{$\bullet$}}  
\put(210.00,79.28){\makebox(0,0)[cc]{$\bullet$}}  
\put(230.00,10.00){\makebox(0,0)[cc]{$\bullet$}}  
\put(230.00,79.28){\makebox(0,0)[cc]{$\bullet$}}  
\put(240.00,27.32){\makebox(0,0)[cc]{$\bullet$}}  
\put(250.00,44.64){\makebox(0,0)[cc]{$\bullet$}}  
\put(250.00,79.28){\makebox(0,0)[cc]{$\bullet$}}  
\put(260.00,61.96){\makebox(0,0)[cc]{$\bullet$}}  
\put(270.00,79.28){\makebox(0,0)[cc]{$\bullet$}}  
\put(140.00,50.00){\makebox(0,0)[cc]{$\overset{\Lambda}{\rightarrow}$}} 
%%%%%%%%%%%%%%%%%%%%%%%%%%%%%%%%%%%%
\linethickness{1pt}
\drawline(10.00,44.64)(30.00,10.00)
\drawline(70.00,10.00)(30.00,10.00)
\drawline(70.00,10.00)(110.00,79.28)
\drawline(110.00,79.28)(30.00,79.28)
\drawline(10.00,44.64)(30.00,79.28)
\drawline(210.00,10.00)(230.00,10.00)
\drawline(270.00,79.28)(230.00,10.00)
\drawline(270.00,79.28)(170.00,79.28)
\drawline(210.00,10.00)(170.00,79.28)
%%%%%%%%%%%%%%%%%%%%%%%%%%%%%%%%%%%%
\linethickness{1pt}
\dottedline(20.00,27.32)(50.00,79.28){1}
\dottedline(30.00,10.00)(70.00,79.28){1}
\dottedline(50.00,10.00)(90.00,79.28){1}
\dottedline(70.00,10.00)(30.00,79.28){1}
\dottedline(50.00,10.00)(20.00,61.96){1}
\dottedline(80.00,27.32)(20.00,27.32){1}
\dottedline(100.00,61.96)(20.00,61.96){1}
\dottedline(10.00,44.64)(90.00,44.64){1}
\dottedline(80.00,27.32)(50.00,79.28){1}
\dottedline(90.00,79.28)(100.00,61.96){1}
\dottedline(70.00,79.28)(90.00,44.64){1}
\dottedline(230.00,10.00)(190.00,79.28){1}
\dottedline(240.00,27.32)(210.00,79.28){1}
\dottedline(250.00,44.64)(230.00,79.28){1}
\dottedline(260.00,61.96)(250.00,79.28){1}
\dottedline(180.00,61.96)(190.00,79.28){1}
\dottedline(190.00,44.64)(210.00,79.28){1}
\dottedline(200.00,27.32)(230.00,79.28){1}
\dottedline(210.00,10.00)(250.00,79.28){1}
\dottedline(200.00,27.32)(240.00,27.32){1}
\dottedline(190.00,44.64)(250.00,44.64){1}
\dottedline(180.00,61.96)(260.00,61.96){1}
\end{picture}
\end{center}

Directly from the definitions it is easy to see that all maps 
$\Phi$, $\Psi$, $\Theta$ and $\Lambda$ do not change the perimeter.
As illustrated by the examples,
all these maps have rather transparent geometric interpretations which could be obtained
by moving the boundary tiling lines of the tiling strips which define the original t-hexagon.

\subsection{Signature and defect}\label{s55.5}

Let $H$ be a distinguished t-hexagon. Assume that $\chi(H)=(a_1,a_2,a_3,a_4,a_5,a_6)$. 
Then the  vector $\mathrm{sign}(H):=(a_1-a_3,a_3-a_5,a_5)\in\mathbb{Z}_{\geq 0}^3$ will be called
the {\em signature} of $H$.
For example, the regular hexagon in Figure~\ref{fign3} has signature $(0,0,1)$ while a
distinguished t-hexagon isomorphic to the t-hexagon in Figure~\ref{fign2} has signature $(1,1,0)$.
Directly from the definitions one computes that for any distinguished t-hexagon $H$ we have:
\begin{equation}\label{eqn5}
\begin{array}{rcl}
\mathrm{sign}(\Phi(H))&=&\mathrm{sign}(H)+(-2,1,0),\\
\mathrm{sign}(\Psi(H))&=&\mathrm{sign}(H)+(-1,-1,1),\\
\mathrm{sign}(\Theta(H))&=&\mathrm{sign}(H)+(0,0,-1),\\
\mathrm{sign}(\Lambda(H))&=&\mathrm{sign}(H)+(1,-2,0),\\
\end{array}
\end{equation}
provided that the t-hexagons $\Phi(H)$, $\Psi(H)$, $\Theta(H)$ or, respectively, $\Lambda(H)$, are defined.

We define the {\em defect} of $H$ as
\begin{displaymath}
\mathrm{def}(H):=a_2+a_4+a_6-a_1-a_3-a_5 
\end{displaymath}
and note that the defect of a distinguished t-hexagon is always non-negative.

\subsection{The number of t-hexagons}\label{s55.9}

Our main result in this section is the following statement which gives a direct connection between
the present paper and \cite{CBCBB,CBC}.

\begin{theorem}\label{thmn72}
For all $n\in\mathbb{Z}_{\geq 0}$, mapping $H$ to $\mathrm{sign}(H)$ induces a bijection between the set of
isomorphism classes of distinguished t-hexagons of perimeter $2n$ and the set $I(n\mathbf{e}(1))$.
\end{theorem}

\begin{proof}
Consider a distinguished t-hexagon $Q$ having the character $(n,0,0,n,0,0)$. The signature of this t-hexagon is
$(n,0,0)=n\mathbf{e}(1)$. Applying, whenever possible, a sequence of operations 
$\Phi$, $\Psi$, $\Theta$ and $\Lambda$ to $Q$, produces a set of t-hexagons of perimeter $2n$.
From \eqref{eqn5} it follows that the set of signatures for all t-hexagons which can be obtained in this way is contained in $I(n\mathbf{e}(1))$. Observe that the operation $\Phi$ is defined
as soon as $a_1-a_3\geq 2$ and that this condition is equivalent to the fact that 
adding $(-2,1,0)$ to the signature of the input t-hexagon gives a vector with non-negative
coordinates. Put differently, if we have a t-hexagon such that the sum of its signature
and $(-2,1,0)$ has non-negative coordinates, then $\Phi$ can be applied to this t-hexagon.
Similar observations also apply to $\Psi$ and $(-1,-1,1)$,
to $\Theta$ and $(0,0,-1)$ and to $\Lambda$ and $(1,-2,0)$. 
Consequently, the set of signatures for all t-hexagons which can be obtained from $Q$
by all possible sequences of $\Phi$, $\Psi$, $\Theta$ and $\Lambda$ coincides with 
$I(n\mathbf{e}(1))$.

Our next step is to show that each distinguished t-hexagon of perimeter $2n$ can be obtained from
$Q$ using a sequence of operations of the form $\Phi$, $\Psi$, $\Theta$ and $\Lambda$
(in fact, the first three would suffice).
Let $K$ be a distinguished t-hexagon of perimeter $2n$ with character $(a_1,a_2,a_3,a_4,a_5,a_6)$. Assume $a_5>0$. Then, 
using Lemma~\ref{lemn17}, it is easy to check that the vector 
\begin{displaymath}
(a_1+1,a_2-1,a_3,a_4+1,a_5-1,a_6) 
\end{displaymath}
satisfies all conditions in \eqref{eqn1} and \eqref{eqn3} and hence is the character of a unique
distinguished t-hexagon. Therefore $K=\Psi(K')$ for some distinguished t-hexagon $K'$ and the character
of $K'$ has a smaller fifth coordinate. In particular, $K$ is obtained, using a sequence of $\Psi$'s, 
from some distinguished t-hexagon $K'$ the character of which has zero fifth coordinate.
 
Let $K$ be a distinguished t-hexagon of perimeter $2n$ with character $(a_1,a_2,a_3,a_4,a_5,a_6)$. Assume 
that $a_5=0$ and $a_3>0$. Then, 
using Lemma~\ref{lemn17}, it is easy to check that the vector 
\begin{displaymath}
(a_1+1,a_2,a_3-1,a_4+1,a_5,a_6-1) 
\end{displaymath}
satisfies all conditions in \eqref{eqn1} and \eqref{eqn3} and hence is the character of a unique
distinguished t-hexagon. Therefore $K=\Phi(K')$ for some distinguished t-hexagon $K'$ and the character
of $K'$ has a smaller third coordinate. In particular, $K$ is obtained, using a sequence of  $\Phi$'s and $\Psi$'s, 
from some distinguished t-hexagon $K'$ the character of which has zero third and fifth coordinates.

Let $K$ be a distinguished t-hexagon of perimeter $2n$ with character $(a_1,a_2,0,a_4,0,a_6)$. Then  $a_4\geq a_1$ by 
Lemma~\ref{lemn17}. If $a_4=a_1$, then from \eqref{eqn3} it follows that $a_2=a_6=0$ and
$K=Q$. If $a_4>a_1$, then from \eqref{eqn3} it follows that $a_2=a_6=a_4-a_1>0$. 
If $a_4>a_1+1$, then, using Lemma~\ref{lemn17}, it is easy to check that the vector 
\begin{displaymath}
(b_1,b_2,b_3,b_4,b_5,b_6):=(a_1+1,a_2-1,a_3+1,a_4-1,a_5+1,a_6-1)  
\end{displaymath}
satisfies all conditions in \eqref{eqn1} and \eqref{eqn3} and hence is the character of a unique
distinguished t-hexagon, say $K'$. Note that, by construction, $\mathrm{def}(K')<\mathrm{def}(K)$
and that $K=\Theta(K')$. Finally, if $a_4=a_1+1$, then from \eqref{eqn3}
it follows that the character of $K$ is of the form $(x,1,0,x+1,0,1)$, for some $x$.
In this case the perimeter of $K$ is $2x+3$, which is an odd number, contradicting our assumptions.
Therefore the case $a_4=a_1+1$ cannot occur.

Using induction on defect and the above steps it follows that any distinguished t-hexagon
of perimeter $2n$ is obtained using $\Phi$, $\Psi$ and $\Theta$ from a distinguished 
t-hexagon of perimeter $2n$ with character $(a_1,0,0,a_4,0,0)$. But from 
Equation~\ref{eqn3} it thus follows that $a_1=a_4=n$ and hence  the latter t-hexagon must be
isomorphic to $Q$. 

As a consequence of the above argument, we have that the image of the signature map is contained in 
$I(n\mathbf{e}(1))$. So, it remains to show that the signature map is injective.

Let $K$ be a distinguished t-hexagon with signature $(x,y,z)$ and of perimeter $2n$. Then the
character of $K$ equals $(x+y+z,a_2,y+z,a_4,z,a_6)$ for some $a_2,a_4,a_6\in\mathbb{Z}_{\geq 0}$. 
From \eqref{eqn3}, we have
\begin{displaymath}
x+y+z-a_4=z-a_2=y+z-a_6. 
\end{displaymath}
Since the perimeter of $K$ is $2n$, we also have
\begin{displaymath}
x+y+z+a_2+y+z+a_4+z+a_6=2n 
\end{displaymath}
and hence $a_2$,  $a_4$ and $a_6$ are uniquely determined. This means that the character of $K$
is uniquely determined and thus $K$ is uniquely determined up to isomorphism by its signature. 
This completes the proof.
\end{proof}

As an immediate corollary from Theorem~\ref{thmn72} we have:

\begin{corollary}\label{corn73}
For all $n\in\mathbb{Z}_{\geq 0}$, $C_{n}^{(3)}$ is the number of
isomorphism  classes of $t$-hexagons with perimeter $2n$. 
\end{corollary}

Our proof of Theorem~\ref{thmn72} provides another connection to 
the sequence $A001399(n)$ giving the number of partitions of $n$ in at most
three parts which was already mentioned in Subsection~\ref{s3.4}.
Let $P_n$ denote the set of all partitions of $n$ in at most three parts.
If $n<0$, we set $P_n=\varnothing$.

\begin{corollary}\label{corn74}
Let $n\in\mathbb{Z}_{\geq 0}$. Mapping $H$ with $\chi(H)=(a_1,a_2,a_3,a_4,a_5,a_6)$ 
to $(a_1,a_3,a_5)$ induces a bijection between the set of isomorphism classes of distinguished 
t-hexagons of perimeter $2n$ and the set $P_{n}\cup P_{n-3}\cup P_{n-6}\cup\dots$.
In particular, we have
\begin{displaymath}
C_n^{(3)}=A001399(n)+A001399(n-3)+A001399(n-6)+\dots.
\end{displaymath}
\end{corollary}

\begin{proof}
Restricting the bijection constructed in the proof of Theorem~\ref{thmn72} 
to the set of  distinguished t-hexagons of defect $2i$ and thereafter mapping
$\mathrm{sign}(H)=(x,y,z)$ to the partition $(x+y+z,y+z,z)$ of $x+2y+3z$, 
provides a bijection from the set of distinguished t-hexagons of defect $2i$ to $P_{n-i}$.
\end{proof}

\section{Partitions modulo $d$}\label{s4}

\subsection{Partitions and refinement}\label{s4.1}

For $n\in\mathbb{Z}_{\geq 0}$ denote by $\Pi_n$ the set of all partitions of
$n$, that is the set of all tuples $\boldsymbol{\lambda}=(\lambda_1,\lambda_2,\dots,\lambda_k)$
such that $\lambda_1,\lambda_2,\dots,\lambda_k\in\mathbb{N}$,
$n=\lambda_1+\lambda_2+\dots+\lambda_k$ and $\lambda_1\geq\lambda_2\geq\dots\geq\lambda_k$.
As usual, we write $\boldsymbol{\lambda}\vdash n$ for $\boldsymbol{\lambda}\in \Pi_n$.

For $\boldsymbol{\lambda}=(\lambda_1,\lambda_2,\dots,\lambda_k)\vdash n$ and
$\boldsymbol{\mu}=(\mu_1,\mu_2,\dots,\mu_l)\vdash n$ we say that 
$\boldsymbol{\lambda}$ {\em refines } $\boldsymbol{\mu}$ and write 
$\boldsymbol{\mu}< \boldsymbol{\lambda}$ provided that $l<k$ and there is a partition
$J_1\cup J_2\cup \dots\cup J_l$ of $\{1,2,\dots,k\}$ into a disjoint union of
non-empty subsets such that 
\begin{displaymath}
\mu_i=\sum_{j\in J_i}\lambda_j,\quad\text{ for all }i=1,2,\dots,l.
\end{displaymath}
The partially ordered set $(\Pi_n,<)$ was studied in \cite{Bi,Bj,Zi}. In particular,
in \cite{Zi} it was shown that it has some nasty properties. We refer the reader to
\cite{Zi} for more details on this poset.

The poset $(\Pi_n,<)$ is graded with respect to the rank function 
$(\lambda_1,\lambda_2,\dots,\lambda_k)\mapsto k$.

\subsection{Partitions modulo $d$}\label{s4.2}

For $d\in \mathbb{N}$, define an equivalence relation $\sim_d$ on $\Pi_n$ as follows:
Given $\boldsymbol{\lambda}=(\lambda_1,\lambda_2,\dots,\lambda_k)\vdash n$ and
$\boldsymbol{\mu}=(\mu_1,\mu_2,\dots,\mu_l)\vdash n$ set 
$\boldsymbol{\lambda}\sim_d\boldsymbol{\mu}$ provided that $k=l$ and there is
$\pi\in S_k$ such that $d$ divides $\lambda_i-\mu_{\pi(i)}$, for all $i$. In other words,
$\boldsymbol{\lambda}\sim_d\boldsymbol{\mu}$ if and only if the multisets of residues modulo
$d$ for parts of  $\boldsymbol{\lambda}$ and $\boldsymbol{\mu}$ coincide. For
$\boldsymbol{\lambda}\vdash n$ we denote the $\sim_d$-class of $\boldsymbol{\lambda}$
by $\overline{\boldsymbol{\lambda}}^{(d)}$.

Since $\sim_d$-equivalent partitions have the same number of parts
and partitions with the same number of parts are incomparable with respect to the refinement order $<$, 
this order induces a partial order $<_d$ on the set $\Pi_{n,d}:=\Pi_{n}/\sim_d$ defined
as the transitive closure  of the relation $\tilde{<}$ given by
$\overline{\boldsymbol{\lambda}}\tilde{<}\overline{\boldsymbol{\mu}}$ if there are
$\boldsymbol{\lambda}'\in \overline{\boldsymbol{\lambda}}$ and
$\boldsymbol{\mu}'\in \overline{\boldsymbol{\mu}}$ such that $\boldsymbol{\lambda}'<\boldsymbol{\mu}'$.
The poset $\Pi_{n,d}$ inherits from $\Pi_{n}$ the structure of a graded poset.

Define the poset $\Pi^*_{n,d}$ as follows: if $d$ does not divide $n$, set 
$\Pi^*_{n,d}:=\Pi_{n,d}$ with the order  $<_d$; if $d$ divides $n$, define 
$\Pi^*_{n,d}$ as the poset obtained from $(\Pi_{n,d},<_d)$ by adding a minimum
element, denoted $\varnothing$ (for simplicity, we will keep the notation $<_d$ for 
the partial order on $\Pi^*_{n,d}$). The structure of a graded poset on  $\Pi_{n}$
induces the structure of a graded poset on $\Pi^*_{n,d}$ by defining
the degree of 
$\varnothing$ to be zero. The class $\overline{(1,1,\dots,1)}^{(d)}=\{(1,1,\dots,1)\}$ of the partition
$(1,1,\dots,1)$ is the maximum element in $\Pi^*_{n,d}$.

\subsection{$\Pi^*_{n,d}$ versus $\cP_{d}$}\label{s4.3}

Our main result in this section is the following:

\begin{theorem}\label{thm21}
The (graded) posets $(\Pi^*_{n,d},<_d)$ and $(I(n\mathbf{e}(1)),\prec)$ are isomorphic.
\end{theorem}

\begin{proof}
To each $\boldsymbol{\lambda}\vdash n$ we associate the vector 
$(v_1^{\boldsymbol{\lambda}},v_2^{\boldsymbol{\lambda}},\dots,
v_d^{\boldsymbol{\lambda}})$, where, for $i=1,2,\dots,d$, we have
\begin{gather*} 
v_i^{\boldsymbol{\lambda}}:=|\{j\,:\,\lambda_j\equiv i\,\,\mathrm{mod}\, d\}|.
\end{gather*}
This map is constant on the $\sim_d$-equivalence classes and hence induces a
map from $\Pi_{n,d}$ to $\cP_{d}$. We extend this map to $\Pi^*_{n,d}$ by sending the
$\varnothing$ element to the zero vector in case $d$ divides $n$.
Denote the resulting map by $\Phi$. Note that $\Phi$ preserves the degree of
an element, namely, it maps a partition with $k$ parts to a vector of height $k$.

First of all, we claim that $\Phi$ is a homomorphism of posets. Indeed, any refinement
of partitions can be written as a composition of {\em elementary} refinements which 
simply refine one part of a smaller partition into two parts of a bigger partition.
Such elementary refinement corresponds to the covering relation 
$\boldsymbol{\mu}\lessdot\boldsymbol{\lambda}$ where $\boldsymbol{\lambda}$ has $k$
parts while $\boldsymbol{\mu}$ has $k-1$ parts. Assume that this refines the part
$\mu_i$ into parts $\lambda_s$ and $\lambda_t$. This means that 
$\mu_i=\lambda_s+\lambda_t$ and hence 
\begin{displaymath}
\mu_i\equiv \lambda_s+\lambda_t \,\,\mathrm{mod}\,\, d. 
\end{displaymath}
Let $a,b,c\in\{1,2,\dots,d\}$ be such that 
$\mu_i\equiv a$, $\lambda_s\equiv b$ and $\lambda_t\equiv c\,\,\mathrm{mod}\,\, d$. 
Then the element $\mathbf{e}(a)-\mathbf{e}(b)-\mathbf{e}(c)$
belong to $X_d$. This implies that $\Phi(\overline{\boldsymbol{\mu}})\prec \Phi(\overline{\boldsymbol{\lambda}})$.
It follows that $\Phi$ is a homomorphism of posets.

Clearly, $\Phi(\overline{(1,1,\dots,1)}^{(d)})=n\mathbf{e}(1)$. Since
$\overline{(1,1,\dots,1)}^{(d)}$ is the maximum element in $\Pi^*_{n,d}$, it follows that
$\Phi$ maps $\Pi^*_{n,d}$ to $I(n\mathbf{e}(1))$.

That $\Phi:\Pi^*_{n,d}\to I(n\mathbf{e}(1))$ is injective follows directly from the 
definition. It remains to show that $\Phi$ is surjective, in particular, invertible,
and that $\Phi^{-1}$ is order preserving. We prove this by downward
induction on the height $h$. If $h=d$, the claim is clear as 
$n\mathbf{e}(1)$ is the only element of $I(n\mathbf{e}(1))$ of height $h$.

For the induction step  $h\to h-1$ let $\mathbf{v}$ and $\mathbf{w}$ be two elements in
$I(n\mathbf{e}(1))$ of heights $h-1$ and $h$, respectively, and assume 
$\mathbf{v}\prec\mathbf{w}$. Then $\mathbf{v}=\mathbf{w}+\mathbf{x}$, for some
$\mathbf{x}\in X_d$. Let $\mathbf{x}=\mathbf{e}(k)-\mathbf{e}(i)-\mathbf{e}(j)$,
for some $i,j,k\in\{1,2,\dots,d\}$. From the inductive assumption, there is 
$\boldsymbol{\lambda}\vdash n$ such that $\Phi(\overline{\boldsymbol{\lambda}})=\mathbf{w}$.
Let $\lambda_s$ and $\lambda_t$ be two different parts of $\lambda$ with residues
$i$ and $j$ modulo $d$, respectively. Define $\mu$ as the partition obtained from
$\lambda$ by uniting $\lambda_s$ and $\lambda_t$. Then $\Phi(\overline{\boldsymbol{\mu}})=\mathbf{v}$.
Therefore $\Phi$ is a bijection.

From the arguments above it follows that the covering relations in 
$\Pi^*_{n,d}$ to $I(n\mathbf{e}(1))$ match precisely under $\Phi$. This implies
that $\Phi$ is an isomorphism of posets, completing the proof of the theorem.
\end{proof}

As an immediate corollary, we have:

\begin{corollary}\label{cor22}
For $n\in\mathbb{Z}_{\geq 0}$ and $d\in \mathbb{N}$,
we have $|\Pi^*_{n,d}|=C_{n}^{(d)}$.
\end{corollary}

\section{Connection to $d$-tonal partition monoid}\label{s5}

\subsection{Partition monoids}\label{s5.1}

For $n\in\mathbb{Z}_{\geq 0}$ consider the sets $\underline{n}=\{1,2,\dots,n\}$ 
and $\underline{n}'=\{1',2',\dots,n'\}$ (these two sets are automatically disjoint).
Set $\mathbf{n}:=\underline{n}\bigcup\underline{n}'$ and consider the set 
$\mathcal{P}(\mathbf{n})$ of all partitions of $\mathbf{n}$ into a disjoint union of
non-empty subsets. The cardinality of $\mathcal{P}(\mathbf{n})$ is the $2n$-th
Bell number, see $A000110$ in \cite{OEIS}.

The set $\mathcal{P}(\mathbf{n})$ has the natural structure of a monoid, see \cite{Jo,Mar,Maz1,Maz}.
The composition $\sigma\circ\pi$ of two partitions $\sigma,\pi\in \mathcal{P}(\mathbf{n})$
is defined as follows (here $\underline{n}''=\{1'',2'',\dots,n''\}$ is disjoint from
$\mathbf{n}$):
\begin{itemize}
\item First consider the partition $\sigma'$ of $\underline{n}'\cup\underline{n}''$ which is 
induced from $\sigma$ via the bijection $\underline{n}\cup\underline{n}'\to \underline{n}'\cup\underline{n}''$ 
which sends $i\mapsto i'$ for $i\in \underline{n}$ and $j'\to j''$
for $j'\in \underline{n}'$.
\item Let $\tilde{\pi}$ be the equivalence relation on $\underline{n}\cup \underline{n}'\cup\underline{n}''$
whose parts are those of $\pi$ combined with singletons of $\underline{n}''$.
\item Let $\tilde{\sigma}$ be the equivalence relation on $\underline{n}\cup \underline{n}'\cup\underline{n}''$
whose parts are those of $\sigma'$ combined with singletons of $\underline{n}$.
\item Let $\tilde{\tau}$ denote the minimal (with respect to inclusions) equivalence relation on the set
$\underline{n}\cup \underline{n}'\cup\underline{n}''$ which contains both $\tilde{\pi}$ and $\tilde{\sigma}$.
\item Let $\tilde{\tau}'$ be the restriction of $\tilde{\tau}$ to $\underline{n}\cup\underline{n}''$.
\item Define $\tau=\sigma\circ\pi$ as the partition of $\underline{n}\cup\underline{n}'$ induced from the partition $\tilde{\tau}'$ by 
the bijection $\underline{n}\cup\underline{n}''\to \underline{n}\cup\underline{n}'$
which sends $i\mapsto i$ for $i\in \underline{n}$ and $j''\to j'$
for $j''\in \underline{n}''$.
\end{itemize}
The identity element in the monoid  $(\mathcal{P}(\mathbf{n}),\circ)$ is the
{\em identity} partition 
\begin{displaymath}
\{\{1,1'\},\{2,2'\},\dots,\{k,k'\}\}\in \mathcal{P}(\mathbf{k}). 
\end{displaymath}
Both elements of $\mathcal{P}(\mathbf{n})$ and the composition $\circ$ admit a
diagrammatic description as shown in Figure~\ref{fig4}
(in the composition $\sigma\circ\pi$ which is the left hand side of 
the equality in Figure~\ref{fig4}, the element $\sigma$ is depicted on the left and the
element $\pi$ on the right). We refer the reader
to  \cite{Jo,Mar,Maz} for further details.

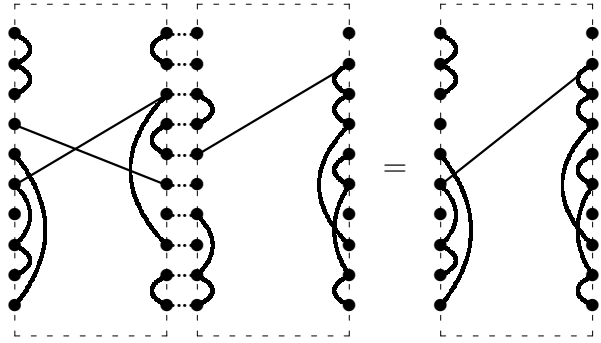
\begin{figure}
\special{em:linewidth 0.4pt} \unitlength 0.80mm
\begin{picture}(105.00,55.00)
\put(67.50,27.50){\makebox(0,0)[cc]{$=$}}
%%%%%%%%%%%%%%%%%%%%%%%%%%%%%%%%%%%%%%%%%%%%%%%%%%%%%%%%
\put(05.00,50.00){\makebox(0,0)[cc]{$\bullet$}}
\put(05.00,45.00){\makebox(0,0)[cc]{$\bullet$}}
\put(05.00,40.00){\makebox(0,0)[cc]{$\bullet$}}
\put(05.00,35.00){\makebox(0,0)[cc]{$\bullet$}}
\put(05.00,30.00){\makebox(0,0)[cc]{$\bullet$}}
\put(05.00,25.00){\makebox(0,0)[cc]{$\bullet$}}
\put(05.00,20.00){\makebox(0,0)[cc]{$\bullet$}}
\put(05.00,15.00){\makebox(0,0)[cc]{$\bullet$}}
\put(05.00,10.00){\makebox(0,0)[cc]{$\bullet$}}
\put(05.00,05.00){\makebox(0,0)[cc]{$\bullet$}}
\put(30.00,50.00){\makebox(0,0)[cc]{$\bullet$}}
\put(30.00,45.00){\makebox(0,0)[cc]{$\bullet$}}
\put(30.00,40.00){\makebox(0,0)[cc]{$\bullet$}}
\put(30.00,35.00){\makebox(0,0)[cc]{$\bullet$}}
\put(30.00,30.00){\makebox(0,0)[cc]{$\bullet$}}
\put(30.00,25.00){\makebox(0,0)[cc]{$\bullet$}}
\put(30.00,20.00){\makebox(0,0)[cc]{$\bullet$}}
\put(30.00,15.00){\makebox(0,0)[cc]{$\bullet$}}
\put(30.00,10.00){\makebox(0,0)[cc]{$\bullet$}}
\put(30.00,05.00){\makebox(0,0)[cc]{$\bullet$}}
\put(35.00,50.00){\makebox(0,0)[cc]{$\bullet$}}
\put(35.00,45.00){\makebox(0,0)[cc]{$\bullet$}}
\put(35.00,40.00){\makebox(0,0)[cc]{$\bullet$}}
\put(35.00,35.00){\makebox(0,0)[cc]{$\bullet$}}
\put(35.00,30.00){\makebox(0,0)[cc]{$\bullet$}}
\put(35.00,25.00){\makebox(0,0)[cc]{$\bullet$}}
\put(35.00,20.00){\makebox(0,0)[cc]{$\bullet$}}
\put(35.00,15.00){\makebox(0,0)[cc]{$\bullet$}}
\put(35.00,10.00){\makebox(0,0)[cc]{$\bullet$}}
\put(35.00,05.00){\makebox(0,0)[cc]{$\bullet$}}
\put(60.00,50.00){\makebox(0,0)[cc]{$\bullet$}}
\put(60.00,45.00){\makebox(0,0)[cc]{$\bullet$}}
\put(60.00,40.00){\makebox(0,0)[cc]{$\bullet$}}
\put(60.00,35.00){\makebox(0,0)[cc]{$\bullet$}}
\put(60.00,30.00){\makebox(0,0)[cc]{$\bullet$}}
\put(60.00,25.00){\makebox(0,0)[cc]{$\bullet$}}
\put(60.00,20.00){\makebox(0,0)[cc]{$\bullet$}}
\put(60.00,15.00){\makebox(0,0)[cc]{$\bullet$}}
\put(60.00,10.00){\makebox(0,0)[cc]{$\bullet$}}
\put(60.00,05.00){\makebox(0,0)[cc]{$\bullet$}}
\put(100.00,50.00){\makebox(0,0)[cc]{$\bullet$}}
\put(100.00,45.00){\makebox(0,0)[cc]{$\bullet$}}
\put(100.00,40.00){\makebox(0,0)[cc]{$\bullet$}}
\put(100.00,35.00){\makebox(0,0)[cc]{$\bullet$}}
\put(100.00,30.00){\makebox(0,0)[cc]{$\bullet$}}
\put(100.00,25.00){\makebox(0,0)[cc]{$\bullet$}}
\put(100.00,20.00){\makebox(0,0)[cc]{$\bullet$}}
\put(100.00,15.00){\makebox(0,0)[cc]{$\bullet$}}
\put(100.00,10.00){\makebox(0,0)[cc]{$\bullet$}}
\put(100.00,05.00){\makebox(0,0)[cc]{$\bullet$}}
\put(75.00,50.00){\makebox(0,0)[cc]{$\bullet$}}
\put(75.00,45.00){\makebox(0,0)[cc]{$\bullet$}}
\put(75.00,40.00){\makebox(0,0)[cc]{$\bullet$}}
\put(75.00,35.00){\makebox(0,0)[cc]{$\bullet$}}
\put(75.00,30.00){\makebox(0,0)[cc]{$\bullet$}}
\put(75.00,25.00){\makebox(0,0)[cc]{$\bullet$}}
\put(75.00,20.00){\makebox(0,0)[cc]{$\bullet$}}
\put(75.00,15.00){\makebox(0,0)[cc]{$\bullet$}}
\put(75.00,10.00){\makebox(0,0)[cc]{$\bullet$}}
\put(75.00,05.00){\makebox(0,0)[cc]{$\bullet$}}
%%%%%%%%%%%%%%%%%%%%%%%%%%%%%%%%%%%%%%%%%%%%%%%%%%
\dashline(05.00,55.00)(30.00,55.00){1}
\dashline(35.00,55.00)(60.00,55.00){1}
\dashline(75.00,55.00)(100.00,55.00){1}
\dashline(05.00,00.00)(30.00,00.00){1}
\dashline(35.00,00.00)(60.00,00.00){1}
\dashline(75.00,00.00)(100.00,00.00){1}
\dashline(75.00,00.00)(75.00,55.00){1}
\dashline(100.00,00.00)(100.00,55.00){1}
\dashline(05.00,00.00)(05.00,55.00){1}
\dashline(35.00,00.00)(35.00,55.00){1}
\dashline(30.00,00.00)(30.00,55.00){1}
\dashline(60.00,00.00)(60.00,55.00){1}
%%%%%%%%%%%%%%%%%%%%%%%%%%%%%%%%%%%%%%%%%%%%%%%%%%
\dottedline(30.00,50.00)(35.00,50.00){1}
\dottedline(30.00,45.00)(35.00,45.00){1}
\dottedline(30.00,40.00)(35.00,40.00){1}
\dottedline(30.00,35.00)(35.00,35.00){1}
\dottedline(30.00,30.00)(35.00,30.00){1}
\dottedline(30.00,25.00)(35.00,25.00){1}
\dottedline(30.00,20.00)(35.00,20.00){1}
\dottedline(30.00,15.00)(35.00,15.00){1}
\dottedline(30.00,10.00)(35.00,10.00){1}
\dottedline(30.00,05.00)(35.00,05.00){1}
%%%%%%%%%%%%%%%%%%%%%%%%%%%%%%%%%%%%%%%%%%%%%%%%%%
\linethickness{0.4mm}
\drawline(05.00,35.00)(30.00,25.00)
\drawline(35.00,30.00)(60.00,45.00)
\drawline(30.00,40.00)(05.00,25.00)
\drawline(75.00,25.00)(100.00,45.00)
%%%%%%%%%%%%%%%%%%%%%%%%%%%%%%%%%%%%%%%%%%%%%%%%%%
\linethickness{1pt}
\qbezier(05,25)(10,20)(05,15)
\qbezier(05,10)(10,12.50)(05,15)
\qbezier(05,05)(15,17.50)(05,30)
\qbezier(75,10)(80,12.50)(75,15)
\qbezier(75,25)(80,20)(75,15)
\qbezier(75,05)(85,17.50)(75,30)
\qbezier(30,50)(25,47.50)(30,45)
\qbezier(05,50)(10,47.50)(05,45)
\qbezier(05,40)(10,42.50)(05,45)
\qbezier(75,50)(80,47.50)(75,45)
\qbezier(75,40)(80,42.50)(75,45)
\qbezier(30,05)(25,07.50)(30,10)
\qbezier(35,05)(40,07.50)(35,10)
\qbezier(35,20)(40,15.00)(35,10)
\qbezier(30,15)(18,27.50)(30,40)
\qbezier(30,35)(25,32.50)(30,30)
\qbezier(35,35)(40,37.50)(35,40)
\qbezier(60,45)(55,42.50)(60,40)
\qbezier(60,35)(55,37.50)(60,40)
\qbezier(60,35)(50,25)(60,15)
\qbezier(100,35)(90,25)(100,15)
\qbezier(100,45)(95,42.50)(100,40)
\qbezier(100,35)(95,37.50)(100,40)
\qbezier(60,05)(55,7.5)(60,10)
\qbezier(100,05)(95,7.5)(100,10)
\qbezier(60,25)(55,27.5)(60,30)
\qbezier(100,25)(95,27.5)(100,30)
\qbezier(60,25)(55,17.5)(60,10)
\qbezier(100,25)(95,17.5)(100,10)
\end{picture}
\caption{Partitions and their composition}
\label{fig4}
\end{figure}

\subsection{$d$-tonal partition monoids}\label{s5.2}

For $d\in\mathbb{N}$, the {\em $d$-tonal partition monoid}
$\mathcal{P}_d(\mathbf{n})$, as introduced in \cite{Ta}, is the submonoid of
$\mathcal{P}(\mathbf{n})$ which consists of all partitions $\sigma$ of $\mathbf{n}$
such that every part $\sigma_i$ of $\sigma$ satisfies the condition that
\begin{displaymath}
d\quad\text{ divides }\quad |\sigma_i\cap\underline{n}|-|\sigma_i\cap\underline{n}'|. 
\end{displaymath}
Thus, for $d=1$ we have $\mathcal{P}_1(\mathbf{n})=\mathcal{P}(\mathbf{n})$.
For $d=2$ the above condition is equivalent to the requirement that all parts of 
$\sigma$ have even cardinality. Therefore $|\mathcal{P}_2(\mathbf{n})|$ is given
by the sequence $A005046$ in \cite{OEIS} (see also \cite{Or2}).

The twisted monoid algebra of the $d$-tonal partition monoid was studied
(under various names) in \cite{Ta,Or}, see also \cite{Ko1,Ko2,Ko3} for related algebras
and \cite[Section~5.1]{Ha} for some recent developments. We record the following open problem:

\begin{problem}\label{problem23}
Compute $|\mathcal{P}_d(\mathbf{n})|$ in a closed form as a function of $d$ and $n$. 
\end{problem}

As the twisted semigroup algebra of $\mathcal{P}_d(\mathbf{n})$ is generically semi-simple,
see \cite{Ta}, and forms, for all $n$, a sequence of embedded algebras with multiplicity-free
restrictions, see \cite{Ko1}, there is a natural analogue of the Robinson-Schensted correspondence for
$\mathcal{P}_d(\mathbf{n})$ and hence Problem~\ref{problem23} admits a combinatorial
reformulation in terms of walks on a certain Bratelli diagram.

\subsection{Rank and $d$-signature}\label{s5.3}

For $\sigma\in \mathcal{P}_d(\mathbf{n})$ the {\em rank} $\mathrm{rank}(\sigma)$
is the number of parts $\sigma_i$ in $\sigma$ such that both
$|\sigma_i\cap\underline{n}|\neq 0$ and
$|\sigma_i\cap\underline{n}'|\neq 0$. Such parts are called {\em propagating}.

Note that, for $\sigma\in \mathcal{P}_d(\mathbf{n})$, the cardinality of any part
of $\sigma$ which is entirely contained in $\underline{n}$ or in $\underline{n}'$
is divisible by $d$.

Define the function $\Psi:\mathcal{P}_d(\mathbf{n})\to\mathbb{Z}_{\geq 0}^d$, called
the {\em $d$-signature} function, as follows: for $\sigma\in \mathcal{P}_d(\mathbf{n})$
define $\Psi(\sigma)=(v_1,v_2,\dots,v_d)$, where for $i=1,2,\dots,d$ the number $v_i$
is the number of parts $\sigma_j$ in $\sigma$ satisfying the conditions
\begin{displaymath}
|\sigma_j\cap\underline{n}|\neq 0,\quad
|\sigma_j\cap\underline{n}'|\neq 0,\quad
d\,\,\text{ divides }\,\, |\sigma_j\cap\underline{n}|-i.
\end{displaymath}
Note that $v_1+v_2+\dots+v_d=\mathrm{rank}(\sigma)$.

\subsection{$\mathcal{J}$-classes of $d$-tonal partition monoids}\label{s5.4}

Two elements $\sigma,\pi\in \mathcal{P}_d(\mathbf{n})$ are called {\em $\mathcal{J}$-equivalent},
written $\sigma\mathcal{J}\pi$, provided that $\mathcal{P}_d(\mathbf{n})\sigma
\mathcal{P}_d(\mathbf{n})=\mathcal{P}_d(\mathbf{n})\pi\mathcal{P}_d(\mathbf{n})$,
see \cite[Section~4.4]{GM}. For $\sigma\in \mathcal{P}_d(\mathbf{n})$ we denote by
$\overline{\sigma}^{\mathcal{J}}$ the $\mathcal{J}$-equivalence class containing
$\sigma$.

There is a natural partial order on the set $\mathcal{P}_d(\mathbf{n})/\mathcal{J}$
given by inclusions: we write $\overline{\sigma}^{\mathcal{J}}
\rightsquigarrow \overline{\pi}^{\mathcal{J}}$
if and only if $\mathcal{P}_d(\mathbf{n})\sigma\mathcal{P}_d(\mathbf{n})
\subset \mathcal{P}_d(\mathbf{n})\pi\mathcal{P}_d(\mathbf{n})$.

\subsection{Canonical elements}\label{s5.5}

An element $\sigma\in \mathcal{P}_d(\mathbf{n})$ will be called {\em canonical} provided
that the following conditions are satisfied:
\begin{itemize}
\item Each part $\sigma_i$ of $\sigma$ satisfies
$|\sigma_i\cap\underline{n}|\leq d$ and $|\sigma_i\cap\underline{n}'|\leq d$.
\item The intersections $\sigma_i\cap\underline{n}$ and $\sigma_i\cap\underline{n}'$
are connected segments of $\underline{n}$ and $\underline{n}'$ respectively ordered by
cardinalities of the intersections for those parts $\sigma_i$ which intersects both 
$\underline{n}$ and $\underline{n}'$ and then followed by those parts of $\sigma$ which
intersect only $\underline{n}$ or $\underline{n}'$.
\end{itemize}
For example, the identity element in $\mathcal{P}_d(\mathbf{n})$ is canonical.

\begin{lemma}\label{lem31}
For  each $\sigma\in \mathcal{P}_d(\mathbf{n})$ there is a canonical 
$\pi\in \mathcal{P}_d(\mathbf{n})$ such that $\sigma\mathcal{J}\pi$.
\end{lemma}

\begin{proof}
If some part $\sigma_i$ of $\sigma$
satisfies $|\sigma_i\cap\underline{n}|>d$ or $|\sigma_i\cap\underline{n}'|>d$,
then there is $\sigma'\in \mathcal{P}_d(\mathbf{n})$ which has exactly the same parts as
$\sigma$ except for $\sigma_i$ which is split into two parts: a part with $d$ elements
which is a subset of  $\underline{n}$ (respectively $\underline{n}'$) 
and its complement. Existence of $\sigma'$ follows
using the construction shown in Figure~\ref{fig5} (in the case $d=3$)
and note that $\sigma\mathcal{J}\sigma'$ also follows from Figure~\ref{fig5}.
Proceeding inductively, we find an element $\tau\in  \mathcal{P}_d(\mathbf{n})$ which is
in the same $\mathcal{J}$-class as $\sigma$ and which satisfies the condition that 
$|\tau_i\cap\underline{n}|\leq d$ and $|\tau_i\cap\underline{n}'|\leq d$ for
each part $\tau_i$ of $\tau$. Permuting, if necessary, the elements of $\underline{n}$ 
and, independently, of $\underline{n}'$ (that is, multiplying from the left 
and/or from the right by permutations, noting that
all permutations belong to $\mathcal{P}_d(\mathbf{n})$), 
one rearranges $\tau$ into a canonical element 
$\pi$ in the same $\mathcal{J}$-class as $\sigma$. The claim follows.
\end{proof}

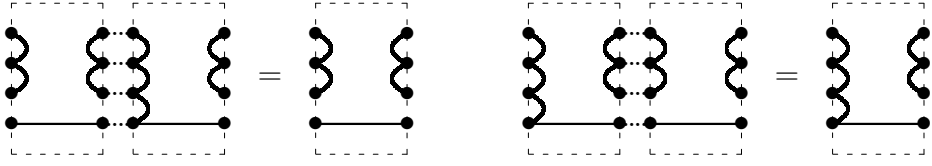
\begin{figure}
\special{em:linewidth 0.4pt} \unitlength 0.80mm
\begin{picture}(160.00,35.00)
\put(47.50,17.50){\makebox(0,0)[cc]{$=$}}
\put(132.50,17.50){\makebox(0,0)[cc]{$=$}}
%%%%%%%%%%%%%%%%%%%%%%%%%%%%%%%%%%%%%%%%%%%%%%%%%%%%%%%%
\put(05.00,10.00){\makebox(0,0)[cc]{$\bullet$}}
\put(05.00,15.00){\makebox(0,0)[cc]{$\bullet$}}
\put(05.00,20.00){\makebox(0,0)[cc]{$\bullet$}}
\put(05.00,25.00){\makebox(0,0)[cc]{$\bullet$}}
\put(20.00,10.00){\makebox(0,0)[cc]{$\bullet$}}
\put(20.00,15.00){\makebox(0,0)[cc]{$\bullet$}}
\put(20.00,20.00){\makebox(0,0)[cc]{$\bullet$}}
\put(20.00,25.00){\makebox(0,0)[cc]{$\bullet$}}
\put(25.00,10.00){\makebox(0,0)[cc]{$\bullet$}}
\put(25.00,15.00){\makebox(0,0)[cc]{$\bullet$}}
\put(25.00,20.00){\makebox(0,0)[cc]{$\bullet$}}
\put(25.00,25.00){\makebox(0,0)[cc]{$\bullet$}}
\put(40.00,10.00){\makebox(0,0)[cc]{$\bullet$}}
\put(40.00,15.00){\makebox(0,0)[cc]{$\bullet$}}
\put(40.00,20.00){\makebox(0,0)[cc]{$\bullet$}}
\put(40.00,25.00){\makebox(0,0)[cc]{$\bullet$}}
\put(55.00,10.00){\makebox(0,0)[cc]{$\bullet$}}
\put(55.00,15.00){\makebox(0,0)[cc]{$\bullet$}}
\put(55.00,20.00){\makebox(0,0)[cc]{$\bullet$}}
\put(55.00,25.00){\makebox(0,0)[cc]{$\bullet$}}
\put(70.00,10.00){\makebox(0,0)[cc]{$\bullet$}}
\put(70.00,15.00){\makebox(0,0)[cc]{$\bullet$}}
\put(70.00,20.00){\makebox(0,0)[cc]{$\bullet$}}
\put(70.00,25.00){\makebox(0,0)[cc]{$\bullet$}}
\put(105.00,10.00){\makebox(0,0)[cc]{$\bullet$}}
\put(105.00,15.00){\makebox(0,0)[cc]{$\bullet$}}
\put(105.00,20.00){\makebox(0,0)[cc]{$\bullet$}}
\put(105.00,25.00){\makebox(0,0)[cc]{$\bullet$}}
\put(90.00,10.00){\makebox(0,0)[cc]{$\bullet$}}
\put(90.00,15.00){\makebox(0,0)[cc]{$\bullet$}}
\put(90.00,20.00){\makebox(0,0)[cc]{$\bullet$}}
\put(90.00,25.00){\makebox(0,0)[cc]{$\bullet$}}
\put(125.00,10.00){\makebox(0,0)[cc]{$\bullet$}}
\put(125.00,15.00){\makebox(0,0)[cc]{$\bullet$}}
\put(125.00,20.00){\makebox(0,0)[cc]{$\bullet$}}
\put(125.00,25.00){\makebox(0,0)[cc]{$\bullet$}}
\put(110.00,10.00){\makebox(0,0)[cc]{$\bullet$}}
\put(110.00,15.00){\makebox(0,0)[cc]{$\bullet$}}
\put(110.00,20.00){\makebox(0,0)[cc]{$\bullet$}}
\put(110.00,25.00){\makebox(0,0)[cc]{$\bullet$}}
\put(155.00,10.00){\makebox(0,0)[cc]{$\bullet$}}
\put(155.00,15.00){\makebox(0,0)[cc]{$\bullet$}}
\put(155.00,20.00){\makebox(0,0)[cc]{$\bullet$}}
\put(155.00,25.00){\makebox(0,0)[cc]{$\bullet$}}
\put(140.00,10.00){\makebox(0,0)[cc]{$\bullet$}}
\put(140.00,15.00){\makebox(0,0)[cc]{$\bullet$}}
\put(140.00,20.00){\makebox(0,0)[cc]{$\bullet$}}
\put(140.00,25.00){\makebox(0,0)[cc]{$\bullet$}}
%%%%%%%%%%%%%%%%%%%%%%%%%%%%%%%%%%%%%%%%%%%%%%%%%%
\dashline(05.00,30.00)(20.00,30.00){1}
\dashline(05.00,30.00)(05.00,05.00){1}
\dashline(20.00,05.00)(05.00,05.00){1}
\dashline(20.00,05.00)(20.00,30.00){1}
\dashline(25.00,30.00)(40.00,30.00){1}
\dashline(25.00,30.00)(25.00,05.00){1}
\dashline(40.00,05.00)(25.00,05.00){1}
\dashline(40.00,05.00)(40.00,30.00){1}
\dashline(55.00,30.00)(70.00,30.00){1}
\dashline(55.00,30.00)(55.00,05.00){1}
\dashline(70.00,05.00)(55.00,05.00){1}
\dashline(70.00,05.00)(70.00,30.00){1}
\dashline(90.00,30.00)(105.00,30.00){1}
\dashline(90.00,30.00)(90.00,05.00){1}
\dashline(105.00,05.00)(90.00,05.00){1}
\dashline(105.00,05.00)(105.00,30.00){1}
\dashline(140.00,30.00)(155.00,30.00){1}
\dashline(140.00,30.00)(140.00,05.00){1}
\dashline(155.00,05.00)(140.00,05.00){1}
\dashline(155.00,05.00)(155.00,30.00){1}
\dashline(110.00,30.00)(125.00,30.00){1}
\dashline(110.00,30.00)(110.00,05.00){1}
\dashline(125.00,05.00)(110.00,05.00){1}
\dashline(125.00,05.00)(125.00,30.00){1}
%%%%%%%%%%%%%%%%%%%%%%%%%%%%%%%%%%%%%%%%%%%%%%%%%%
\dottedline(20.00,10.00)(25.00,10.00){1}
\dottedline(20.00,15.00)(25.00,15.00){1}
\dottedline(20.00,20.00)(25.00,20.00){1}
\dottedline(20.00,25.00)(25.00,25.00){1}
\dottedline(110.00,10.00)(105.00,10.00){1}
\dottedline(110.00,15.00)(105.00,15.00){1}
\dottedline(110.00,20.00)(105.00,20.00){1}
\dottedline(110.00,25.00)(105.00,25.00){1}
%%%%%%%%%%%%%%%%%%%%%%%%%%%%%%%%%%%%%%%%%%%%%%%%%%
\linethickness{1pt}
\drawline(05.00,10.00)(20.00,10.00)
\drawline(25.00,10.00)(40.00,10.00)
\drawline(55.00,10.00)(70.00,10.00)
\drawline(90.00,10.00)(105.00,10.00)
\drawline(140.00,10.00)(155.00,10.00)
\drawline(110.00,10.00)(125.00,10.00)
%%%%%%%%%%%%%%%%%%%%%%%%%%%%%%%%%%%%%%%%%%%%%%%%%%
\linethickness{1pt}
\qbezier(05,25)(10,22.50)(05,20)
\qbezier(05,15)(10,17.50)(05,20)
\qbezier(20,15)(15,17.50)(20,20)
\qbezier(20,25)(15,22.50)(20,20)
\qbezier(25,25)(30,22.50)(25,20)
\qbezier(25,15)(30,17.50)(25,20)
\qbezier(25,15)(30,12.50)(25,10)
\qbezier(40,15)(35,17.50)(40,20)
\qbezier(40,25)(35,22.50)(40,20)
\qbezier(55,25)(60,22.50)(55,20)
\qbezier(55,15)(60,17.50)(55,20)
\qbezier(70,15)(65,17.50)(70,20)
\qbezier(70,25)(65,22.50)(70,20)
\qbezier(90,25)(95,22.50)(90,20)
\qbezier(90,15)(95,17.50)(90,20)
\qbezier(105,15)(100,17.50)(105,20)
\qbezier(105,25)(100,22.50)(105,20)
\qbezier(110,25)(115,22.50)(110,20)
\qbezier(110,15)(115,17.50)(110,20)
\qbezier(90,15)(95,12.50)(90,10)
\qbezier(140,15)(145,12.50)(140,10)
\qbezier(125,15)(120,17.50)(125,20)
\qbezier(125,25)(120,22.50)(125,20)
\qbezier(140,25)(145,22.50)(140,20)
\qbezier(140,15)(145,17.50)(140,20)
\qbezier(155,15)(150,17.50)(155,20)
\qbezier(155,25)(150,22.50)(155,20)
\end{picture}
\caption{Illustration of the proof of Lemma~\ref{lem31}}
\label{fig5}
\end{figure}

\begin{proposition}\label{prop32}
We have $\Psi(\mathcal{P}_d(\mathbf{n}))=I(n\mathbf{e}(1))$. 
\end{proposition}

\begin{proof}
We use downward induction to prove that, for each $k=n,n-1,n-2,\dots,0$, the map $\Psi$ 
induces a bijection between the set of all canonical elements of 
rank $k$ in $\mathcal{P}_d(\mathbf{n})$ and the set of all elements of height 
$k$ in $I(n\mathbf{e}(1))$. The statement of the corollary then will follow from
Lemma~\ref{lem31}.

The basis of the induction is $k=n$. In this case on the left hand side we have only
one canonical element, the identity element, while on the right hand side we have
$n\mathbf{e}(1)$ which is the image of the identity element under $\Psi$.

Let $\mathbf{v}\in I(n\mathbf{e}(1))$ be an element of height $k$ and let 
$\sigma$ be a canonical element such that $\Psi(\sigma)=\mathbf{v}$. Let
$\mathbf{e}(k)-\mathbf{e}(i)-\mathbf{e}(j)\in X_d$ be such that 
$\mathbf{v}+(\mathbf{e}(k)-\mathbf{e}(i)-\mathbf{e}(j))\in I(n\mathbf{e}(1))$. Then $\sigma$ has a part
$\sigma_s$ such that $d$ divides $|\sigma_s\cap\underline{n}|-i$ and a different part
$\sigma_t$ such that $d$ divides $|\sigma_t\cap\underline{n}|-j$. Consider the
element $\sigma'$ obtained from $\sigma$ by uniting $\sigma_s$ with $\sigma_t$
and keeping all other parts. Then $\sigma'\sigma=\sigma'$ and hence 
$\mathcal{P}_d(\mathbf{n})\sigma'\mathcal{P}_d(\mathbf{n})\subset \mathcal{P}_d(\mathbf{n})\sigma\mathcal{P}_d(\mathbf{n})$. Moreover, 
$\Psi(\sigma')=\mathbf{v}+\mathbf{e}(k)-\mathbf{e}(i)-\mathbf{e}(j)$.
This implies surjectivity of the induction step.

At the same time, the form of the canonical element immediately implies that it is 
obtained from the identity element using the unification procedure described in the 
previous paragraph, followed by splitting off $d$-element parts contained in 
$\underline{n}$ or $\underline{n}'$ (note that the latter parts do not affect the value of $\Psi$
by definition). This implies that $\Psi$ takes values inside
$I(n\mathbf{e}(1))$ and completes the proof.
\end{proof}

\begin{corollary}\label{cor33}
For  each $\sigma\in \mathcal{P}_d(\mathbf{n})$, there is a unique canonical 
$\pi\in \mathcal{P}_d(\mathbf{n})$ such that $\sigma\mathcal{J}\pi$.
\end{corollary}

\begin{proof}
Taking into account Proposition~\ref{prop32}, the claim follows from the observation that
different canonical elements are sent by $\Psi$ to different elements in 
$I(n\mathbf{e}(1))$. 
\end{proof}

\subsection{A combinatorial description of the $\mathcal{J}$-order}\label{s5.6}

Our second main result, which explains our interest in $C_n^{(d)}$, is the following:

\begin{theorem}\label{thm34}
The map $\Psi:(\mathcal{P}_d(\mathbf{n})/\mathcal{J},\rightsquigarrow)\to
(I(n\mathbf{e}(1)),\prec)$ is an isomorphism of posets.
\end{theorem}

\begin{proof}
From Proposition~\ref{prop32}, we have a map $\Psi:\mathcal{P}_d(\mathbf{n})/\mathcal{J}\to
I(n\mathbf{e}(1))$. This map is bijective by the combination of
Proposition~\ref{prop32} and Corollary~\ref{cor33}. From the third paragraph of the
proof of  Proposition~\ref{prop32} it follows that, for each pair of elements
$\mathbf{v},\mathbf{w}\in I(n\mathbf{e}(1))$ such that $\mathbf{v}\prec \mathbf{w}$,
there are $\sigma,\pi\in \mathcal{P}_d(\mathbf{n})$ such that 
$\mathcal{P}_d(\mathbf{n})\sigma\mathcal{P}_d(\mathbf{n})\subset \mathcal{P}_d(\mathbf{n})\pi\mathcal{P}_d(\mathbf{n})$,
$\Psi(\sigma)=\mathbf{v}$ and $\Psi(\pi)=\mathbf{w}$.

On the other hand, from the last paragraph of the
proof of  Proposition~\ref{prop32} it follows that the poset 
$(\mathcal{P}_d(\mathbf{n})/\mathcal{J},\rightsquigarrow)$ is a graded poset.
Now, applying the argument from the third paragraph of the
proof of  Proposition~\ref{prop32} once more and counting modulo $d$, one
checks that the covering relations in $(\mathcal{P}_d(\mathbf{n})/\mathcal{J},\rightsquigarrow)$
and $(I(n\mathbf{e}(1)),\prec)$ match precisely via $\Psi$. The claim follows.
\end{proof}

\section{Enumeration of $\mathcal{J}$-classes for arbitrary $d$}\label{s6}

\subsection{Enumeration via $d$-part partitions}\label{s6.01}

The proof of Proposition~\ref{prop32} gives a way to write a formula for $C_{n}^{(d)}$
in the general case. 
Let $d\in\mathbb{N}$ and $n\in\mathbb{Z}_{\geq 0}$. 
Denote by $P_n^{(d)}$ the number of partitions of $n$ with at most $d$ parts.
By taking the dual partition, we get the usual fact that 
$P_n^{(d)}$ also equals the number of partitions of $n$ in which
each part does not exceed $d$.
For simplicity, we set  $P_n^{(d)}=0$ when $n<0$.

\begin{theorem}\label{thmdpart}
We have $C_{n}^{(d)}=P_n^{(d)}+P_{n-d}^{(d)}+P_{n-2d}^{(d)}+P_{n-3d}^{(d)}+\dots$. 
\end{theorem}

\begin{proof}
To prove this claim we analyze the  proof of Proposition~\ref{prop32}.
According to the latter proof, $C_{n}^{(d)}$ enumerates canonical elements
in $\mathcal{P}_d(\mathbf{n})$. Let $\sigma$ be a canonical element. 
Let $\sigma_1,\sigma_2,\dots,\sigma_k$ be the list of all parts of $\sigma$
contained in $\underline{n}$ (note that $k$ might be zero). 
Then each of these parts has cardinality $d$ and we may consider the set 
\begin{displaymath}
\underline{n}^{\sigma}:= 
\underline{n}\setminus(\sigma_1\cup\sigma_2\cup\dots\cup\sigma_k) 
\end{displaymath}
which thus has cardinality $n-kd$.

Cardinalities of intersections of all propagating parts of $\sigma$ with 
$\underline{n}^{\sigma}$ determine a partition of $n-kd$ in which
each part does not exceed  $d$. It is straightforward that this gives a bijection
between the set of all canonical elements
in $\mathcal{P}_d(\mathbf{n})$ with $2k$ non-propagating parts and 
all partitions of $n-kd$ for which each part does not exceed $d$.
The claim follows.
\end{proof}

\begin{corollary}\label{cordpart}
For $d\geq 1$, we have
\begin{displaymath}
\sum_{n\geq 1}C_{n}^{(d)}t^n=\frac{1}{(1-t^d)\cdot(1-t)(1-t^2)(1-t^3)\dots (1-t^d)}. 
\end{displaymath}
\end{corollary}

\begin{proof}
This follows by combining the usual equality
\begin{displaymath}
\sum_{n\geq 1}P_{n}^{(d)}t^n=\frac{1}{(1-t)(1-t^2)(1-t^3)\dots (1-t^d)} 
\end{displaymath}
with the statement of Theorem~\ref{thmdpart}.
\end{proof}

\begin{remark}
{\rm  
It is easy to check that, for $d=3$, we indeed have the equality
\begin{displaymath}
\frac{1+t^2+t^3+t^5}{(1-t)(1-t^3)(1-t^4)(1-t^6)}=\frac{1}{(1-t)(1-t^2)(1-t^3)^2}. 
\end{displaymath}
Here the left hand side is the original generating function for $A028289$.
}
\end{remark}

\begin{remark}\label{remnewnn}
{\rm 
The poset $\Pi^{(d)}$ of partitions with at most $d$ parts can be defined using the same 
approach as we used to define $\Lambda_d$. The assertion of Theorem~\ref{thmdpart} can 
then be interpreted as a bijection between certain (co)ideals in $\Pi^{(d)}$ and $\Lambda_d$.
Such a bijection admits a direct combinatorial construction. 
} 
\end{remark}

\subsection{Examples for $d=4$ and $d=5$}\label{s6.1}

The sequence $C_n^{(4)}$ starts as follows:
\begin{displaymath}
1,1,2,3,6,7,11,14,21,25,\dots.
\end{displaymath}
The sequence $C_n^{(5)}$ starts as follows:
\begin{displaymath}
1,1,2,3,5,8,11,15,21,\dots.
\end{displaymath}
%%It seems that n
We note that 
none of the sequences $C_n^{(d)}$ for $d\geq 4$ appeared on
\cite{OEIS} before. However, as noted, they are simple cumulative sums 
of classical
sequences. 

\subsection{Relation to partition function}\label{s6.3}

Recall the classical {\em partition function} $P(n)$ which gives, 
for $n\in\mathbb{Z}_{\geq 0}$, the  number of partitions of $n$, 
see the sequence  $A000041$ in \cite{OEIS}. One general 
observation for the numbers $C_{n,h}^{(d)}$ is the following:

\begin{proposition}\label{prop777}
If $n-h<d$ and $2(n-h)<n$, then $C_{n,h}^{(d)}=P(n-h)$. 
\end{proposition}

\begin{proof}
To prove the assertion we construct a bijective map between $I(n\mathbf{e}(1))\cap\Lambda_d^{(h)}$ 
and the set of all partitions of $n-h$.

For $\mathbf{v}\in\Lambda_d$, set $\alpha(\mathbf{v})=v_2+2v_3+3v_4+\dots$. For $i,j\in\{1,2,\dots,d\}$ 
such that $i+j<d$, we have $(i+j-1)-(i-1)-(j-1)=1$. Therefore, for such values of $i$ and $j$ and for any 
$\mathbf{v},\mathbf{w}\in \Lambda_d$, we have $\alpha(\mathbf{v})=\alpha(\mathbf{w})+1$ provided that
\begin{displaymath}
\mathbf{v}=\mathbf{w}+\mathbf{e}(i+j)-\mathbf{e}(i)-\mathbf{e}(j).
\end{displaymath}
Note that $\mathbf{e}(i+j)-\mathbf{e}(i)-\mathbf{e}(j)\in X_d$. 

Assume now that $n-h<d$ and $\mathbf{v}\in I(n\mathbf{e}(1))$ is of height $h<n$.
Then $\mathbf{v}$ is obtained from $n\mathbf{e}(1)$ by adding $n-h$ vectors from $X_d$ of the form 
$\mathbf{e}(i+j)-\mathbf{e}(i)-\mathbf{e}(j)$, for some $i$ and $j$ as above. Indeed, let $s>1$ be the 
smallest index such that $v_i\neq 0$ (which exists as $h<n$ and $n-h<d$). 
Then the vector $\mathbf{v}$ is obtained from the vector
\begin{displaymath}
\mathbf{w}=(v_1+1,v_2,\dots,v_{s-2},v_{s-1}+1,v_{s}-1,v_{s+1},\dots,v_d)
\end{displaymath}
by adding $\mathbf{e}((s-1)+1)-\mathbf{e}(s-1)-\mathbf{e}(1)$
(here $w_1=v_1+2$ if $s=2$). Applying a similar procedure
to $\mathbf{w}$ and proceeding inductively, we get the claim. This implies that
\begin{displaymath}
\Upsilon(\mathbf{v}):=(v_2+v_3+v_4+\dots,v_3+v_4+\dots,\dots)
\end{displaymath}
is a partition of $n-h$. Since $n-h$ is fixed and $\mathbf{v}$ is, clearly, recoverable
from $\Upsilon(\mathbf{v})$, the map
$\Upsilon$ from $I(n\mathbf{e}(1))\cap\Lambda_d^{(h)}$ 
to  the set of all partitions of $n-h$ is injective.

To prove surjectivity of $\Upsilon$, assume that $n-h=x_2+2x_3+3x_4+\dots$, for some non-negative $x_2,x_3,\dots$.
We proceed by induction on $n-h$. If $n-h=0$, surjectivity of our map is obvious. To prove the induction step,
we write $k=i+j$, for some $1\leq i,j \leq k-1$, and consider the partition of $n-h-1$ given by decreasing
$x_k$ by $1$, increasing $x_i$ by $1$ and increasing $x_j$ by $1$ (if $i=j$, the outcome is that
$x_i$ is increased by $2$).  From the combination of the inductive assumption and the condition 
$2(n-h)<n$, it follows that the 
resulting partition of $n-h-1$ is in the image of our map. Applying the definition of $\prec$ it follows
that the original partition of  $n-h$ is also in the image of our map. This completes the proof.
\end{proof}

\begin{problem}\label{problem2}
Find a closed formula for $C_{n,h}^{(d)}$ for all $d,n,h$.
\end{problem}

\vspace{0.3cm}

\noindent
Ch.~A.: Department of Pure Mathematics, University of Leeds, Leeds, 
LS2 9JT, UK, e-mail: {\tt mmcaa\symbol{64}maths.leeds.ac.uk }\\
current address: Department of Mathematics, College of Science, 
University of Sulaimani,  Iraq
\vspace{0.3cm}

\noindent
P.~M.: Department of Pure Mathematics, University of Leeds, Leeds, 
LS2 9JT, UK, e-mail: {\tt ppmartin\symbol{64}maths.leeds.ac.uk }
\vspace{0.3cm}

\noindent
V.~M: Department of Mathematics, Uppsala University, Box. 480,
SE-75106, Uppsala, SWEDEN, email: {\tt mazor\symbol{64}math.uu.se}


\begin{thebibliography}{99999999}
\bibitem[Ko3]{Ko3} M.~Kosuda. Party algebra of type $B$ and construction of its 
irreducible representations. Manuscript available via google.
\end{thebibliography}
\end{document}